\documentclass[reqno]{amsart}

\usepackage{amsmath}
\usepackage{amsthm}
\usepackage{amsfonts}
\usepackage{amssymb}
\usepackage{enumerate}
\usepackage{graphicx}
\usepackage{color}

\newcommand{\indicator}[1]{\ensuremath{\mathbf{1}_{\{#1\}}}}
\newcommand{\oindicator}[1]{\ensuremath{\mathbf{1}_{{#1}}}}

\numberwithin{equation}{section}

\DeclareMathOperator{\var}{Var}
\DeclareMathOperator{\tr}{tr}

\DeclareMathOperator{\dist}{dist}

\newcommand{\Prob}{\mathbb{P}}
\newcommand{\E}{\mathbb{E}}
\newcommand{\C}{\mathbb{C}}
\renewcommand{\P}{\mathbb{P}}
\renewcommand\Re{\operatorname{Re}}
\renewcommand\Im{\operatorname{Im}}
\newcommand{\eps}{\varepsilon}

\newcommand{\mat}{\mathbf}

\theoremstyle{plain}
  \newtheorem{theorem}{Theorem}[section]

  \newtheorem{proposition}[theorem]{Proposition}
  
  \newtheorem{lemma}[theorem]{Lemma}
  \newtheorem{corollary}[theorem]{Corollary}

\theoremstyle{definition}
  \newtheorem{definition}[theorem]{Definition}
  
  \newtheorem{remark}[theorem]{Remark}

\begin{document}
\title[Central limit theorem for linear eigenvalue statistics]{Central limit theorem for linear eigenvalue statistics of elliptic random matrices}

\author[S. O'Rourke]{Sean O'Rourke}
\address{Department of Mathematics, University of Colorado at Boulder, Boulder, CO 80309  }
\thanks{S. O'Rourke has been supported by grant AFOSAR-FA-9550-12-1-0083}
\email{sean.d.orourke@colorado.edu}

\author[D. Renfrew]{David Renfrew} 
\address{Department of Mathematics, UCLA  }
\thanks{D. Renfrew has been supported by grant DMS-0838680}
\email{dtrenfrew@math.ucla.edu}

\begin{abstract}
We consider a class of elliptic random matrices which generalize two classical ensembles from random matrix theory: Wigner matrices and random matrices with iid entries.  In particular, we establish a central limit theorem for linear eigenvalue statistics of real elliptic random matrices under the assumption that the test functions are analytic.  As a corollary, we extend the results of Rider and Silverstein \cite{RS} to real iid random matrices.   
\end{abstract}

\maketitle

\section{Introduction}

Eigenvalues of large dimensional random matrices have been widely studied in recent years due, in part, to their relevance to statistics, computer science, and theoretical physics.  Two classical ensembles which have received considerable attention are Wigner matrices and iid random matrices. Elliptic random matrices (defined below in Section \ref{sec:ell}) were original introduced by Girko \cite{Gorig,Gten} as a natural generalization of both Wigner matrices and iid random matrices.

\subsection{Classical ensembles}
We begin with some definitions and examples.  

\begin{definition}[Wigner matrix]
Let $\xi, \zeta$ be real random variables.  We say $\mat{Y}_N$ is a \emph{real symmetric Wigner matrix} of size $N$ with atom variables $\xi,\zeta$ if $\mat{Y}_N = (y_{ij})_{i,j=1}^N$ is a random real symmetric $N \times N$ matrix that satisfies the following conditions.
\begin{itemize}
\item $\{y_{ij} : 1 \leq i \leq j \leq N\}$ is a collection of independent random variables.
\item $\{y_{ij} : 1 \leq i < j \leq N\}$ is a collection of independent and identically distributed (iid) copies of $\xi$.
\item $\{y_{ii} : 1 \leq i \leq N\}$ is a collection of iid copies of $\zeta$.  
\end{itemize}
\end{definition}

The prototypical example of a Wigner real symmetric matrix is the \emph{Gaussian orthogonal ensemble} (GOE).  The GOE is defined by the probability distribution  
\begin{equation*}
	\Prob(d \mat{M}) = \frac{1}{Z_N} \exp\left({-\frac{1}{4}\tr{\mat{M}^2}}\right) d \mat{M}
\end{equation*}
on the space of $N \times N$ real symmetric matrices, where $d \mat{M}$ refers to the Lebesgue measure on the $N(N+1)/2$ different elements of the matrix.  Here $Z_N$ denotes the normalization constant.  So for a matrix $\mat{Y}_N = (y_{ij})_{i,j=1}^N$ drawn from the GOE, the elements $\{ y_{ij} : 1 \leq i \leq j \leq N \}$ are independent Gaussian random variables with mean zero and variance $1+\delta_{ij}$. 

\begin{definition}[iid random matrix]
Let $\xi$ be a random variable.  We say $\mat{Y}_N$ is an \emph{iid random matrix} of size $N$ with atom variable $\xi$ if $\mat{Y}_N$ is a $N \times N$ matrix whose entries are iid copies of $\xi$.  
\end{definition}

In 1965, Ginibre \cite{Gi} introduced several classes of iid random matrices.  The \emph{complex (real) Ginibre ensemble} consists of $N \times N$ matrices whose entries are iid copies of a standard complex (real) Gaussian random variable.    

For both Wigner and iid random matrix ensembles, the most basic object of study is the limiting spectral distribution of the eigenvalues.  For any $N \times N$ matrix $\mat{M}$, we let $\lambda_1(\mat{M}), \ldots, \lambda_N(\mat{M})$ denote the eigenvalues of $\mat{M}$.  In this case, the \emph{empirical spectral measure} $\mu_{\mat{M}}$ is given by 
$$ \mu_{\mat{M}} := \frac{1}{N} \sum_{i=1}^N \delta_{\lambda_i(\mat{M})}. $$
In general, $\mu_{\mat{M}}$ is a probability measure supported on $\mathbb{C}$.  However, if the matrix $\mat{M}$ is Hermitian, then the eigenvalues $\lambda_1(\mat{M}), \ldots, \lambda_N(\mat{M})$ are real.  In this case, $\mu_{\mat{M}}$ is a probability measure on $\mathbb{R}$.  
%The corresponding \emph{empirical spectral distribution} (ESD) $F^{\mat{M}}$ is defined as
%$$ F^{\mat{M}}(x,y) := \frac{1}{N} \# \left\{ 1 \leq i \leq N : \Re(\lambda_i(\mat{M})) \leq x, \Im(\lambda_i(\mat{M})) \leq y \right\}, $$
%where $\# E$ denotes the cardinality of the set $E$.  If the matrix $\mat{M}$ is Hermitian, then the eigenvalues $\lambda_1(\mat{M}), \ldots, \lambda_N(\mat{M})$ are real.  In this case, the ESD is given by
%$$ F^{\mat{M}}(x) := \frac{1}{N} \# \left\{ 1 \leq i \leq N : \lambda_i(\mat{M}) \leq x \right\}. $$

A fundamental result for Wigner random matrices is Wigner's semicircle law \cite[Theorem 2.5]{BSbook}.  In particular, Wigner's semicircle law describes the convergence of the empirical spectral measure of $\frac{1}{\sqrt{N}} \mat{Y}_N$ when $\mat{Y}_N$ is a Wigner matrix.  

\begin{theorem}[Wigner's semicircle law]
Let $\xi, \zeta$ be real random variables, and assume $\xi$ has mean zero and unit variance.  For each $N \geq 1$, let $\mat{Y}_N$ be a real symmetric Wigner matrix of size $N$ with atom variables $\xi,\zeta$.  Then, for any bounded and continuous function $f: \mathbb{R} \to \mathbb{R}$, 
$$ \int_{\mathbb{R}} f(x) d \mu_{ \frac{1}{\sqrt{N}} \mat{Y}_N}(x)  \longrightarrow \int_{-2}^2 f(x) \frac{1}{2\pi} \sqrt{4 - x^2} dx $$
almost surely as $N \to \infty$.  
\end{theorem}

\begin{remark}
One can write 
$$ \int_{\mathbb{R}} f(x) d \mu_{ \frac{1}{\sqrt{N}} \mat{Y}_N}(x) = \frac{1}{N} \sum_{i=1}^N f \left( \lambda_i \left( \frac{1}{\sqrt{N}} \mat{Y}_N \right) \right). $$
In this way, Wigner's semicircle law can be viewed as a law of large numbers for Wigner matrices.  
\end{remark}

For iid random matrices, the limiting empirical spectral measure is described by the circular law.  Many authors have proved versions of the circular law under various assumptions on the atom variable $\xi$; see for instance \cite{Bcirc,BC,GTcirc,PZ,TVcirc,TVesd} and references therein.  We present the most general version due to Tao and Vu \cite{TVesd}.  

\begin{theorem}[Circular Law]
Let $\xi$ be a complex random variable with mean zero and unit variance.  For each $N \geq 1$, let $\mat{Y}_N$ be an iid random matrix of size $N$ with atom variable $\xi$.  Then, for any bounded and continuous function $f:\mathbb{C} \to \mathbb{C}$, 
$$ \int_{\mathbb{C}} f(z) d \mu_{ \frac{1}{\sqrt{N}} \mat{Y}_N }(z) \longrightarrow \frac{1}{\pi} \int_{\mathbb{U}} f(z) d^2z $$
almost surely as $N \to \infty$, where $\mathbb{U}$ is the unit disk in the complex plane and $d^2z = d \Re(z) d \Im(z)$.  
\end{theorem}

\subsection{Fluctuations of linear eigenvalue statistics}
While both Wigner's semicircle law and the circular law can be viewed as versions of the law of large numbers for random matrices, it is also natural to consider the fluctuations of linear spectral statistics.  That is, for any $N \times N$ matrix $\mat{M}$, we wish to study the sum
$$ \tr f( \mat{M} ) := \sum_{i=1}^N f (\lambda_i(\mat{M})), $$
where $f$ is a sufficiently smooth test function.  

Unlike the classical central limit theorem, the variance of linear spectral statistics for many ensembles of random matrices is $O(1)$ as the size of the matrix tends to infinity\footnote{See Section \ref{sec:notation} for a complete description of the asymptotic notation used here and throughout the paper.}.  More precisely, if $\mat{Y}_N$ is a Wigner or iid random matrix of size $N$, then, for any sufficiently smooth test function $f$,
$$ \limsup_{N \to \infty} \var \left[ \tr f \left( \frac{1}{\sqrt{N}} \mat{Y}_N \right) \right] < \infty. $$

A variety of results can be found in the random matrix literature concerning the fluctuations of linear spectral statistics for various random matrix ensembles under differing assumptions on the test functions $f$.  We refer the reader to \cite{AZ,BS,DS,DE,J,LP,NP,RS,S,SS,Scompact,SoWa} and references therein.  We present the following result for Wigner random matrices due to Shcherbina \cite{S}.  

\begin{theorem}[Shcherbina \cite{S}] \label{thm:shcherbina}
Let $\xi$ be a real random variable with mean zero, unit variance, and $\E|\xi|^4 < \infty$.  Let $\zeta$ be a real random variable with mean zero and variance $\sigma^2$.  For each $N \geq 1$, let $\mat{Y}_N$ be a real symmetric Wigner matrix of size $N$ with atom variables $\xi, \zeta$.  Let $f$ be a real-valued test function which satisfies 
$$ \int (1 + 2|l|)^{3/2 + \eps} |\hat{f}(l)|^2 dl < \infty $$
for some $\eps > 0$, where
$$ \hat{f}(l) := \frac{1}{\sqrt{2\pi}} \int f(x) e^{ix l} dx $$
is the Fourier transform of $f$.  Then 
$$ \tr f \left( \frac{1}{\sqrt{N}} \mat{Y}_N \right) - \E \tr f \left( \frac{1}{\sqrt{N}} \mat{Y}_N \right) $$
converges in distribution as $N \to \infty$ to a mean-zero Gaussian random variable with variance
\begin{align*}
	\frac{1}{2 \pi^2} & \int_{-2}^2 \int_{-2}^2 \left( \frac{f(x) - f(y)}{x - y} \right)^2 \frac{4 - xy}{\sqrt{4 - x^2} \sqrt{4 - y^2}} dx dy \\
	&+ \frac{\E|\xi|^4 - 3}{2 \pi^2} \left( \int_{-2}^2 f(x) \frac{2- x^2}{ \sqrt{4 - x^2}} dx \right)^2 + \frac{\sigma^2 - 2}{4 \pi^2} \left( \int_{-2}^2 \frac{f(x) x}{ \sqrt{4-x^2}} dx \right)^2.
\end{align*}
\end{theorem}

A similar result was obtained for iid random matrices with complex entries by Rider and Silverstein \cite{RS}.   

\begin{theorem}[Rider-Silverstein \cite{RS}] \label{thm:rs}
Let $\xi$ be a complex random variable with mean zero, unit variance, and which satisfies 
\begin{enumerate}[(i)]
\item $\E[ \xi^2] = 0$,
\item $\E|\xi|^k \leq k^{\alpha k}$ for $k > 2$ and some $\alpha > 0$,
\item $\Re(\xi)$ and $\Im(\xi)$ possess a bounded joint density.
\end{enumerate}
For each $N \geq 1$, let $\mat{Y}_N$ be an iid random matrix with atom variable $\xi$.  Consider test functions $f_1, \ldots, f_k$ analytic in a neighborhood of the disk $|z| \leq 4$ and otherwise bounded.  Then, as $N \to \infty$, the random vector
$$ \left( \tr f_j \left( \frac{1}{\sqrt{N}} \mat{Y}_N \right) - N f_j(0) \right)_{j=1}^k $$
converges in distribution to a mean-zero multivariate Gaussian vector $(\mathcal{G}(f_1), \ldots, \mathcal{G}(f_k))$ with covariances 
$$ \E[ \mathcal{G}(f_i) \overline{\mathcal{G}(f_j)}] = \frac{1}{\pi} \int_{\mathbb{U}} \frac{d}{dz} f_i(z) \overline{ \frac{d}{dz} f_j(z) } d^2 z, $$
in which $\mathbb{U}$ is the unit disk and $d^2 z = d \Re(z) d \Im(z)$.  
\end{theorem}

A version of Theorem \ref{thm:rs} was proved by Nourdin and Peccati \cite{NP} when $\mat{Y}_N$ is a real iid random matrix and $f_1, \ldots, f_k$ are polynomials.  

\subsection{Elliptic random matrices}
\label{sec:ell}

Elliptic random matrices generalize both Wigner matrices and iid random matrices. 

\begin{definition}[Real elliptic random matrix] \label{def:elliptic}
Let $(\xi_1, \xi_2)$ be a random vector in $\mathbb{R}^2$, and let $\zeta$ be a real random variable.  We say $\mat{Y}_N = (y_{ij})_{i,j=1}^N$ is a $N \times N$ \emph{real elliptic random matrix} with atom variables $(\xi_1,\xi_2), \zeta$ if the following conditions hold.  
\begin{itemize}
\item (independence) $\{ y_{ii} : 1 \leq i \leq N\} \cup \{ (y_{ij}, y_{ji}) : 1 \leq i < j \leq N\}$ is a collection of independent random elements.
\item (off-diagonal entries) $\{ (y_{ij}, y_{ji}) : 1 \leq i < j \leq N\}$ is a collection of iid copies of $(\xi_1,\xi_2)$.
\item (diagonal entries) $\{y_{ii} : 1 \leq i \leq N\}$ is a collection of iid copies of $\zeta$.  
\end{itemize}
\end{definition}

Let $\mat{Y}_N$ be a $N \times N$ real elliptic random matrix with atom variables $(\xi_1, \xi_2), \zeta$.  Assume $\xi_1, \xi_2$ have mean zero and unit variance.  The key parameter when studying elliptic random matrices turns out to be the covariance $\rho := \E[\xi_1 \xi_2]$.  By the Cauchy-Schwarz inequality, it follows that $|\rho| \leq 1$.  

If $\rho = 1$, then $\xi_1 = \xi_2$ almost surely and hence $\mat{Y}_N$ is a real symmetric Wigner matrix.  If $\xi_1, \xi_2, \zeta$ are iid, then $\rho = 0$ and $\mat{Y}_N$ is an iid random matrix.  

For elliptic random matrices, the limiting spectral distribution is known as the elliptic law.  In particular, the limiting distribution is given by the uniform measure on the ellipsoid 
\begin{equation} \label{def:Erho}
	\mathcal{E}_{\rho} := \left\{ z \in \C : \frac{(\Re z)^2}{(1+\rho)^2} +\frac{(\Im z)^2}{(1-\rho)^2} < 1 \right\} 
\end{equation}
for $|\rho| < 1$.  Versions of the elliptic law, under various assumptions on the entries, have been established in \cite{Nell, NO}.   

\begin{theorem}[Elliptic law]
Let $(\xi_1, \xi_2)$ be a random vector in $\mathbb{R}^2$, where $\xi_1, \xi_2$ each have mean zero and unit variance.  Set $\rho := \E[\xi_1 \xi_2]$, and assume $|\rho| < 1$.  Let $\zeta$ be a real random variable with mean zero and finite variance.  For each $N \geq 1$, let $\mat{Y}_N$ be an $N \times N$ real elliptic random matrix with atom variables $(\xi_1, \xi_2), \zeta$.  Then, for any bounded and continuous $f: \mathbb{C} \to \mathbb{C}$, 
$$ \int_{\mathbb{C}} f(z) d \mu_{\frac{1}{\sqrt{N}} \mat{Y}_N}(z) \longrightarrow \frac{1}{\pi (1- \rho^2)} \int_{\mathcal{E}_{\rho}} f(z) d^2 z $$
almost surely as $N \to \infty$, where $\mathcal{E}_{\rho}$ is defined in \eqref{def:Erho} and $d^2 z = d \Re(z) d \Im(z)$.  
\end{theorem}

\section{New results}

The goal of this note is to study the fluctuations of linear eigenvalues statistics for elliptic random matrices.  That is, we will prove versions of Theorems \ref{thm:shcherbina} and \ref{thm:rs} for a class of real elliptic random matrices.  In particular, we consider real elliptic random matrices whose atom variables $(\xi_1, \xi_2), \zeta$ satisfy the following conditions.
\begin{definition}[Condition {\bf C0}]
We say the atom variables $(\xi_1,\xi_2), \zeta$ satisfy condition {\bf C0} if
\begin{enumerate}[(i)]
\item $\xi_1, \xi_2$ each have mean zero and unit variance,
\item $\zeta$ has mean zero and variance $\sigma^2$,
\item there exists $\tau > 0$ such that 
\begin{equation} \label{eq:finitemoments}
	\E|\xi_1|^{6+\tau} + \E|\xi_2|^{6+\tau} + \E|\zeta|^{4+\tau} < \infty.
\end{equation}
\end{enumerate}
\end{definition}

Recall the definition of the ellipsoid $\mathcal{E}_{\rho}$ given in \eqref{def:Erho} for $|\rho| < 1$.  For $\rho = 1$, we define $\mathcal{E}_{1}$ to be the interval $[-2,2]$ on the real line.  For $\rho = -1$, let $\mathcal{E}_{-1}$ be the interval from $-2i$ to $2i$ on the imaginary axis.  For any $\delta > 0$, define the neighborhoods 
$$ \mathcal{E}_{\rho,\delta} := \left\{ z \in \mathbb{C} : \dist(z,\mathcal{E}_\rho) \leq \delta \right\}. $$

As in \cite{OR}, we will also need the function 
\begin{equation} \label{eq:def:mz}
	m(z) := \left\{
     		\begin{array}{lr}
       		\frac{ -z + \sqrt{z^2 - 4 \rho}}{2 \rho} & \text{for} \quad \rho \neq 0\\
       		\frac{-1}{z} & \text{for}\quad \rho = 0
     		\end{array}
   	\right. ,
\end{equation}
where $\sqrt{z^2 - 4 \rho}$ is the branch of the square root with branch cut $[-2\sqrt{\rho}, 2\sqrt{\rho}]$ for $\rho > 0$ and $[-2\sqrt{|\rho|}, 2\sqrt{|\rho|}]i$ for $\rho < 0$, and which equals $z$ at infinity.  In particular, the function $m(z)$ is analytic outside $\mathcal{E}_{\rho}$ and satisfies 
$$ m(z) = -\frac{1}{z + \rho m(z)}. $$

We are now ready to state our main result.  

\begin{theorem}[Main result] \label{thm:nice}
For each $N \geq 1$, let $\mat{Y}_N$ be an $N \times N$ real elliptic random matrix with atom variables $(\xi_1,\xi_2), \zeta$ which satisfy condition {\bf C0}.  Set $\rho :=\E[\xi_1 \xi_2]$.  Let $\delta > 0$.  Let $f_1, \ldots, f_k$ be analytic in a neighborhood of $\mathcal{E}_{\rho,\delta}$ and bounded otherwise.  In addition, assume
\begin{equation} \label{eq:fjreal}
	f_j(z) + f_j(\bar{z}) \in \mathbb{R}
\end{equation}  
for all $z \in \mathcal{E}_{\rho,\delta}$ and each $1 \leq j \leq k$.  Then, as $N \to \infty$, the random vector
$$ \left( \tr f_j \left( \frac{1}{\sqrt{N}} \mat{Y}_N \right) - \E  \tr f_j \left( \frac{1}{\sqrt{N}} \mat{Y}_N \right) \right)_{j=1}^k $$
converges in distribution to a mean-zero multivariate Gaussian vector $(\mathcal{G}(f_1), \ldots, \mathcal{G}(f_k))$ with covariances
\begin{align}
	\E[ \mathcal{G}(f_i) \mathcal{G}(f_j) ] &:= -\frac{1}{4 \pi^2} \oint_{\mathcal{C}} \oint_{\mathcal{C}} f_i(z) f_j(w) \upsilon(z,w) dz dw \label{eq:cov:form1}\\
		&= -\frac{1}{4 \pi^2} \oint_{\mathcal{C}} \oint_{\mathcal{C}} f_i'(z) f_j'(w) m(z) m(w) \beta(z,w) dz dw, \nonumber
\end{align}
where $\mathcal{C}$ is the contour around the boundary of $\mathcal{E}_{\rho,\delta}$, 
\begin{equation} \label{def:upsilon}
	\upsilon(z,w) := \frac{\partial^2}{\partial z \partial w} m(z) m(w) \beta(z,w),
\end{equation} 
\begin{align}
	\beta(z,w) &:= \sigma^2 - \rho - 1 - \frac{ \log( 1- \rho m(z) m(w)) }{ m(z) m(w) } - \frac{ \log (1 - m(z) m(w)) }{m(z) m(w)} \nonumber \\
	&\qquad+ \left( \frac{\E[\xi_1^2 \xi_2^2] - 2 \rho^2 - 1}{2} \right) m(z) m(w), \label{def:beta}
\end{align}
 and $m(z)$ is defined in \eqref{eq:def:mz}.  
\end{theorem}

\begin{remark}
We only require that the functions $f_1, \ldots, f_k$ be analytic in a neighborhood of $\mathcal{E}_{\rho}$.  As remarked in \cite{RS}, this is a more natural assumption than in Theorem \ref{thm:rs}, which requires analyticity on a larger domain. The proof in \cite{RS} uses estimates for the spectral norm of $\mat Y_N$ but not for the spectral radius. In Appendix \ref{sec:lsvproof}, we prove, using techniques from \cite{OR}, a sufficiently strong estimate on the spectral radius to use as an input in our proof.
\end{remark}

\begin{remark}
Condition \eqref{eq:fjreal} ensures that $\tr f_j \left( \frac{1}{\sqrt{N}} \mat{Y}_N \right)$ is real valued.  This is a natural condition since we only consider elliptic random matrices with real entries.  Many analytic functions possess this property.  For instance, if 
$$ f(z) = \sum_{n=0}^\infty a_n z^n $$
in a neighborhood of $\mathcal{E}_{\rho,\delta}$, where $a_n \in \mathbb{R}$, then $f$ satisfies condition \eqref{eq:fjreal}. 

Furthermore, if $f$ is an analytic function that does not satisfy condition \eqref{eq:fjreal}, then $f(z)$ can be rewritten as $\frac{1}{2} \left( f(z) + \overline{ f( \overline z) }\right) + \frac{1}{2} \left( f(z) - \overline{ f( \overline z) )} \right)$. The functions $ \frac{1}{2} \left( f(z) + \overline{ f( \overline z) }\right)$ and $\frac{1}{2i} \left( f(z) - \overline{ f( \overline z)} \right)$ both satisfy condition \eqref{eq:fjreal}.  Thus, Theorem \ref{thm:nice} can be applied to compute the limiting covariances of the real and imaginary parts of $\tr f \left( \frac{1}{\sqrt{N}} \mat{Y}_N \right)$.  See Remark \ref{rem:fjnotreal} for further discussion.  
\end{remark}

\begin{remark} \label{rem:rho0}
When $\xi_1, \xi_2, \zeta$ are iid random variables with mean zero and unit variance (hence $\rho=0$), it follows that
$$ v(z,\overline{w}) = \frac{1}{(1-z \overline{w})^2}. $$
Up to a factor of $\frac{1}{\pi}$, this function is the Bergman kernel function for the unit disk in $\mathbb{C}$.  See \cite{Bell} for further details regarding the Bergman kernel.  
\end{remark}

In the case that $\xi_1, \xi_2, \zeta$ are iid random variables, we have the following immediate corollary.  (See also Remark \ref{rem:rho0} above.)

\begin{corollary}
Let $\xi$ be a real random variable with mean zero, unit variance, and $\E|\xi|^{6 + \tau} < \infty$, for some $\tau > 0$.  For each $N \geq 1$, let $\mat{Y}_N$ be a iid random matrix of size $N$.  Let $f_1, \ldots, f_k$ be analytic in a neighborhood of the disk $|z| \leq 1$ and bounded otherwise.  In addition, assume \eqref{eq:fjreal} holds for all $z$ in a neighborhood of the disk $|z| \leq 1$ and each $1 \leq j \leq k$.  Then, as $N \to \infty$, the random vector
$$ \left( \tr f_j \left( \frac{1}{\sqrt{N}} \mat{Y}_N \right) - \E  \tr f_j \left( \frac{1}{\sqrt{N}} \mat{Y}_N \right) \right)_{j=1}^k $$
converges in distribution to a mean-zero multivariate Gaussian vector $(\mathcal{G}(f_1), \ldots, \mathcal{G}(f_k))$ with covariances
$$ \E[ \mathcal{G}(f_i) \mathcal{G}(f_j) ] := -\frac{1}{4 \pi^2} \oint_{\mathcal{C}} \oint_{\mathcal{C}} f_i(z) f_j(w) (1 - zw)^{-2} dz dw, $$
where $\mathcal{C}$ is a contour lying within the region of analyticity of $f_1, \ldots, f_k$, but enclosing the unit disk.  
\end{corollary}

\begin{remark}
It follows from the calculations in \cite{RS} that 
$$ -\frac{1}{4 \pi^2} \oint_{\mathcal{C}} \oint_{\mathcal{C}} f_i(z) f_j(w) (1 - zw)^{-2} dz dw = \frac{1}{\pi} \int_{\mathbb{U}} \frac{d}{dz} f_i(z) \frac{d}{dz} f_j(z) d^2 z, $$
where $\mathbb{U}$ is the unit disk and $d^2 z = d \Re(z) d \Im(z)$.  
\end{remark}

The expression for the covariance $\E[ \mathcal{G}(f_i) \mathcal{G}(f_j) ]$ given in \eqref{eq:cov:form1} is rather unintuitive.  In the case when $\rho =1$ and $\mat{Y}_N$ is a real-symmetric Wigner matrix, the covariance can be written in terms of the Chebyshev polynomials and as an integral involving the semicircle density function; see \cite[Chapter 9]{BSbook} and Theorem \ref{thm:shcherbina} for further details.  In Proposition \ref{prop:Cheb} below, we show that for general $-1 \leq \rho \leq 1$, the covariance is related to a suitably rescaled version of the Chebyshev polynomials.  

\begin{proposition} \label{prop:Cheb}
Under the assumptions of Theorem \ref{thm:nice}, there exists $0 < r < 1$ (depending only on $\delta$) such that
\begin{align*}
	\E[ \mathcal{G}(f_i) \mathcal{G}(f_j) ] &=  ( \sigma^2 - \rho - 1) a_{1}(f_i) a_{1}(f_j) + \left( \E[\xi_1^2 \xi_2^2] - 2 \rho^2 - 1 \right)  a_{2}(f_i) a_{2}(f_j)  \\
		&\qquad\qquad +  \sum_{l=1}^{\infty} l (1+\rho^{l} )  a_{l}(f_i) a_{l}(f_j), 
\end{align*}
where
\[a_l(f) := \frac{1}{2 \pi i } \oint_{|s|=r}  s^{l-1}  f(\rho s+ 1/s)  d s  \]
for $l \geq 1$.  Furthermore, if $T_j$ is the $j^{th}$ Chebyshev polynomial, and $F_j(z) :=  2  \rho^{j/2} T_j(  z/(2 \sqrt{\rho}))$ for $\rho \not =0$ (and $F_j(z) :=  z^j $ for the case when $\rho =0$), one has
\[ \E[ \mathcal{G}(F_j) \mathcal{G}(F_k) ] = \delta_{jk} \left(  ( \sigma^2 - \rho - 1) \delta_{1j} +   \left( \E[\xi_1^2 \xi_2^2] - 2 \rho^2 - 1 \right)  \delta_{2j}  + j (1+\rho^{j} ) \right),\]
where $\delta_{jk}$ is the Kronecker delta.  
\end{proposition}

The polynomials $F_l$ introduced above are known as the Faber polynomials associated with the function $\rho z + z^{-1}$ or the domain $\mathcal{E}_\rho$; see, for instance, \cite{F, Ma}. The Faber polynomials form a basis for analytic functions on $\mathcal{E}_\rho$ such that $f(z) = \sum_{l=0}^\infty a_l F_{l}(z)$; see the proof of Proposition \ref{prop:Cheb} and \eqref{eq:fexpandFaber} below for further details.

\begin{proof}[Proof of Proposition \ref{prop:Cheb}]

We assume $\rho \not= 0$. The $\rho = 0$ case was studied in \cite{RS}, and our argument can easily be modified to handle that case as well.

When computing the covariance in \eqref{eq:cov:form1} we can, by Cauchy's theorem, integrate along $\mathcal{C'}$, parametrized by $z=r \rho e^{ i t} + r^{-1} e^{- i t}$, $0 \leq t \leq 2 \pi$ for some $0<r<1$ sufficiently close to $1$. Note that in contrast to all other contours of integration, this contour is traversed counter-clockwise.

Then changing variables $s=m(z)$ and $t=m(w)$, so $z = \rho s + s^{-1}$ and $w = \rho t + t^{-1}$, leads to
\begin{align} \label{chebvar}
	\E[ \mathcal{G}(f_i) \mathcal{G}(f_j) ] &= -\frac{1}{4 \pi^2} \oint_{\mathcal{C'}} \oint_{\mathcal{C'}} f_i(z) f_j(w) \left( m'(z) m'(w) ( \sigma^2 - \rho - 1) \right. \\
	&\qquad + \frac{  \rho m'(z) m'(w) }{(1 - \rho m(z) m(w))^2 } + \frac{   m'(z) m'(w) }{(1-m(z) m(w))^2} \nonumber \\
	&\qquad \left. + \left( \E[\xi_1^2 \xi_2^2] - 2 \rho^2 - 1 \right) m'(z) m'(w) m(z) m(w) \right)  dz dw  \nonumber \\
&= -\frac{1}{4 \pi^2} \oint_{|s|=r} \oint_{|t|=r} f_i(\rho s+ 1/s) f_j(\rho t+1/t)  \bigg( ( \sigma^2 - \rho - 1) \nonumber \\
         &\qquad  - \frac{ - \rho  }{(1 - \rho s t)^2 } - \frac{  - 1 }{(1- s t)^2} + \left( \E[\xi_1^2 \xi_2^2] - 2 \rho^2 - 1 \right) s t \bigg)  ds dt\nonumber \\
%&= -\frac{1}{4 \pi^2} \oint_{|s|=r} \oint_{|t|=r} f_i(\rho z+ 1/z) f_j(\rho w+1/w) ( ( \sigma^2 - \rho - 1) + \rho  \sum_{l=1}^{\infty} l (\rho z w )^{l-1}   +   \sum_{l=1}^{\infty} l ( z w )^{l-1}\nonumber \\
	%&\qquad+ \left( \E[\xi_1^2 \xi_2^2] - 2 \rho^2 - 1 \right) z w )  dz dw \\	
&= -\frac{1}{4 \pi^2} \oint_{|s|=r} \oint_{|t|=r} f_i(\rho s+ 1/s) f_j(\rho t+1/t)  \bigg( ( \sigma^2 - \rho - 1) \nonumber \\
 	&\qquad  +   \sum_{l=1}^{\infty} l (1+\rho^{l} )(s t )^{l-1} + \left( \E[\xi_1^2 \xi_2^2] - 2 \rho^2 - 1 \right) s t \bigg)  ds dt. \nonumber
\end{align}

On the other hand, for an arbitrary function $f$ analytic in a neighborhood containing $\mathcal{E}_{\rho}$, we have
\[f(z) =  \frac{1}{2 \pi i} \oint_{\mathcal{C'}}  \frac{f(s)}{z- s} ds = \frac{1}{2 \pi i} \oint_{|s|=r} \frac{f(\rho s + 1/s) \frac{d}{ds}(\rho s + 1/s)}{z- (\rho s+ 1/s)} ds.\]
Rewriting, we obtain
\[  \frac{ \rho - 1/s^2}{z- ( \rho s+1/ s)} %= s^{-1}  \frac{ s^2 \rho - 1 }{s z-  \rho s^2 - 1 }  
=   s^{-1} \left( 2  \frac{ 1  - \frac{s}{2} z  }{1  -s z  + \rho s^2} -1  \right), \]
and recalling that $\sum_{l=0}^{\infty} T_l(x) t^l = \frac{1 - t x}{ 1 - 2tx + t^2}$ is the generating function for the Chebyshev polynomials, we find that
\[\frac{ \rho - 1/s^2}{z- ( \rho s+1/ s)} = 2 s^{-1} \sum_{l=0}^{\infty} \left( \sqrt{\rho}  \right)^l T_l(  z/(2 \sqrt{\rho})  ) s^l     - s^{-1}.\]
Thus, we obtain the following expansion of $f$ in terms of the rescaled Chebyshev polynomials:
\begin{equation} \label{eq:fexpandFaber}
	f(z) = \sum_{l=0}^{\infty} a_l(f) F_l(z), 
\end{equation} 
where
\[ a_l(f) := \frac{1}{2 \pi i } \oint_{|s|=r}  s^{l-1}  f(\rho s+ 1/s)  d s,  \]
\[ F_0 := 1, \text{  and }  F_l(z) :=  2  \rho^{l/2} T_l(  z/(2 \sqrt{\rho})) \text{ for } l \geq 1. \]

Returning to the covariance, \eqref{chebvar}, we conclude that
\begin{align*}
	\E[ \mathcal{G}(f_i) \mathcal{G}(f_j) ] &=  ( \sigma^2 - \rho - 1) a_{1}(f_i) a_{1}(f_j) +  \sum_{l=1}^{\infty} l (1+\rho^{l} )  a_{l}(f_i) a_{l}(f_j) \nonumber \\
	&\qquad+ \left( \E[\xi_1^2 \xi_2^2] - 2 \rho^2 - 1 \right)  a_{2}(f_i) a_{2}(f_j), 
\end{align*}	
as desired.
\end{proof}

\begin{remark} \label{rem:fjnotreal}
Consider a function $f$ analytic in a neighborhood that contains $\mathcal{E}_{\rho}$ and bounded otherwise, and assume $f$ does not satisfy \eqref{eq:fjreal}.  In view of Proposition \ref{prop:Cheb} and \eqref{eq:fexpandFaber}, we can write
$$ f(z) = \sum_{l=0}^\infty a_l(f) F_l(z), $$
where $F_l$ are the Faber polynomials associated to the domain $\mathcal{E}_{\rho}$ (see \cite{F, Ma} for further details).  Thus, the functions
$$ g(z) := \sum_{l=0}^\infty \Re(a_l(f)) F_l(z) $$ 
and
$$ h(z) := \sum_{l=0}^\infty \Im(a_l(f)) F_l(z) $$
both satisfy the assumptions of Theorem \ref{thm:nice}.  Applying Theorem \ref{thm:nice} to the functions $g$ and $h$, we find that
$$ \tr f \left( \frac{1}{\sqrt{N}} \mat{Y}_N \right) - \E \tr f \left( \frac{1}{\sqrt{N}} \mat{Y}_N \right) $$
converges in distribution, as $N \to \infty$, to the mean-zero complex Gaussian $\mathcal{G}(g) + i \mathcal{G}(h)$.  
\end{remark}

Theorem \ref{thm:nice} will follow from the slightly more technical Theorem \ref{thm:main} below.  We first state some notation.  For an $N \times N$ matrix $\mat{M}$, let $\sigma_1(\mat{M}) \geq \cdots \geq \sigma_N(\mat{M}) \geq 0$ denote the singular values of $\mat{M}$.  In particular, $\sigma_N(\mat{M})$ is the least singular value of $\mat{M}$ and can be written
$$ \sigma_N(\mat{M}) = \min_{\|x\| = 1} \|\mat{M} x \|. $$
Here $\|x\|$ denotes the Euclidean norm of the vector $x$.  A bound for the least singular value of a shifted elliptic random matrix will play a significant role in the proof of Theorem \ref{thm:main}.  

\begin{theorem} \label{thm:main}
For each $N \geq 1$, let $\mat{Y}_N$ be an $N \times N$ real elliptic random matrix with atom variables $(\xi_1,\xi_2), \zeta$ which satisfy condition {\bf C0}.  Set $\rho :=\E[\xi_1 \xi_2]$.  Let $f_1, \ldots, f_k$ be analytic in a neighbor of $\mathcal{E}_{\rho,\delta}$ for some $\delta > 0$, and assume \eqref{eq:fjreal} holds for all $z \in \mathcal{E}_{\rho,\delta}$ and each $1 \leq j \leq k$.  Then there exists $c > 0$ such that the event 
\begin{equation} \label{eq:def:EN}
	E_N := \left\{ \inf_{\dist(z, \mathcal{E}_{\rho}) \geq \delta} \sigma_N\left(\frac{1}{\sqrt{N}} \mat{Y}_N - z \mat{I} \right) \geq \frac{c}{2} \right\}, 
\end{equation}
holds with probability $1-o(N^{-1})$, and, as $N \to \infty$, the random vector
$$ \left( \tr f_j \left( \frac{1}{\sqrt{N}} \mat{Y}_N \right) \oindicator{E_N} - \E  \tr f_j \left( \frac{1}{\sqrt{N}} \mat{Y}_N \right) \oindicator{E_N} \right)_{j=1}^k $$
converges in distribution to a mean-zero multivariate Gaussian vector $(\mathcal{G}(f_1), \ldots, \mathcal{G}(f_k))$ with covariances
\begin{align*}
	\E[ \mathcal{G}(f_i) \mathcal{G}(f_j) ] &:= -\frac{1}{4 \pi^2} \oint_{\mathcal{C}} \oint_{\mathcal{C}} f_i(z) f_j(w) \upsilon(z,w) dz dw \\
		&= -\frac{1}{4 \pi^2} \oint_{\mathcal{C}} \oint_{\mathcal{C}} f_i'(z) f_j'(w) m(z) m(w) \beta(z,w) dz dw, 
\end{align*}
where $\oindicator{E_N}$ denotes the indicator function of the event $E_N$, $\mathcal{C}$ is the contour around the boundary of $\mathcal{E}_{\rho,\delta}$ and $m(z)$, $\upsilon(z,w)$, $\beta(z,w)$ are defined in \eqref{eq:def:mz}, \eqref{def:upsilon}, \eqref{def:beta}.
\end{theorem}

\begin{remark} \label{rem:singeig}
On the event $E_N$ (defined in \eqref{eq:def:EN}) all eigenvalues of $\frac{1}{\sqrt{N}} \mat{Y}_N$ are contained in the interior of $\mathcal{E}_{\rho,\delta}$.  Indeed, we observe that $z$ is an eigenvalue of $\frac{1}{\sqrt{N}} \mat{Y}_N$ if and only if 
$$ \det \left( \frac{1}{\sqrt{N}} \mat{Y}_N - z \mat{I}_N \right) = 0, $$
where $\mat{I}_N$ is the $N \times N$ identity matrix.  Since 
$$ \left| \det \left( \frac{1}{\sqrt{N}} \mat{Y}_N - z \mat{I}_N \right) \right| = \prod_{i=1}^N \sigma_i\left( \frac{1}{\sqrt{N}} \mat{Y}_N - z \mat{I}_N \right), $$
it follows that $z$ is an eigenvalue of $\frac{1}{\sqrt{N}} \mat{Y}_N$ if and only if $\sigma_N \left( \frac{1}{\sqrt{N}} \mat{Y}_N - z \mat{I}_N \right) = 0$.  
\end{remark}

We now prove Theorem \ref{thm:nice} assuming Theorem \ref{thm:main}.  

\begin{proof}[Proof of Theorem \ref{thm:nice}]
Define
$$ \mat{X}_N := \frac{1}{\sqrt{N}} \mat{Y}_N. $$ 
By the Cram\'{e}r-Wold device, it suffices to consider linear combinations of the real random variables
$$ \tr f_j \left( \mat{X}_N \right) - \E \tr f_j \left( \mat{X}_N \right), \quad 1 \leq j \leq k. $$
Let $f$ be a linear combination of $f_1, \ldots, f_k$.  Then $f$ is analytic in a neighborhood of $\mathcal{E}_{\rho,\delta}$ and bounded otherwise.  We write
\begin{align*}
	\tr f \left( \mat{X}_N \right) - \E \tr f \left( \mat{X}_N \right) &= \tr f \left( \mat{X}_N \right)\oindicator{E_N} - \E \tr f \left( \mat{X}_N \right)\oindicator{E_N} \\
	&\quad+ \tr f \left( \mat{X}_N \right)\oindicator{E_N^C} - \E \tr f \left( \mat{X}_N \right)\oindicator{E_N^C},
\end{align*}
where $E_N^C$ denotes the complement of the event $E_N$.  Thus, in view of Theorem \ref{thm:main}, it suffices to show that
$$ \tr f \left( \mat{X}_N \right)\oindicator{E_N^C} - \E \tr f \left( \mat{X}_N \right)\oindicator{E_N^C} $$
converges to zero in probability as $N \to \infty$.  

Also from Theorem \ref{thm:main}, we observe that $\Prob(E_N^C) = o(N^{-1})$.  Therefore, we have
$$ \E \left| \tr f(\mat{X}_N) \oindicator{E_N^C} \right| \leq N \|f \|_{\infty} \Prob(E_N^C) = o(1), $$
where $\| \cdot \|_{\infty}$ denotes the supremum norm.  Finally, we note that, for any $\eta > 0$, 
$$ \Prob\left( \left| \tr f \left( \mat{X}_N \right)\oindicator{E_N^C} \right| > \eta \right) \leq \Prob(E_N^C) = o(N^{-1}), $$
and the proof of Theorem \ref{thm:nice} is complete.  
\end{proof}

\subsection{Overview and organization}

It remains to prove Theorem \ref{thm:main}.  Broadly speaking the proof of Theorem \ref{thm:main} proceeds as follows.  For any analytic function $f$, we write
$$ \tr f \left( \frac{1}{\sqrt{N}} \mat{Y}_N \right) \oindicator{E_N} - \E  \tr f \left( \frac{1}{\sqrt{N}} \mat{Y}_N \right) \oindicator{E_N} = \frac{-1}{2\pi i} \oint_{\mathcal{C}} f(z) \Xi_N(z)dz, $$
where $\mathcal{C}$ is a contour whose interior contains the eigenvalues of $\frac{1}{\sqrt{N}} \mat{Y}_N$, 
$$ \Xi_N(z) := \tr \left( \frac{1}{\sqrt{N}} \mat{Y}_N - z \mat{I}_N \right)^{-1} \oindicator{E_N} - \E \tr \left( \frac{1}{\sqrt{N}} \mat{Y}_N - z \mat{I}_N \right)^{-1} \oindicator{E_N}, $$
and $\mat{I}_N$ denotes the $N \times N$ identity matrix.  We view $\Xi_N(z)$ as a process in $z \in \mathcal{C}$.  Thus, the proof of Theorem \ref{thm:main} reduces to showing that $\Xi_N$ converges weakly to an appropriate Gaussian process in the space of continuous functions on the contour $\mathcal{C}$.  The main technical challenge that arises when working with $\Xi_N(z)$ is the need to bound the spectral norm of the matrix $\left(\frac{1}{\sqrt{N}} \mat{Y}_N - z \mat{I}_N \right)^{-1}$ (or, equivalently, control the least singular value of $\frac{1}{\sqrt{N}} \mat{Y}_N - z \mat{I}_N$), and hence it becomes necessary to work on the event $E_N$.  

The paper is organized as follows.  In Section \ref{sec:prelim}, we present some preliminary tools we will need to prove Theorem \ref{thm:main}, including a bound on the least singular value of $\frac{1}{\sqrt{N}} \mat{Y}_N - z \mat{I}_N$.  We begin the proof of Theorem \ref{thm:main} in Section \ref{sec:proof}.  In Sections \ref{sec:distr} and \ref{sec:tight}, we show that a truncated version of $\Xi_N$ converges weakly to a Gaussian process in the space of continuous functions on an appropriately chosen contour $\mathcal{C}$.  In general, these two sections are based on \cite[Chapter 9]{BSbook} and \cite{RS}.  In fact, our proof seems to inherit many of the technical challenges present in \cite[Chapter 9]{BSbook} and \cite{RS}.  In particular, the material presented in these two sections is rather technical and some of the calculations are tedious.  Finally, we complete the proof of Theorem \ref{thm:main} in Section \ref{sec:conclusion}.  In addition, the appendix contains a number of auxiliary results.

\subsection{Notation} \label{sec:notation}

We use asymptotic notation (such as $O,o$) under the assumption that $N \rightarrow \infty$.  We use $X = O(Y)$ to denote the bound $X \leq CY$ for all sufficiently large $N$ and for some constant $C$.  Notation such as $X=O_k(Y)$ mean that the hidden constant $C$ depends on another constant $k$.  $X=o(Y)$ or $Y=\omega(X)$ means that $X/Y \rightarrow 0$ as $N \rightarrow \infty$.  

An event $E$, which depends on $N$, is said to hold with \emph{overwhelming probability} if $\Prob(E) \geq 1 - O_C(N^{-C})$ for every constant $C > 0$.  We let $\oindicator{E}$ denote the indicator function of the event $E$.  $E^{C}$ denotes the complement of the event $E$.  We write a.s. for almost surely. 

For any matrix $\mat M$, we denote the Hilbert-Schmidt norm $\|\mat M\|_2$ by the formula
\begin{equation*} 
	\|\mat M\|_2 := \sqrt{ \tr (\mat{M} \mat{M}^\ast) } = \sqrt{ \tr (\mat{M}^\ast \mat{M})}. 
\end{equation*}
Let $\|\mat{M}\|$ denote the spectral norm of $\mat{M}$.  We let $\mat{I}_N$ denote the $N \times N$ identity matrix.  Often we will just write $\mat{I}$ for the identity matrix when the size can be deduced from the context.     

We let $C$ and $K$ denote constants that are non-random and may take on different values from one appearance to the next.  The notation $K_p$ means that the constant $K$ depends on another parameter $p$.

\section{Preliminary tools} \label{sec:prelim}

Let $\mat{Y}_N$ be an elliptic random matrix with atom variables $(\xi_1,\xi_2),\zeta$ which satisfy condition {\bf C0}.  Define
$$ \mat{X}_N := \frac{1}{\sqrt{N}} \mat{Y}_N. $$ 

\subsection{Truncation} \label{sec:truncation}
Instead of working with the matrix $\mat{Y}_N$ directly, we will work with a truncated version of this matrix.  From \eqref{eq:finitemoments}, there exists $\eps > 0$ such that, taking $\eps_N := N^{-\eps}$, 
\begin{equation} \label{eq:xiilim}
	\lim_{N \to \infty} \frac{1}{\eps_N^6} \E|\xi_i|^6\indicator{|\xi_i| > \eps_N\sqrt{N}} = 0 
\end{equation}
for $i=1,2$ and
$$ \lim_{N \to \infty} \frac{1}{\eps_N^4} \E|\zeta|^4 \indicator{|\zeta| > \eps_N\sqrt{N}} = 0. $$

Define
$$ \tilde{\xi}_i = \xi_i \indicator{|\xi_i| \leq \eps_N \sqrt{N}} - \E \xi_i \indicator{|\xi_i| \leq \eps_N \sqrt{N}} $$
for $i=1,2$ and
$$ \tilde{\zeta} = \zeta \indicator{|\zeta| \leq \eps_N \sqrt{N}} - \E\zeta \indicator{|\zeta| \leq \eps_N \sqrt{N}} . $$
We also define
$$ \hat{\xi}_i = \frac{\tilde{\xi}_i}{\sqrt{\var(\tilde{\xi}_i)}}, \quad \hat{\zeta} = \sigma \frac{\tilde{\zeta}}{\sqrt{\var(\tilde{\zeta})}} $$
for $i=1,2$.  Set $\hat{\rho} := \E[ \hat{\xi}_1 \hat{\xi}_2]$.  Of course, $\tilde{\xi}_i, \hat{\xi}_i, \tilde{\zeta}, \hat{\zeta}, \hat{\rho}$ depend on $N$, but we do not denote this dependence in our notation.  

\begin{lemma}[Truncation] \label{lemma:truncation}
Assume the atom variables $(\xi_1, \xi_2), \zeta$ satisfy condition {\bf C0}.  Then
\begin{enumerate}[(i)]
\item $|1 - \var(\tilde{\xi_i})| = o(\eps_N^2 N^{-2})$ for $i=1,2$,  \label{item:var1}
\item $|\sigma^2 - \var(\tilde{\zeta})| = o(\eps_N^2 N^{-1})$, \label{item:varsigma}
\item there exists $N_0 > 0$ such that for any $N > N_0$, $\hat{\xi}_1, \hat{\xi}_2$ both have mean zero and unit variance, $\hat{\zeta}$ has mean zero and variance $\sigma^2$, and a.s. \label{item:asbnd}
$$ |\hat{\xi}_i| \leq 4 \eps_N \sqrt{N} \quad \text{and} \quad |\hat{\zeta}| \leq 4 \eps_N \sqrt{N} $$
for $i=1,2$,  
\item there exists $N_0 > 0$ such that for any $N > N_0$,  \label{item:momentbnd}
$$ \E |\hat{\xi}_i|^6 \leq 2^{12} \E |\xi_i|^6 \quad \text{and} \quad \E|\hat{\zeta}|^4 \leq 2^8 \E |\zeta|^4 $$
for $i=1,2$, 
\item $\hat{\rho} = \rho + o(N^{-1})$,  \label{item:rho}
\item $\E[\hat{\xi}_1^2 \hat{\xi}_2^2] = \E[\xi_1^2 \xi_2^2] + o(1)$. \label{item:highmom}
\end{enumerate}
\end{lemma}
\begin{proof}
For \eqref{item:var1} and \eqref{item:varsigma}, we observe that
$$ |1 - \var(\tilde{\xi_i})| \leq 2 \E|\xi_i|^2 \indicator{|\xi_i| > \eps_N \sqrt{N}} \leq \frac{1}{ \eps_N^4 N^2} \E|\xi_i|^6 \indicator{|\xi_i| > \eps_N \sqrt{N}} = o(\eps_N^2 N^{-2}) $$
and, similarly,  $|\sigma^2 - \var(\tilde{\zeta})| = o(\eps_N^2 N^{-1})$.

By the estimates above, there exists $N_0 > 0$ such that $\var(\tilde{\xi}_1) \geq 1/4$, $\var(\tilde{\xi}_2) \geq 1/4$, and $\var(\tilde{\zeta}) \geq \sigma^2/4$ for all $N > N_0$.  Taking $N > N_0$, we observe that $\hat{\xi}_1, \hat{\xi}_2$ have mean zero and unit variance, and $\hat{\zeta}$ has mean zero and variance $\sigma^2$ by construction.  In addition, the almost sure bounds in \eqref{item:asbnd} follow from the bounds above and the definitions of $\tilde{\xi}_1$, $\tilde{\xi}_2$, $\tilde{\zeta}$.

For \eqref{item:momentbnd}, we take $N_0$ as above.  Then for $N > N_0$, we obtain $\E |\hat{\xi}_i|^6 \leq 2^6 \E |\tilde{\xi}_i|^6$.  By the binomial theorem and H\"{o}lder's inequality, we conclude that $\E|\tilde{\xi}_i|^6 \leq 2^6 \E|\xi_i|^6$ for $i=1,2$.  Similarly, we have $\E|\hat{\zeta}|^4 \leq 2^8 \E |\zeta|^4$.  

We now prove \eqref{item:rho}.  We observe that 
$$ |\hat{\rho} - \rho| \leq |\hat{\rho} - \tilde{\rho}| + |\tilde{\rho} - \rho|, $$  
where $\tilde{\rho} := \E[\tilde{\xi}_1 \tilde{\xi}_2]$.  For the first term, we apply the Cauchy--Schwarz inequality and obtain
\begin{align*}
	|\hat{\rho} - \tilde{\rho}| &= \left| \E[ \hat{\xi}_1 \hat{\xi}_2 - \tilde{\xi}_1 \tilde{\xi}_2] \right| \\
		&\leq \sqrt{ \E|\hat{\xi}_1|^2 \E |\hat{\xi}_2|^2 } \left| 1 - \sqrt{\var(\tilde{\xi}_1)} \sqrt{\var(\tilde{\xi}_2)} \right| \\
		&\leq \left| 1 - \var(\tilde{\xi}_1) \var(\tilde{\xi}_2) \right| \\
		&\leq \left| 1- \var(\tilde{\xi}_1) \right| + \var(\tilde{\xi}_1) \left| 1 - \var(\tilde{\xi}_2) \right| \\
		&= o(\eps_N^2 N^{-2})
\end{align*}
by \eqref{item:var1}.  

For the second term, we have
\begin{align*}
	|\tilde{\rho} - \rho| &\leq \sum_{i=1}^2 \left( \sqrt{ \E |\xi_i|^2 \indicator{|\xi_i| > \eps_N \sqrt{N}} } + \E |\xi_i| \indicator{|\xi_i| > \eps_N \sqrt{N}} \right) \\
		&\qquad + \sqrt{ \E |\xi_1|^2 \indicator{|\xi_1| > \eps_N \sqrt{N}} \E |\xi_2|^2 \indicator{|\xi_2| > \eps_N \sqrt{N}} } \\
		&\qquad + \E|\xi_1| \indicator{|\xi_1| > \eps_N \sqrt{N}} \E |\xi_2| \indicator{|\xi_2| > \eps_N \sqrt{N}}
\end{align*}
by the Cauchy--Schwarz inequality.  It now follows from \eqref{eq:xiilim} that $|\tilde{\rho} - \rho| = o(N^{-1})$.  This completes the proof of \eqref{item:rho}.  

The treatment of \eqref{item:highmom} is similar.  Indeed, \eqref{item:highmom} follows from \eqref{item:var1} and the dominated convergence theorem; we omit the details.    
\end{proof}

Define
$$ \tilde{y}_{ij} = y_{ij} \indicator{|y_{ij}| \leq \eps_N \sqrt{N}} - \E y_{ij} \indicator{|y_{ij}| \leq \eps_N \sqrt{N}}. $$
For $i \neq j$, we define
$$ \hat{y}_{ij} = \frac{\tilde{y}_{ij}}{\sqrt{\var(\tilde{y}_{ij})}}, $$
and for the diagonal entries, we define
$$ \tilde{y}_{ii} = \sigma \frac{\tilde{y}_{ii}}{\sqrt{\var(\tilde{y}_{ii})}}. $$
Define the matrices $\tilde{\mat Y}_N = (\tilde{y}_{ij})_{i,j=1}^N$ and $\hat{\mat Y}_{N} = (\hat{y}_{ij})_{i,j=1}^N$ as well as 
$$ \tilde{\mat X}_N := \frac{1}{\sqrt{N}} \tilde{\mat Y}_N, \quad \hat{\mat X}_N := \frac{1}{\sqrt{N}} \hat{\mat Y}_N. $$

Using Lemma \ref{lemma:truncation}, we will verify the following bounds.  

\begin{lemma} \label{lemma:norm}
Under the assumptions of Theorem \ref{thm:main}, 
\begin{equation} \label{eq:fndiff}
	\E \| \mat{X}_N - \hat{\mat X}_N \|_2^2 = o(\eps_N^2 N^{-1})
\end{equation}
and
\begin{equation} \label{eq:probdiff}
	\Prob \left( \| \mat{X}_N - \hat{\mat X}_N \| > \eps_N \right) = o(N^{-1}). 
\end{equation}
\end{lemma}
\begin{proof}
By Markov's inequality, 
\begin{align*}
	\Prob \left( \| \mat{X}_N - \hat{\mat X}_N \| > \eps_N \right) &\leq \frac{1}{\eps_N^2} \E \| \mat{X}_N - \hat{\mat X}_N \|^2 \leq \frac{1}{\eps_N^2} \E \| \mat{X}_N - \hat{\mat X}_N \|_2^2.
\end{align*}
Thus, it suffices to prove \eqref{eq:fndiff}.  

By the triangle inequality, 
$$ \E \| \mat{X}_N - \hat{\mat X}_N \|_2^2 \leq 2 \left( \E\| \mat{X}_N - \tilde{\mat X}_N \|_2^2 + \E \| \tilde{\mat X}_N - \hat{\mat X}_N \|_2^2 \right). $$
We will show that both terms on the right-hand side are $o(\eps_N^2 N^{-1})$.  For the first term, we observe that
\begin{align*} 
	\E\| \mat{X}_N - \tilde{\mat{X}}_N \|_2^2 &= \frac{1}{N} \sum_{i,j=1}^N \E|y_{ij} - \tilde{y}_{ij}|^2 \\
		&\leq \frac{4}{N} \sum_{i,j=1}^N \E |y_{ij}|^2 \indicator{|y_{ij}| > \eps_N \sqrt{N}} \\
		&\leq \frac{4}{\eps_N^4 N} \E|\xi_1|^6 \indicator{|\xi_1| > \eps_N \sqrt{N}} + \frac{4}{\eps_N^4 N} \E|\xi_2|^6 \indicator{|\xi_2| > \eps_N \sqrt{N}} \\
		&\qquad + \frac{4}{\eps_N^2 N} \E |\zeta|^4 \indicator{|\zeta| > \eps_N \sqrt{N}} \\
		&= o(\eps_N^2 N^{-1}).
\end{align*}

Similarly, by Lemma \ref{lemma:truncation}, we obtain
\begin{align*}
	\E \| \tilde{\mat{X}}_N - \hat{\mat{X}}_N \|_2^2 &\leq \frac{1}{N} \sum_{i \neq j} \E |\hat{y}_{ij}|^2 \left| \sqrt{\var(\tilde{y}_{ij})} -1 \right| + \frac{1}{N} \sum_{i=1}^N \E |\hat{y}_{ii}|^2 \left| \frac{\sqrt{\var(\tilde{y}_{ii})}}{\sigma} - 1 \right|  \\
		&\leq \frac{1}{N} \sum_{i \neq j} \E |\hat{y}_{ij}|^2 | {\var(\tilde{y}_{ij})} -1 | + \frac{1}{N} \sum_{i=1}^N \E |\hat{y}_{ii}|^2 \left| \frac{{\var(\tilde{y}_{ii})}}{\sigma^2} - 1 \right|  \\
		&\leq N | \var(\tilde{\xi}_1) - 1| + N |\var(\tilde{\xi}_2) - 1| + | \var(\tilde{\zeta}) - \sigma^2|  \\
		&= o(\eps_N^2 N^{-1}).
\end{align*}
The proof of the lemma is now complete by combining the bounds above.  
\end{proof}

\subsection{Least singular value bound}
We will also need the following result which allows us to control the least singular value of $\mat{X}_N$ given control of the least singular value of the truncated matrix $\hat{\mat X}_N$.  

\begin{lemma} \label{lemma:lsv}
Let $D \subset \mathbb{C}$.  Then, under the assumptions of Theorem \ref{thm:main}, for any $c > 0$, 
$$ \Prob \left( \inf_{z \in D} \sigma_N(\mat{X}_N - z \mat{I}) < c/2 \right) \leq \Prob \left( \inf_{z \in D} \sigma_N(\hat{\mat X}_N - z \mat{I}) < c \right) + o(N^{-1}). $$
\end{lemma}
\begin{proof}
We observe that
\begin{align*}
	&\Prob \left( \inf_{z \in D} \sigma_N(\mat{X}_N - z \mat{I}) < c/2 \right)  \\
	&\quad = \Prob \left( \inf_{z \in D} \sigma_N(\mat{X}_N - z \mat{I}) < c/2 \text{ and } \|\mat{X}_N - \hat{\mat X}_N \| \leq \eps_N \right)  \\
	&\qquad + \Prob \left( \inf_{z \in D} \sigma_N(\mat{X}_N - z \mat{I}) < c/2 \text { and }  \|\mat{X}_N - \hat{\mat X}_N \| > \eps_N\right) \\
	&\quad \leq \Prob \left( \inf_{z \in D} \sigma_N(\mat{X}_N - z \mat{I}) < c/2 \text{ and } \|\mat{X}_N - \hat{\mat X}_N \| \leq \eps_N \right) + o(N^{-1})
\end{align*}
by Lemma \ref{lemma:norm}.

Now, if there exists $z_o \in D$ such that $\sigma_N(\mat{X}_N - z_0\mat{I}) < c/2$ and $\|\mat{X}_N - \hat{\mat X}_N \| \leq \eps_N < c/2$, then by Weyl's inequality $\sigma_N(\hat{\mat X}_N - z_0 \mat{I}) < c$.  Thus, we conclude that, for $N$ sufficiently large (so that $\eps_N < c/2$), 
$$ \Prob \left( \inf_{z \in D} \sigma_N(\mat{X}_N - z \mat{I}) < c/2 \text{ and } \|\mat{X}_N - \hat{\mat X}_N \| \leq \eps_N \right) \leq \Prob \left( \inf_{z \in D} \sigma_N(\hat{\mat X}_N - z \mat{I}) < c \right), $$
and the proof is complete. 
\end{proof}

We will need the following bound on the least singular value of $\hat{\mat X}_N$.  

\begin{theorem} \label{thm:lsv}
Under the assumptions of Theorem \ref{thm:main}, for any $\delta > 0$, there exists $c > 0$ such that the event
$$ \left\{ \inf_{\dist(z,\mathcal{E}_{\rho}) \geq \delta} \sigma_N(\hat{\mat X}_N - z \mat{I}) \geq c \right\} $$
holds with overwhelming probability. 
\end{theorem}

The proof of Theorem \ref{thm:lsv} follows the arguments from \cite{OR}; we present the proof in Appendix \ref{sec:lsvproof}. 

From Theorem \ref{thm:lsv} and Lemma \ref{lemma:lsv}, we immediately obtain the following analog of Theorem \ref{thm:lsv} for the event $E_N$ (defined in \eqref{eq:def:EN}).  
\begin{corollary} \label{cor:ENC}
Under the assumptions of Theorem \ref{thm:main}, for any $\delta > 0$, there exists $c > 0$ such that $\Prob(E_N^C) = o(N^{-1})$.  
\end{corollary}

\section{Proof of Theorem \ref{thm:main}} \label{sec:proof}

Define the resolvents
$$ \mat{G}_N(z) := (\mat{X}_N -  z \mat{I})^{-1}, \quad \hat{\mat{G}}_N(z) := (\hat{\mat{X}}_N - z \mat{I})^{-1} $$
for $z \in \mathbb{C}$.  Also, for any $1 \leq k \leq N$, let $\hat{\mat{X}}_{N,k}$ be the matrix $\hat{\mat{X}}_N$ with the $k$-th row and $k$-th column removed.  We will index the entries of $\hat{\mat{X}}_{N,k}$ by the set $\{1,\ldots, k-1, k+1, \ldots, N\}$.  Define 
$$ \hat{\mat{G}}_{N,k}(z) := (\hat{\mat{X}}_{N,k} - z \mat{I})^{-1}. $$
Again, we index the entries of $\hat{\mat{G}}_{N,k}(z)$ by the set $\{1,\ldots, k-1, k+1, \ldots, N\}$.

Fix $\delta > 0$.  Let $\mathcal{C}$ be the contour along the boundary of $\mathcal{E}_{\rho,\delta}$.  We will mostly work on the events 
$$ \hat{E}_N := \left\{ \inf_{z \in \mathcal{C}} \sigma_N(\hat{\mat{X}}_N - z \mat{I}) \geq c \right\} $$
and
$$ \hat{\Omega}_N := \bigcap_{k=1}^N \hat{\Omega}_{N,k} \bigcap \hat{E}_N, $$
where
$$ \hat{\Omega}_{N,k} := \left\{ \inf_{z \in \mathcal{C}} \sigma_{N-1}(\hat{\mat{X}}_{N,k} - z \mat{I}) \geq c \right\}. $$

By Theorem \ref{thm:lsv} (applied to $\hat{\mat X}_N$ as well as $\hat{\mat X}_{N-1}$), the union bound, and Corollary \ref{cor:ENC}, there exists $c > 0$ such that
\begin{equation} \label{eq:hatONC}
	\hat{\Omega}_N \text{ holds with overwhelming probability}
\end{equation}
and 
\begin{equation} \label{eq:ENC}
	\Prob(E_N^C) = o(N^{-1}). 
\end{equation}
Here we used the fact that each $\hat{\mat X}_{N,k}$ has the same distribution as $\hat{\mat X}_{N-1}$.  Clearly, \eqref{eq:hatONC} implies that $\hat{E}_N$ holds with overwhelming probability.  

By the Cram\'{e}r-Wold device, in order to prove Theorem \ref{thm:main}, it suffices to show that 
$$ \tr f(\mat{X}_N) \oindicator{E_N} - \E \tr f(\mat{X}_N) \oindicator{E_N} $$
converges in distribution to an appropriate mean-zero Gaussian random variable, where $f$ is a linear combination of the functions $f_1, \ldots, f_k$.  In particular, any such $f$ is analytic in a neighborhood containing $\mathcal{E}_{\rho,\delta}$.  

Let $f$ be a linear combination of $f_1, \ldots, f_k$.  In view of Remark \ref{rem:singeig}, all eigenvalues of $\mat{X}_N$ are contained on the interior of $\mathcal{E}_{\rho, \delta}$ on the event $E_N$.  By Cauchy's integral formula, we have
$$ \tr f(\mat{X}_N) \oindicator{E_N} - \E \tr f(\mat{X}_N) \oindicator{E_N} = \frac{-1}{2 \pi i} \oint_{\mathcal{C}} f(z) \left( \tr \mat{G}_N(z) \oindicator{E_N} - \E \tr \mat{G}_N(z) \oindicator{E_N} \right) dz. $$
Therefore, the proof of Theorem \ref{thm:main} reduces to showing that
$$ \tr \mat{G}_N(z) \oindicator{E_N} - \E \tr \mat{G}_N(z) \oindicator{E_N} $$
converges weakly to an appropriate Gaussian process in the space of continuous functions on the contour $\mathcal{C}$.

We now reduce to the case where we only need to consider the truncated resolvent $\hat{\mat{G}}_N(z)$.  Indeed, we observe, by the resolvent identity, that
\begin{align*}
	\sup_{z \in \mathcal{C}} \E &| \tr \mat{G}_N(z) \oindicator{E_N} - \tr \hat{ \mat{G}}_N(z) \oindicator{\hat{\Omega}_N} | \\
		&\leq \frac{2N}{c} \left( \Prob(E_N^C) + \Prob(\hat{\Omega}_N^C) \right) + \sup_{z \in \mathcal{C}} \E| \tr \mat{G}_N(z) \oindicator{E_N} (\mat{X}_N - \hat{\mat{X}}_N) \hat{\mat{G}}_N(z) \oindicator{\hat{\Omega}_N} | \\
		&\leq \frac{2N}{c} \left( \Prob(E_N^C) + \Prob(\hat{\Omega}_N^C) \right) + \frac{2}{c^2} \E \|\mat{X}_N - \hat{\mat{X}}_N \|_2.
\end{align*}
From \eqref{eq:hatONC} and \eqref{eq:ENC}, we find that the first term is $o(1)$.  By Lemma \ref{lemma:norm}, we observe that the second term is also $o(1)$.  Therefore, the problem reduces to verifying the weak convergence of
$$ \hat{\Xi}_N(z) := \tr \hat{\mat{G}}_N(z) \oindicator{\hat{\Omega}_N} - \E \tr \hat{\mat{G}}_N(z) \oindicator{\hat{\Omega}_N} $$
in the space of continuous functions on the contour $\mathcal{C}$.  

We proceed as follows.  In Section \ref{sec:distr}, we show the convergence of the finite dimensional distributions.  In Section \ref{sec:tight}, we prove that $\hat{\Xi}_N$ is tight in the space of continuos functions on the contour $\mathcal{C}$.  Finally, we complete the proof of Theorem \ref{thm:main} in Section \ref{sec:conclusion}.  Since we will only work with the truncated process, we adjust our notation as follows.  We will no longer include the superscript decorations.  Instead, we will simply write $\mat{X}_N, \mat{X}_{N,k}, \mat{G}_N, \mat{G}_{N,k}, \Omega_N, \Xi_N$, etc.

\section{Finite dimensional distributions} \label{sec:distr}

This section is devoted to studying the convergence of the finite dimensional distributions of $\Xi_N$.  We will make use of the following standard central limit theorem for martingale difference sequences. 

\begin{theorem}[Theorem 35.12 of \cite{PB}] \label{thm:clt}
For each $N$, suppose $Z_{N1}, Z_{N2}, \ldots, Z_{Nr_N}$ is a real martingale difference sequence with respect to the increasing 
$\sigma$-field $\{\mathcal{F}_{N,j}\}$ having second moments.  Suppose, for any $\eta > 0$ and a positive constant $v^2$,
\begin{equation} \label{eq:clt_variance}
	\lim_{N \to \infty} \Prob \left( \left| \sum_{j=1}^{r_N} \E(Z^2_{Nj} \mid \mathcal{F}_{N,j-1}) - v^2 \right| > \eta \right) = 0,
\end{equation}
and
\begin{equation} \label{eq:clt_lf}
	\lim_{N \to \infty} \sum_{j=1}^{r_N} \E(Z^2_{Nj} \indicator{|Z_{Nj}| \geq \eta}) = 0.
\end{equation}
Then as $N \to \infty$, the distribution of $\sum_{j=1}^{r_N} Z_{Nj}$ converges weakly to a Gaussian distribution with mean zero and variance $v^2$.  
\end{theorem}

We rewrite
$$\Xi_N(z) = \sum_{k=1}^N (\E_k - \E_{k-1}) \tr \mat{G}_N(z) \oindicator{\Omega_N} =: \sum_{k=1}^N Z_{N,k}(z), $$
where $\E_k$ denotes conditional expectation with respect to the $\sigma$-algebra generated by $\{y_{ij} : i,j \leq k \}$.  Here we use the convention that $\E_0$ denotes conditional expectation with respect to the trivial $\sigma$-algebra, and hence $\E_0 = \E$.  By the Cram\'{e}r-Wold device, it suffices to verify the conditions of Theorem \ref{thm:clt} for martingale difference sums of the form
$$ \mathcal{M}_N := \sum_{k=1}^N \sum_{l=1}^L \left( \alpha_l Z_{N,k}(z_l) + \beta_l \overline{ Z_{N,k}(z_l)} \right) =: \sum_{k=1}^N \mathcal{M}_{N,k} $$
for any fixed $L$, any choice of $z_l \in \mathcal{C}$, and any $\alpha_l, \beta_l \in \mathbb{C}$ such that $\mathcal{M}_N$ is real.

The goal of this section is to prove the following.  

\begin{theorem} \label{thm:finitedim}
Under the assumptions of Theorem \ref{thm:main}, $\mathcal{M}_N$ converges weakly to a Gaussian random variable with mean zero and variance
\begin{align*}
	\sum_{l_1, l_2=1}^L \left( \alpha_{l_1} \alpha_{l_2} \upsilon(z_{l_1}, z_{l_2}) + \alpha_{l_1} \beta_{l_2} \upsilon(z_{l_1}, \overline{z}_{l_2}) + \beta_{l_1} \alpha_{l_2} \upsilon(\overline{z}_{l_1}, z_{l_2}) + \beta_{l_1} \beta_{l_2} \upsilon(\overline{z}_{l_1}, \overline{z}_{l_2}) \right),
\end{align*}
where $\upsilon(z,w)$ is defined in \eqref{def:upsilon}.
\end{theorem}

We now begin the proof of Theorem \ref{thm:finitedim}.  The proof is similar to the proofs given in \cite{BSbook, S} for Wigner matrices.  However, in this case, the eigenvalues of $\mat{Y}_N$ can be complex and special care must be taken to work on events in which the resolvent can be adequately bounded.  We will show that these events occur with sufficiently high probability.

\subsection{A simple reduction}
For each $1 \leq k \leq N$, let $r_k$ be $k$-th row of $\mat{X}_N$ with the $k$-th entry removed, and let $c_k$ be the $k$-th column of $\mat{X}_N$ with the $k$-th entry removed.    

\begin{lemma} \label{lemma:reduction}
For any $\eta > 0$, 
$$ \lim_{N \to \infty} \Prob( | \mathcal{M}_N - \mathcal{M}_N'| > \eta ) = 0, $$
where $\mathcal{M}_N'$ is defined by replacing each appearance of $Z_{N,k}(z)$ in $\mathcal{M}_{N,k}$ with 
$$ Z_{N,k}'(z) := (\E_k - \E_{k-1}) \frac{1 + r_k \mat{G}_{N,k}^2(z) c_k}{\frac{1}{\sqrt{N}} y_{kk} - z - r_k \mat{G}_{N,k}(z) c_k} \oindicator{\Omega_N}. $$
\end{lemma}
\begin{proof}
We begin by observing that
$$ \Xi_N(z) = \sum_{k=1}^N (\E_k - \E_{k-1}) \left(  \tr \mat{G}_N(z) \oindicator{\Omega_N} - \tr \mat{G}_{N,k}(z) \oindicator{\Omega_{N,k}} \right). $$
Moreover, from \eqref{eq:hatONC}, we have
\begin{align*}
	\sum_{k=1}^N \E \left|  \tr \mat{G}_{N,k}(z) \oindicator{\Omega_{N,k}} -  \tr \mat{G}_{N,k}(z) \oindicator{\Omega_N} \right| &\leq \frac{N}{c} \sum_{k=1}^N \Prob ( \Omega_{N,k} \setminus \Omega_N) \\
		&\leq \frac{N^2}{c} \Prob(\Omega_{N}^C) = o(1).
\end{align*}
Thus, it suffices to show that on the event $\Omega_N$, 
\begin{equation} \label{eq:trident}
	\tr \mat{G}_N(z) - \tr \mat{G}_{N,k}(z) = \frac{1 + r_k \mat{G}_{N,k}^2(z) c_k}{\frac{1}{\sqrt{N}} y_{kk} - z - r_k \mat{G}_{N,k}(z) c_k}.
\end{equation}
We note that on $\Omega_N$, the matrices $\mat{X}_N - z \mat{I}$ and $\mat{X}_{N,k} - z \mat{I}$ are invertible.  Moreover, from \cite[Theorem A.4]{BSbook}, the $(k,k)$-entry of $\mat{G}_{N}(z)$ is given by 
$$ \frac{1}{\frac{1}{\sqrt{N}} y_{kk} - z - r_k \mat{G}_{N,k}(z) c_k}. $$
Therefore, \eqref{eq:trident} will follow from Proposition \ref{prop:trace} below.
\end{proof}

\begin{proposition} \label{prop:trace}
Let $\mat{A}_N$ be an $N \times N$ matrix of the form 
$$ \mat{A}_N = \begin{pmatrix} a & r \\ c & \mat{A}_{N-1} \end{pmatrix}, $$
where $\mat{A}_{N-1}$ is a $(N-1) \times (N-1)$ major submatrix, $a \in \mathbb{C}$, $r$ is a row vector, and $c$ is a column vector.  Then 
$$ \tr { \mat{A}_N} - \tr \mat{A}_{N-1} = \frac{1 + r \mat{A}_{N-1}^{-2} c}{a - r \mat{A}_{N-1}^{-1} c} $$
provided $\mat{A}_N$ and $\mat{A}_{N-1}$ are invertible and $a - r \mat{A}_{N-1}^{-1} c \neq 0$.  
\end{proposition}

The proof of Proposition \ref{prop:trace} can be found in \cite[Section 0.7]{HJ2}. 

\subsection{Preliminary tools}
We will eventually proceed with another reduction.  However, first we collect some preliminary tools we will need throughout this section.  

We write
$$ \frac{1}{\sqrt{N}} y_{kk} - z - r_k \mat{G}_{N,k}(z) c_k = - z - \frac{\rho}{N} \tr \mat{G}_{N,k} + \alpha_{N,k}(z), $$
where 
\begin{equation} \label{def:alpha}
	\alpha_{N,k}(z) := \frac{1}{\sqrt{N}}y_{kk} - \left(r_k \mat{G}_{N,k}(z) c_k - \frac{\rho}{N} \tr \mat{G}_{N,k}(z) \right). 
\end{equation}
We verify the following concentration results.  

\begin{lemma} \label{lemma:alpha}
Let $D \subset \mathbb{C}$.  Define the event
$$ \Gamma_{N,k} := \left\{ \inf_{z \in D} \sigma_{N-1}(X_{N,k} - z I) \geq \gamma \right\} $$
for some constant $\gamma > 0$. Then
\begin{enumerate}
\item
for any $z,w \in D$ and $p\geq 2$, there exists a constant $K_p > 0$ (depending only on $p$) such that
 \begin{align}
 \E&|\alpha_{N,k}(z)|^p \oindicator{\Gamma_{N,k}} \nonumber \\
 &\leq K_p \left( \frac{\E|\zeta|^p}{N^{p/2}} + \frac{\left( \max \{ \E|\xi_1|^4, \E|\xi_2|^4 \} \right)^{p/2}}{\gamma^{p} N^{p/2}}  + \frac{\max \{ \E|\xi|^{2p}, \E|\xi_2|^{2p} \}}{\gamma^p N^{p-1}}  \right)  \label{eq:alpha}
 \end{align}
and 
\begin{align}
\label{eq:GNk2}
	\E&\left| r_k \mat{G}_{N,k}(z)\mat{G}_{N,k}(w) c_k - \frac{\rho}{N} \tr \mat{G}_{N,k}(z)\mat{G}_{N,k}(w) \right|^p \oindicator{\Gamma_{N,k}} \\
	&\qquad \leq K_p \left( \frac{\left( \max \{ \E|\xi_1|^4, \E|\xi_2|^4 \} \right)^{p/2}}{\gamma^{2p} N^{p/2}}  + \frac{\max \{ \E|\xi|^{2p}, \E|\xi_2|^{2p} \}}{\gamma^{2p} N^{p-1}}  \right). \notag
\end{align}
\item if $D$ is compact, the event 
\begin{equation}
\label{eq:alphaprob}
 \left\{ \sup_{z \in D} |\alpha_{N,k}(z)| \oindicator{\Gamma_{N,k}} \leq \frac{1}{\log N} \right\} 
 \end{equation}
holds with overwhelming probability.  
\item for $N \geq \exp(2/\gamma)$ and $D$ compact,
\begin{align}
\label{eq:interprob}
	\Prob&\left( \left\{ \inf_{z \in D} \left| z + \frac{\rho}{N} \tr \mat{G}_{N,k}(z) \right| < \frac{\gamma}{2} \right\} \bigcap \Gamma_{N,k} \right) \\
	&\qquad\leq \Prob\left( \sup_{z \in D} \| G_{N}(z) \| > \gamma^{-1}\right) + \Prob \left( \sup_{z \in D} |\alpha_{N,k}(z)| \oindicator{\Gamma_{N,k}} > \frac{1} {\log N} \right). \nonumber
\end{align}
\end{enumerate}

\end{lemma}

\begin{remark}
In most cases, we will apply Lemma \ref{lemma:alpha} by taking $D = \mathcal{C}$ and $\gamma = c$.  In this case, $\Gamma_{N,k} = \Omega_{N,k}$.   
\end{remark}

In order to prove Lemma \ref{lemma:alpha}, we will need to make use of the following concentration result from \cite{OR}.  
\begin{lemma}[Concentration of bilinear forms] \label{lemma:lde}
Let $(x_1,y_1), (x_2,y_2), \ldots, (x_N,y_N)$ be iid random vectors in $\mathbb{C}^2$ such that
$$ \E[x_1] = \E[y_1] = 0, \quad \E|x_1|^2 = \E|y_1|^2 = 1, \quad \E[\bar{x}_1 y_1] = \rho. $$
Let $\mu_p = \max\{ \E|x_1|^p, \E|y_1|^p \}$ for $p \geq 4$.  Let $\mat{B}=(b_{ij})$ be a deterministic complex $N \times N$ matrix and write $X = (x_1, x_2, \ldots, x_N)^\mathrm{T}$ and $Y = (y_1, y_2, \ldots, y_N)^\mathrm{T}$.  Then, for any $p \geq 2$, 
\[ \E \left| X^\ast \mat{B} Y - \rho \tr \mat{B} \right|^p \leq K_p \left( ( \mu_4 \tr (\mat{B} \mat{B}^\ast) )^{p/2} + \mu_{2p} \tr (\mat{B} \mat{B}^\ast)^{p/2} \right). \]
\end{lemma}

\begin{proof}[Proof of Lemma \ref{lemma:alpha}]
To prove \eqref{eq:alpha}, we apply the triangle inequality, to get
$$ \E|\alpha_{N,k}(z)|^p \oindicator{\Gamma_{N,k}} \leq K_p \left( \frac{\E|\zeta|^p}{N^{p/2}} + \E\left| r_k \mat{G}_{N,k}(z) c_k - \frac{\rho}{N} \tr \mat{G}_{N,k}(z) \right|^p\oindicator{\Gamma_{N,k}} \right). $$
As $(r_k, c_k)$ is independent of $\mat{G}_{N,k}(z)$ and $\Gamma_{N,k}$, we apply Lemma \ref{lemma:lde} to the second term and obtain
\begin{align*}
	\E&\left| r_k \mat{G}_{N,k}(z) c_k - \frac{\rho}{N} \tr \mat{G}_{N,k}(z) \right|^p \oindicator{\Gamma_{N,k}} \\
	&\leq K_p \bigg[ \E\left( \frac{\max\{\E|\xi_1|^4, \E|\xi_2|^4 \}}{N^2} \tr \left( \mat{G}_{N,k}(z) \mat{G}_{N,k}(z)^\ast \right) \oindicator{\Gamma_{N,k}} \right)^{p/2} \\
	& \qquad\qquad + \frac{\max\{\E|\xi_1|^{2p}, \E|\xi_2|^{2p}\}}{N^p} \E \tr \left(\mat{G}_{N,k}(z) \mat{G}_{N,k}(z)^\ast \right)^{p/2} \oindicator{\Gamma_{N,k}} \bigg].
\end{align*}
Since
$$ \tr \left(\mat{G}_{N,k}(z) \mat{G}_{N,k}(z)^\ast \right)^{p/2} \leq N \|\mat{G}_{N,k}(z) \|^p \leq N \gamma^{-p} $$
on the event $\Gamma_{N,k}$, the proof of \eqref{eq:alpha} is complete.  
The proof of \eqref{eq:GNk2} is nearly identical to the proof of \eqref{eq:alpha}; we omit the details.  

To prove \eqref{eq:alphaprob}, we observe that, for $z,z' \in D$, we have the deterministic bound
\begin{align*}
	\left| \alpha_{N,k}(z) - \alpha_{N,k}(z') \right| \oindicator{\Gamma_{N,k}} &\leq \|r_k \| \|c_k\| \| \mat{G}_{N,k}(z) - \mat{G}_{N,k}(z') \| \oindicator{\Gamma_{N,k}} \\
	&\qquad + \left| \frac{1}{N} \tr \mat{G}_{N,k}(z) - \frac{1}{N} \tr \mat{G}_{N,k}(z') \right| \oindicator{\Gamma_{N,k}} \\
	&\leq 16 N \eps_N^2 \frac{|z-z'|}{c^2} + \frac{|z-z'|}{c^2}
\end{align*} 
by the resolvent identity and Lemma \ref{lemma:truncation}.  Let $\mathcal{N}$ be a $N^{-2}$-net of $D$.  Then $|\mathcal{N}| = O_D(N^4)$ (see for example \cite[Lemma 3.11]{OR}).  Thus, we have, for $N$ sufficiently large and for any $p \geq 2$, 
\begin{align*}
	\Prob \left( \sup_{z \in D} |\alpha_{N,k}(z)| \oindicator{\Gamma_{N,k}} > \frac{1}{\log N} \right) &\leq \sum_{z \in \mathcal{N}} \Prob \left( |\alpha_{N,k}(z)| \oindicator{\Gamma_{N,k}} > \frac{1}{2 \log N} \right) \\
	&\leq (2\log N)^p \sum_{z \in \mathcal{N}} \E|\alpha_{N,k}(z)|^p \oindicator{\Gamma_{N,k}}.
\end{align*}
The claim now follows from \eqref{eq:alpha} by taking $p$ sufficiently large.  (Here we use the bounds in Lemma \ref{lemma:truncation} to control the higher moments of $\zeta, \xi_1, \xi_2$ as well as the fact that $\eps_N := N^{-\eps}$, for some $\eps > 0$.)

Finally, to prove \eqref{eq:interprob}, we note that
\begin{align*}
	&\Prob\left( \left\{ \inf_{z \in D} \left| z + \frac{\rho}{N} \tr \mat{G}_{N,k}(z) \right| < \frac{\gamma}{2} \right\} \bigcap \Gamma_{N,k} \right) \\
	&\qquad\leq \Prob\left( \left\{ \inf_{z \in D} \left| z + \frac{\rho}{N} \tr \mat{G}_{N,k}(z) \right| < \frac{\gamma}{2} \right\} \bigcap \Gamma_{N,k} \bigcap \Delta_{N,k} \right) \\
	&\qquad\qquad + \Prob\left( \sup_{z \in D} \| G_{N}(z) \| > \gamma^{-1}\right) + \Prob \left( \sup_{z \in D} |\alpha_{N,k}(z)| \oindicator{\Gamma_{N,k}} > \frac{1}{\log N} \right),
\end{align*}
where
$$ \Delta_{N,k} := \left\{ \sup_{z \in D} \| G_{N}(z) \| \leq \gamma^{-1}\ \right\} \bigcap \left\{ \sup_{z \in D} |\alpha_{N,k}(z)| \oindicator{\Gamma_{N,k}} \leq \frac{1}{\log N}  \right\}. $$
It now suffices to show that the event
\begin{equation} \label{eq:eventcap}
	\left\{ \inf_{z \in D} \left| z + \frac{\rho}{N} \tr \mat{G}_{N,k}(z) \right| < \frac{\gamma}{2} \right\} \bigcap \Gamma_{N,k} \bigcap \Delta_{N,k} 
\end{equation}
is empty.  Indeed, if there exists $z_0 \in D$ such that 
$$ \left|z_0 + \frac{\rho}{N} \tr \mat{G}_{N,k}(z_0) \right| < \frac{\gamma}{2} \quad \text{and} \quad |\alpha_{N,k}(z_0)| \leq \frac{1}{\log N} \leq \frac{\gamma}{2}, $$
then, by definition of $\alpha_{N,k}(z_0)$, it follows that
$$ \left| \frac{1}{\sqrt{N}} y_{kk} - z_0 - r_k \mat{G}_{N,k}(z_0) c_k\right| < \gamma. $$
However, this implies that $\| \mat{G}_{N}(z_0) \| \geq | (\mat{G}_{N}(z_0))_{kk} | > \gamma^{-1}$, which cannot happen on the event \eqref{eq:eventcap}.  The proof of the lemma is complete.   
\end{proof}

Define the events
$$ Q_{N,k} := \left\{ \inf_{z \in \mathcal{C}} \left|z + \frac{\rho}{N} \tr \mat{G}_{N,k}(z) \right| \geq \frac{c}{2} \right\} $$
for $1 \leq k \leq N$, and the event
$$ Q_N := \bigcap_{k=1}^N Q_{N,k}. $$

\begin{corollary} \label{cor:QN}
Under the assumptions of Theorem \ref{thm:main}, the event $Q_N$ holds with overwhelming probability.  
\end{corollary}
\begin{proof}
We note that each $\mat{G}_{N,k}(z)$, $1 \leq k \leq N$ has the same distribution.  Thus, by the union bound, it suffices to show that each $Q_{N,k}$ holds with overwhelming probability.  Fix $ 1\leq k \leq N$.  Taking $D = \mathcal{C}$, we find that the event in \eqref{eq:alphaprob} holds with overwhelming probability.  As $E_N$ and $\Omega_{N,k}$ both hold with overwhelming probability, the conclusion follows from \eqref{eq:interprob}.  
\end{proof}

We show that the diagonal entries of $\mat{G}_{N,k}(z)$ converge to $m(z)$.  

\begin{lemma} \label{lemma:diagonal}
Under the assumptions of Theorem \ref{thm:main}, one has
\begin{align}
	\sup_{z \in \mathcal{C}} \sup_{1 \leq k \leq N} \sup_{i \neq k} \E \left| \left( \mat{G}_{N,k}(z) \right)_{ii} \oindicator{\Omega_{N,k} \cap Q_{N,k}} - m(z) \right|^2 &= o(1), \label{eq:Giim} \\
	\sup_{z \in \mathcal{C}} \sup_{1 \leq k \leq N} \sup_{i \neq k} \E \left| \left( \mat{G}_{N,k}(z) \right)_{ii} - m(z) \right|^2 \oindicator{\Omega_{N,k} \cap Q_{N,k}} &= o(1), \label{eq:Giim2} \\
	\sup_{z \in \mathcal{C}} \sup_{1 \leq k \leq N} \E \left| \left(z + \frac{\rho}{N} \tr \mat{G}_{N,k}(z) \right)^{-1} \oindicator{\Omega_{N,k} \cap Q_{N,k}} + m(z) \right|^2 &= o(1), \label{eq:zm} \\
	\sup_{z \in \mathcal{C}} \sup_{1 \leq k \leq N} \E \left| \left(z + \frac{\rho}{N} \tr \mat{G}_{N,k}(z) \right)^{-1} + m(z) \right|^2 \oindicator{\Omega_{N,k} \cap Q_{N,k}} &= o(1). \label{eq:zm2}
\end{align}
\end{lemma}

The proof of Lemma \ref{lemma:diagonal} is based on the arguments in \cite{OR}.  In fact, a version of Lemma \ref{lemma:diagonal} also appears in \cite{OR} (see \cite[Lemma 7.1]{OR}).  

\begin{proof}[Proof of Lemma \ref{lemma:diagonal}]
As $\sup_{z \in \mathcal{C}} |m(z)| < \infty$ and $\cap_{k=1}^N \Omega_{N,k} \cap Q_{N,k}$ holds with overwhelming probability, it suffices to verify \eqref{eq:Giim} and \eqref{eq:zm}.  Indeed, \eqref{eq:Giim2} follows from \eqref{eq:Giim}, and \eqref{eq:zm2} follows from \eqref{eq:zm}.  

Since
$$ \sup_{z \in \mathcal{C}} \sup_{1 \leq k \leq N} \sup_{i \neq k} \left| \left(\mat{G}_{N,k}(z) \right)_{ii} \right| \oindicator{\Omega_{N,k} \cap Q_{N,k}} \leq \frac{1}{c}, $$
in order to verify \eqref{eq:Giim}, it suffices to show that 
\begin{equation} \label{eq:diagonalsuffice1}
	\sup_{z \in \mathcal{C}} \sup_{1 \leq k \leq N} \sup_{i \neq k} \left| \left( \mat{G}_{N,k}(z) \right)_{ii} \oindicator{\Omega_{N,k} \cap Q_{N,k}} - m(z) \right| \leq \frac{C}{ \log N }
\end{equation}
with overwhelming probability, where $C>0$ depends only on $c$.  Similarly, since 
$$ \sup_{z \in \mathcal{C}} \sup_{1 \leq k \leq N} \sup_{i \neq k} \left| \left(z + \frac{\rho}{N} \tr \mat{G}_{N,k}(z) \right)^{-1} \oindicator{\Omega_{N,k} \cap Q_{N,k}} \right| \leq \frac{2}{c}, $$
in order to prove \eqref{eq:zm}, it suffices to show that
\begin{equation} \label{eq:diagonalsuffice2}
	\sup_{z \in \mathcal{C}} \sup_{1 \leq k \leq N} \left| \left(z + \frac{\rho}{N} \tr \mat{G}_{N,k}(z) \right)^{-1} \oindicator{\Omega_{N,k} \cap Q_{N,k}} + m(z) \right| \leq \frac{C}{\log N}
\end{equation}
with overwhelming probability.  

We first verify \eqref{eq:diagonalsuffice1}.  Let $D \subset \mathbb{C}$ be a connected, compact set outside $\mathcal{E}_{\rho, \delta}$ which contains the contour $\mathcal{C}$.  In addition, assume there exists $z_0 \in D$ with $|z_0| > C_0$ for some sufficiently large $C_0 > 0$ to be chosen later.  Define the events 
$$ \Gamma_{N,k} := \left\{ \inf_{z \in D} \sigma_{N-1} (\mat{X}_{N,k} - z \mat{I}) \geq c \right\} $$
and
$$ \Gamma_{N} := \bigcap_{k=1}^N \Gamma_{N,k} \bigcap \left\{ \inf_{z \in D} \sigma_{N} (\mat{X}_N - z \mat{I} ) \geq c \right\}. $$
By Theorem \ref{thm:lsv} and the union bound, $\Gamma_N$ holds with overwhelming probability.  Thus, in order to verify \eqref{eq:diagonalsuffice1}, it suffices to show that
$$ \sup_{z \in D} \sup_{1 \leq k \leq N} \sup_{i \neq k} \left| \left( \mat{G}_{N,k}(z) \right)_{ii} \oindicator{\Gamma_{N,k}} - m(z) \right| \leq \frac{C}{\log N } $$
with overwhelming probability.  As $\mat{G}_{N,k}(z)$ has the same distribution as $\mat{G}_{N-1}(z)$, it suffices to show that 
\begin{equation} \label{eq:diagonalsuffice11}
	\sup_{z \in D} \sup_{1 \leq k \leq N} \left| \left( \mat{G}_{N}(z) \right)_{kk} \oindicator{\Gamma_{N}} - m(z) \right| \leq \frac{C}{\log N }
\end{equation}
with overwhelming probability.  

Define the events 
$$ \Delta_{N,k} := \left\{ \sup_{z \in D} |\alpha_{N,k}| \oindicator{\Gamma_{N,k}} \leq \frac{1}{ \log N } \right\}. $$
By \eqref{eq:alphaprob} and the union bound, $\cap_{k=1}^N \Delta_{N,k}$ holds with overwhelming probability.  In addition, by the resolvent identity\footnote{Technically, one cannot compare $\mat{G}_{N,k}(z)$ with $\mat{G}_N(z)$ using the resolvent identity since these matrices are different sizes.  However, one can compare $\mat{G}_{N}(z)$ with $\breve{\mat G}_{N,k}(z)$ using the resolvent identity, where $\breve{\mat G}_{N,k}(z) := (\breve{\mat X}_{N,k} - z \mat{I})^{-1}$ and $\breve{\mat X}_{N,k}$ is formed from $\mat{X}_{N}$ by replacing the $k$-th row and $k$-th column with zeros.  Thus, $\breve{\mat G}_{N,k}(z)- \mat{G}_{N}(z)$ is at most rank $2$.  In addition, one can compare $\frac{\rho}{N} \tr \breve{\mat G}_{N,k}(z)$ with $\frac{\rho}{N} \tr \mat{G}_{N,k}(z)$ since all eigenvalues of $\mat{X}_{N,k}$ are also eigenvalues of $\breve{\mat{X}}_{N,k}$.} and the almost sure bounds 
$$ \sup_{1 \leq k \leq N} \| r_k \| \leq 4 \eps_N \sqrt{N}, \quad \sup_{1 \leq k \leq N} \| c_k \| \leq 4 \eps_N \sqrt{N}, $$
we find that
$$ \sup_{z \in D} \sup_{1 \leq k \leq N} \left| \frac{\rho}{N} \tr \mat{G}_{N,k}(z) - \frac{\rho}{N} \tr \mat{G}_{N}(z) \right|\oindicator{\Gamma_N} \leq \frac{C}{\sqrt{N}} $$
for some constant $C > 0$ depending only on $c$.

Combining the bounds above, we obtain 
$$ \sup_{z \in D} \sup_{1 \leq k \leq N} \left| \frac{1}{\sqrt{N}} y_{kk} \oindicator{\Gamma_{N} \cap \Delta_{N,k}} - \left( r_k \mat{G}_{N,k}(z) c_k - \frac{\rho}{N} \tr \mat{G}_N(z) \right) \oindicator{\Gamma_{N} \cap \Delta_{N,k}} \right| \leq \frac{C}{ \log N}. $$
By the Schur complement of a matrix, we have
$$ \left( \mat{G}_{N}(z) \right)_{kk} \oindicator{\Gamma_{N} \cap \Delta_{N,k}} = \frac{\oindicator{\Gamma_{N} \cap \Delta_{N,k}}}{ \frac{1}{\sqrt{N}} y_{kk} \oindicator{\Gamma_{N} \cap \Delta_{N,k}} - z - r_k \mat{G}_{N,k}(z) c_k \oindicator{\Gamma_{N} \cap \Delta_{N,k}} }. $$
Thus, by the trivial bound 
$$ \sup_{z \in D} \sup_{1 \leq k \leq N} |\left( \mat{G}_{N}(z) \right)_{kk}| \oindicator{\Gamma_{N} \cap \Delta_{N,k}} \leq \frac{1}{c}, $$ 
we find that
$$ \inf_{z \in D} \inf_{1 \leq k \leq N} \left|  \frac{1}{\sqrt{N}} y_{kk} \oindicator{\Gamma_{N} \cap \Delta_{N,k}} - z - r_k \mat{G}_{N,k}(z) c_k \oindicator{\Gamma_{N} \cap \Delta_{N,k}} \right| \geq c, $$
and hence 
$$ \inf_{z \in D} \inf_{1 \leq k \leq N} \left| z + \frac{\rho}{N} \tr \mat{G}(z) \oindicator{\Gamma_{N} \cap \Delta_{N,k}} \right| \geq \frac{c}{2} $$
for $N$ sufficiently large.  Therefore, we conclude that
\begin{equation} \label{eq:diagonalalmost}
	\sup_{z \in D} \sup_{1 \leq k \leq N} \left| \left( \mat{G}_{N}(z) \right)_{kk} \oindicator{\Gamma_{N} \cap \Delta_{N,k}} + \frac{\oindicator{\Gamma_{N} \cap \Delta_{N,k}}}{ z + \frac{\rho}{N} \tr \mat{G}_N(z) \oindicator{\Gamma_{N} \cap \Delta_{N,k}} } \right| \leq \frac{C}{\log N} 
\end{equation}
for $N$ sufficiently large, where $C>0$ depends only on $c$.   

By summing over $k$, we find that, with $m_N(z) := \frac{1}{N} \tr \mat{G}_N(z)$, 
$$ \sup_{z \in D} \left| m_N(z) \oindicator{\Gamma_{N} \cap \Delta_{N,k}} + \frac{1}{z + \rho m_N(z) \oindicator{\Gamma_{N} \cap \Delta_{N,k}} } \right| \leq \frac{C}{\log N} $$
with overwhelming probability.  We now wish to apply \cite[Lemma 4.7]{OR} to conclude that
\begin{equation} \label{eq:mnm}
	\sup_{z \in D} \left| m_N(z) \oindicator{\Gamma_{N} \cap \Delta_{N,k}} - m(z) \right| \leq \frac{C}{\log N}. 
\end{equation}
However, in order to avoid option (2) of the dichotomy presented in \cite[Lemma 4.7]{OR}, we will need to use the value $z_0 \in D$ with $|z_0| > C_0$.  Indeed, for any $\eta > 0$, there exists $C_0$ such that if $|z_0| > C_0$, then by \cite[Lemma 3.1]{OR} and Lemma \ref{lemma:normowp} below, we have
$$ |m_N(z_0) - m(z_0)| \leq |m_N(z_0)| + |m(z_0)| \leq \eta $$
with overwhelming probability.  This rules out option (2) in \cite[Lemma 4.7]{OR}, and we conclude that \eqref{eq:mnm} holds with overwhelming probability for any sufficiently large (in terms of $\delta$ and $\rho$) choice of $C_0$.  

Since 
$$ \frac{1}{z + \rho m(z)} = -m(z), $$ 
it follows from \eqref{eq:diagonalalmost} and \eqref{eq:mnm} that \eqref{eq:diagonalsuffice11} holds with overwhelming probability.  Similarly, \eqref{eq:diagonalsuffice2} also follows from \eqref{eq:diagonalalmost}.  
\end{proof}

\begin{lemma} \label{lemma:normowp}
There exists $C_0 > 0$ such that
$$ \|\mat{X}_N \| \leq C_0 $$
with overwhelming probability.  
\end{lemma}
\begin{proof}
We write $\mat{X}_N = \mat{U}_N + \mat{L}_N$, where $\mat{U}_N$ is an upper-triangular matrix and $\mat{L}_N$ is a strictly lower-triangular matrix (i.e. $\mat{L}_N$ is a lower-triangular matrix whose diagonal entries are all zero).  Thus, by the triangle inequality, it suffices to show that
\begin{equation} \label{eq:norm:un}
	\| \mat{U}_N \| \leq C_0
\end{equation}
and
\begin{equation} \label{eq:norm:ln}
	\| \mat{L}_N \| \leq C_0 
\end{equation}
with overwhelming probability.  

By Lemma \ref{lemma:truncation}, the entries of $\mat{U}_N$ are bounded in magnitude by $4 \eps_N$ almost surely and are jointly independent.  Thus, from \cite[Proposition 2.3.10]{Tbook}, we have, for any $t > 0$, 
\begin{equation} \label{eq:talagrand}
	\Prob \left( \left| \| \mat{U}_N \| - \E \| \mat{U}_N \| \right| \geq t \right) \leq C_1 \exp\left( - c_1 \frac{t^2}{\eps_N^2} \right), 
\end{equation}
where $C_1,c_1 > 0$ are absolute constants.  Moreover, from \cite[Theorem 2]{Lat}, there exists $C_0 > 0$ such that
\begin{equation} \label{eq:latala}
	\E\|\mat{U}_N\| \leq \frac{C_0}{2}
\end{equation}
for all $N \geq 1$.  Combining \eqref{eq:talagrand} (taking $t=C_0/2$) with \eqref{eq:latala}, we conclude that \eqref{eq:norm:un} holds with overwhelming probability.  Similarly, \eqref{eq:norm:ln} also holds with overwhelming probability, and the proof of the lemma is complete.  
\end{proof}

\subsection{A further reduction}
The goal of this subsection is to prove the following reduction.  

\begin{lemma} \label{lemma:reduction2}
For any $\eta > 0$, 
$$ \lim_{N \to \infty} \Prob( | \mathcal{M}_N' - \mathcal{M}_N''| > \eta ) = 0, $$
where $\mathcal{M}_N''$ is defined by replacing each appearance of $Z_{N,k}'(z)$ in $\mathcal{M}'_{N,k}$ with 
\begin{align*}
	Z_{N,k}''(z) &:= - (\E_k - \E_{k-1}) \bigg[ \left(1+ \frac{\rho}{N} \tr \mat{G}_{N,k}^2(z)\right)\left(z + \frac{\rho}{N} \tr \mat{G}_{N,k}(z)\right)^{-2} \alpha_{N,k}(z) \\
		&\quad\qquad\qquad\qquad\qquad + (1+ r_k\mat{G}_{N,k}^2(z) c_k)\left(z + \frac{\rho}{N} \tr \mat{G}_{N,k}(z)\right)^{-1}  \bigg]\oindicator{\Omega_{N,k} \cap Q_{N,k}}.
\end{align*}
\end{lemma}
\begin{proof}
We first observe that 
\begin{align*}
	\sum_{k=1}^N \E &\left | (\tr \mat{G}_N(z) - \tr \mat{G}_{N,k}(z)) \oindicator{\Omega_N} - (\tr \mat{G}_N(z) - \tr \mat{G}_{N,k}(z)) \oindicator{\Omega_N \cap Q_N} \right| \\
		&= \sum_{k=1}^N \E \left| (\tr \mat{G}_N(z) - \tr \mat{G}_{N,k}(z)) \right| |\oindicator{\Omega_N}  \oindicator{Q_N^C} \\
		&\leq \frac{2N^2}{c} \Prob(Q_N^C) = o(1)
\end{align*}
by Corollary \ref{cor:QN}.  Thus, by \eqref{eq:trident}, it suffices to consider
$$ \sum_{k=1}^N (\E_k - \E_{k-1}) \left( \frac{1 + r_k \mat{G}_{N,k}^2(z) c_k}{\frac{1}{\sqrt{N}} y_{kk} - z - r_k \mat{G}_{N,k}(z) c_k} \oindicator{\Omega_N \cap Q_N} \right). $$

By \eqref{def:alpha}, we have that
$$ \frac{\alpha_{N,k}(z)}{ \frac{1}{\sqrt{N}} y_{kk} - z - r_k \mat{G}_{N,k}(z) c_k } = 1 + \frac{z + \frac{\rho}{N} \tr \mat{G}_{N,k}(z) }{ \frac{1}{\sqrt{N}} y_{kk} - z - r_k \mat{G}_{N,k}(z) c_k }. $$
Therefore, on the event $\Omega_N \cap Q_N$, we obtain
\begin{align*}
	&\frac{1 + r_k \mat{G}_{N,k}^2(z) c_k}{\frac{1}{\sqrt{N}} y_{kk} - z - r_k \mat{G}_{N,k}(z) c_k} \\
	&\quad=\frac{ (1+ r_k\mat{G}_{N,k}^2(z) c_k)\left(z + \frac{\rho}{N} \tr \mat{G}_{N,k}(z)\right)^{-1}}{ \frac{1}{\sqrt{N}} y_{kk} - z - r_k \mat{G}_{N,k}(z) c_k } \alpha_{N,k}(z) \\
	&\qquad\qquad - (1+ r_k\mat{G}_{N,k}^2(z) c_k)\left(z + \frac{\rho}{N} \tr \mat{G}_{N,k}(z)\right)^{-1} \\
	&\quad= (1+ r_k\mat{G}_{N,k}^2(z) c_k)\left(z + \frac{\rho}{N} \tr \mat{G}_{N,k}(z)\right)^{-2} \alpha_{N,k}(z) \left[ \frac{\alpha_{N,k}(z)}{ \frac{1}{\sqrt{N}} y_{kk} - z - r_k \mat{G}_{N,k}(z) c_k } - 1 \right] \\
	&\qquad\qquad - (1+ r_k\mat{G}_{N,k}^2(z) c_k)\left(z + \frac{\rho}{N} \tr \mat{G}_{N,k}(z)\right)^{-1} \\
	&\quad= -\left(1+ \frac{\rho}{N} \tr \mat{G}_{N,k}^2(z)\right)\left(z + \frac{\rho}{N} \tr \mat{G}_{N,k}(z)\right)^{-2} \alpha_{N,k}(z) \\
	&\qquad\qquad - \left(r_k \mat{G}_{N,k}^2(z) c_k - \frac{\rho}{N} \tr \mat{G}_{N,k}^2(z)(z) \right) \alpha_{N,k}(z) \left(z + \frac{\rho}{N} \tr \mat{G}_{N,k}(z)\right)^{-2} \\
	&\qquad\qquad + (1+ r_k\mat{G}_{N,k}^2(z) c_k)\left(z + \frac{\rho}{N} \tr \mat{G}_{N,k}(z)\right)^{-2} \frac{ \alpha_{N,k}^2(z)}{ \frac{1}{\sqrt{N}} y_{kk} - z - r_k \mat{G}_{N,k}(z) c_k } \\
	&\qquad\qquad - (1+ r_k\mat{G}_{N,k}^2(z) c_k)\left(z + \frac{\rho}{N} \tr \mat{G}_{N,k}(z)\right)^{-1} \\
	&\quad=: -A_{N,k,1} - A_{N,k,2} + A_{N,k,3} - A_{N,k,4}. 
\end{align*}

We will show that the contributions from $A_{N,k,2}$ and $A_{N,k,3}$ are negligible.  Indeed, we have
\begin{align*}
	\E \left| \sum_{k=1}^N (\E_{k} - \E_{k-1}) A_{N,k,3} \oindicator{\Omega_N \cap Q_N} \right|^2 &\leq 4 \sum_{k=1}^N \E|A_{N,k,3}|^2 \oindicator{\Omega_N \cap Q_{N}} \\
		&\leq 4\frac{2^4}{c^6} \sum_{k=1}^N \E|1 + r_k \mat{G}_{N,k}(z) c_k |^2 |\alpha_{N,k}(z)|^4 \oindicator{\Omega_N \cap Q_{N}}. 
\end{align*}
By the triangle inequality, we have
\begin{align*}
	\left| 1 + r_k \mat{G}_{N,k}^2(z) c_k \right| &\leq 1 + \left| r_k \mat{G}_{N,k}^2(z) c_k - \frac{\rho}{N} \tr \mat{G}_{N,k}^2(z) \right| + \left| \frac{\rho}{N} \tr \mat{G}_{N,k}^2(z) \right|.
\end{align*}
Thus, we obtain
\begin{align*}
	&\sum_{k=1}^N \E |A_{N,k,3}|^2 \oindicator{\Omega_N \cap Q_N} \\
	&\leq C \sum_{k=1}^N \left( \E|\alpha_{N,k}(z)|^4 \oindicator{\Omega_{N,k}} + \E|\alpha_{N,k}(z)|^4 \left|r_k\mat{G}_{N,k}^2(z) c_k - \frac{\rho}{N} \tr \mat{G}_{N,k}^2(z) \right|^2 \oindicator{\Omega_{N,k}} \right) \\
	&\leq C \sum_{k=1}^N \bigg( \E|\alpha_{N,k}(z)|^4 \oindicator{\Omega_{N,k}} \\
	&\qquad\qquad + \sqrt{ \E |\alpha_{N,k}(z)|^8 \oindicator{\Omega_{N,k}} \E\left| r_k \mat{G}_{N,k}^2(z) c_k - \frac{\rho}{N} \tr \mat{G}_{N,k}^2(z) \right|^4 \oindicator{\Omega_{N,k}}  } \bigg) \\
	&= o(1).
\end{align*}
by the Cauchy-Schwarz inequality and Lemma \ref{lemma:alpha}.  Similarly, 
\begin{align*}
	\sum_{k=1}^N &\E |A_{N,k,2}|^2 \oindicator{\Omega_N \cap Q_N} \\
	&\leq C \sum_{k=1}^N \sqrt{ \E |\alpha_{N,k}(z)|^4 \oindicator{\Omega_{N,k}} \E \left| r_k \mat{G}_{N,k}^2(z) c_k - \frac{\rho}{N} \tr \mat{G}_{N,k}^2(z) \right|^4 \oindicator{\Omega_{N,k}} } \\
	&= o(1).
\end{align*}

Therefore, the problem reduces to studying 
\begin{align*}
	-\sum_{k=1}^N (\E_k - \E_{k-1}) (A_{N,k,1} + A_{N,k,4}) \oindicator{\Omega_{N} \cap Q_N}.
\end{align*}
We now show that we can replace the indicator function $\oindicator{\Omega_N \cap Q_N}$ with $\oindicator{\Omega_{N,k} \cap Q_{N,k}}$.  Indeed, on the event $\Omega_{N,k} \cap Q_{N,k}$, we have $|A_{N,k,1}| \leq C |\alpha_{N,k}(z)|$ and
\begin{align*}
	|A_{N,k,4}| \leq C \left( 1 + \left| r_k \mat{G}_{N,k}^2(z) c_k - \frac{\rho}{N} \tr \mat{G}_{N,k}^2(z) \right| \right).
\end{align*}
Thus, by the Cauchy-Schwarz inequality, Lemma \ref{lemma:alpha}, and Corollary \ref{cor:QN}, we obtain
\begin{align*}
	\sum_{k=1}^N &\E |A_{N,k,1} + A_{N,k,4}| |\oindicator{\Omega_{N} \cap Q_N} - \oindicator{\Omega_{N,k} \cap Q_{N,k}} | \\
		&\leq C \sum_{k=1}^N \E\left( |\alpha_{N,k}(z)| \oindicator{\Omega_{N,k}} + \left(1 + \left| r_k \mat{G}_{N,k}^2(z) c_k - \frac{\rho}{N} \tr \mat{G}_{N,k}^2(z) \right| \right) \oindicator{\Omega_{N,k}} \right) \oindicator{\Omega_{N}^C \cup Q_N^C} \\
		&\leq C N \sqrt{ \Prob( \Omega_{N}^C \cup Q_{N}^C) } = o(1). 
\end{align*}
Since 
$$ Z_{N,k}''(z) = -(\E_k - \E_{k-1}) (A_{N,k,1} + A_{N,k,4}) \oindicator{\Omega_{N,k} \cap Q_{N,k}}, $$
the proof of the lemma is complete.  
\end{proof}

\begin{lemma} \label{lemma:reduction2calc}
Under the assumptions of Theorem \ref{thm:main}, 
$$ Z_{N,k}''(z) = \E_k \left[ \frac{d}{dz} \left( \alpha_{N,k}(z) \left( z + \frac{\rho}{N} \tr \mat{G}_{N,k}(z) \right)^{-1} \right) \oindicator{\Omega_{N,k} \cap Q_{N,k}} \right]. $$
\end{lemma}
\begin{proof}
We observe that
$$ E_{k-1}\left[ \left(1+ \frac{\rho}{N} \tr \mat{G}_{N,k}^2(z)\right)\left(z + \frac{\rho}{N} \tr \mat{G}_{N,k}(z)\right)^{-2} \alpha_{N,k}(z) \oindicator{\Omega_{N,k} \cap Q_{N,k}} \right] = 0$$
and
\begin{align}
	&(E_k - E_{k-1}) \left[ (1+ r_k\mat{G}_{N,k}^2(z) c_k)\left(z + \frac{\rho}{N} \tr \mat{G}_{N,k}(z)\right)^{-1} \oindicator{\Omega_{N,k} \cap Q_{N,k}} \right] \label{def:ZNk2} \\
	&\qquad= E_k \left[ \left(z + \frac{\rho}{N} \tr \mat{G}_{N,k}(z)\right)^{-1} \left( r_k \mat{G}_{N,k}^2(z) c_k - \frac{\rho}{N} \tr \mat{G}_{N,k}^2(z) \right) \oindicator{\Omega_{N,k} \cap Q_{N,k}} \right]. \nonumber
\end{align}
Thus,
\begin{align*}
	Z_{N,k}''(z) = - E_{k} &\bigg[ \left(1+ \frac{\rho}{N} \tr \mat{G}_{N,k}^2(z)\right)\left(z + \frac{\rho}{N} \tr \mat{G}_{N,k}(z)\right)^{-2} \alpha_{N,k}(z) \oindicator{\Omega_{N,k} \cap Q_{N,k}} \\
		& \quad + \left(z + \frac{\rho}{N} \tr \mat{G}_{N,k}(z)\right)^{-1} \left( r_k \mat{G}_{N,k}^2(z) c_k - \frac{\rho}{N} \tr \mat{G}_{N,k}^2(z) \right) \oindicator{\Omega_{N,k} \cap Q_{N,k}} \bigg]. 
\end{align*}
As
$$ \frac{d}{dz} G_{N,k}(z) = G_{N,k}^2(z), $$
it follows that
\begin{align*}
	\frac{d}{dz} &\left( \alpha_{N,k}(z) \left( z + \frac{\rho}{N} \tr \mat{G}_{N,k}(z) \right)^{-1} \right) \\
		&\qquad= -\left(1+ \frac{\rho}{N} \tr \mat{G}_{N,k}^2(z)\right)\left(z + \frac{\rho}{N} \tr \mat{G}_{N,k}(z)\right)^{-2} \alpha_{N,k}(z) \\
		&\qquad\qquad - \left(r_k\mat{G}_{N,k}^2(z) c_k - \frac{\rho}{N} \tr \mat{G}_{N,k}^2(z) \right)\left(z + \frac{\rho}{N} \tr \mat{G}_{N,k}(z)\right)^{-1},
\end{align*}
and the proof of the lemma is complete.  
\end{proof}

\subsection{Proof of Theorem \ref{thm:finitedim}}
We now prove Theorem \ref{thm:finitedim}.  In view of Lemmas \ref{lemma:reduction} and \ref{lemma:reduction2}, it suffices to apply Theorem \ref{thm:clt} to $\mathcal{M}_N''$.  Thus, we now verify conditions \eqref{eq:clt_variance} and \eqref{eq:clt_lf} of Theorem \ref{thm:clt}.  

For the Lindeberg condition, we observe that
\begin{align*}
	\sum_{k=1}^N \E \left| \mathcal{M}_{N,k}'' \right|^2 \indicator{|\mathcal{M}_{N,k}''| \geq \eta} &\leq \frac{1}{\eta^2} \sum_{k=1}^N \E \left| \mathcal{M}_{N,k}'' \right|^4 \\
		&\leq \frac{CL^3}{\eta^2} \max_{l=1,\ldots, L} \left\{ |\alpha_l| +|\beta_l| \right\} \sum_{k=1}^N \sum_{l=1}^L \E \left| Z_{N,k}''(z_l)\right|^4
\end{align*}
for some absolute constant $C>0$.  By \eqref{def:ZNk2} and Lemma \ref{lemma:alpha}, we have
\begin{align*}
	\sum_{k=1}^N &\sum_{l=1}^L \E \left|Z_{N,k}''(z_l) \right|^4 \\
	&\leq C \sum_{k=1}^N \sum_{l=1}^L \left( \E|\alpha_{N,k}(z_l)|^4 \oindicator{\Omega_{N,k}} + \E\left| r_k \mat{G}_{N,k}^2(z_l) c_k - \frac{\rho}{N} \tr \mat{G}_{N,k}^2(z_l) \right|^4\oindicator{\Omega_{N,k}} \right) \\
	&= o(1),
\end{align*}
where $C>0$ is a constant depending only on $c$.  This verifies condition \eqref{eq:clt_lf} of Theorem \ref{thm:clt}.  

In order to verify condition \eqref{eq:clt_variance}, we consider terms of the form
$$ \sum_{k=1}^N \E_{k-1} \left[ Z_{N,k}''(z) Z_{N,k}''(w) \right] $$
for $z,w \in \mathcal{C}$.  
\begin{remark}
Here we will take advantage of the fact that $\overline{Z_{N,k}''(z)} = Z_{N,k}''(\bar{z})$ since $\mat{Y}_N$ contains real entries.  Moreover, by symmetry, $z \in \mathcal{C}$ if and only if $\bar{z} \in \mathcal{C}$.  
\end{remark}

By  Lemma \ref{lemma:reduction2calc} and Vitali's theorem (see for instance \cite[Lemma 2.14]{BSbook}), we only need to find the limit of
\begin{align*} 
	\sum_{k=1}^N \E_{k-1} &\Bigg[ \alpha_{N,k}(z) \left(z + \frac{\rho}{N} \tr \mat{G}_{N,k}(z) \right)^{-1} \oindicator{\Omega_{N,k} \cap Q_{N,k}} \\
		&\quad \times \E_{k} \left[ \alpha_{N,k}(w) \left(w + \frac{\rho}{N} \tr \mat{G}_{N,k}(w) \right)^{-1} \oindicator{\Omega_{N,k} \cap Q_{N,k}} \right] \Bigg]
\end{align*}
for $z,w \in \mathcal{C}$.  

We now note that we can replace $\left(z + \frac{\rho}{N} \tr \mat{G}_{N,k}(z) \right)^{-1}$ with $-m(z)$ using Lemma \ref{lemma:diagonal}.  Indeed, 
\begin{align*}
	\E &\Bigg| \sum_{k=1}^N \E_{k-1} \Bigg[ \alpha_{N,k}(z) \left(z + \frac{\rho}{N} \tr \mat{G}_{N,k}(z) \right)^{-1} \oindicator{\Omega_{N,k} \cap Q_{N,k}} \\
	&\qquad \times \E_k \left[ \alpha_{N,k}(w) \left(w + \frac{\rho}{N} \tr \mat{G}_{N,k}(w) \right)^{-1} \oindicator{\Omega_{N,k} \cap Q_{N,k}} \right] \Bigg] \\
	& + m(z) \sum_{k=1}^N \E_{k-1} \Bigg[ \alpha_{N,k}(z) \oindicator{\Omega_{N,k} \cap Q_{N,k}} \\
	&\qquad \times \E_k \left[ \alpha_{N,k}(w) \left(w + \frac{\rho}{N} \tr \mat{G}_{N,k}(w) \right)^{-1} \oindicator{\Omega_{N,k} \cap Q_{N,k}} \right] \Bigg] \Bigg| \\
	&\leq \frac{2}{c} \sum_{k=1}^N \E \Bigg| \E_{k} \left| \alpha_{N,k}(w) \oindicator{\Omega_{N,k} \cap Q_{N,k}} \right| \alpha_{N,k}(z) \oindicator{\Omega_{N,k} \cap Q_{N,k}} \\
	&\qquad \times \left[ \left( z + \frac{\rho}{N} \tr \mat{G}_{N,k}(z) \right)^{-1} \oindicator{\Omega_{N,k} \cap Q_{N,k}} + m(z) \right] \Bigg| \\
	&\leq \frac{2}{c} \sum_{k=1}^N \left( \E|\alpha_{N,k}(w)|^4 \oindicator{\Omega_{N,k}} \right)^{1/4} \left( \E|\alpha_{N,k}(z) |^4 \oindicator{\Omega_{N,k}} \right)^{1/4} \\
	&\qquad \times \left( \E \left| \left( z + \frac{\rho}{N} \tr \mat{G}_{N,k}(z) \right)^{-1} \oindicator{\Omega_{N,k} \cap Q_{N,k}} + m(z) \right|^2 \right)^{1/2} \\
	&= o(1)
\end{align*}
by the generalized H\"{o}lder inequality and Lemmas \ref{lemma:alpha} and \ref{lemma:diagonal}.  Similarly, we obtain
\begin{align*}
	\E &\Bigg| \sum_{k=1}^N \E_{k-1} \left[ \alpha_{N,k}(w) \left(w + \frac{\rho}{N} \tr \mat{G}_{N,k}(w) \right)^{-1} \oindicator{\Omega_{N,k} \cap Q_{N,k}}  \E_{k} \left[ \alpha_{N,k}(z) \oindicator{\Omega_{N,k} \cap Q_{N,k}} \right] \right] \\
	& + m(w) \sum_{k=1}^N \E_{k-1} \left[ \alpha_{N,k}(w) \oindicator{\Omega_{N,k} \cap Q_{N,k}}  \E_{k} \left[ \alpha_{N,k}(z) \oindicator{\Omega_{N,k} \cap Q_{N,k}} \right] \right] \Bigg| \\
	& = o(1).
\end{align*}
Thus, since $m(z)$ is uniformly bounded on $\mathcal{C}$, we only need to find the limit of 
\begin{align*}
	m(z) &m(w) \beta_{N}(z,w) \\
	&:= m(z) m(w) \sum_{k=1}^N \E_{k-1} \left[ \E_{k} \left( \alpha_{N,k}(z) \oindicator{\Omega_{N,k} \cap Q_{N,k}} \right) \E_{k} \left( \alpha_{N,k}(w) \oindicator{\Omega_{N,k} \cap Q_{N,k}} \right) \right]. 
\end{align*}
In particular, in order to complete the proof of Theorem \ref{thm:finitedim}, it suffices to show that $\beta_N(z,w)$ converges in probability to $\beta(z,w)$ as $N \to \infty$, where $\beta(z,w)$ is defined in \eqref{def:beta}.  

By definition of $\beta_{N}$, we have
\begin{align}
	\beta_N(z,w) = \sum_{k=1}^N &\Bigg( \frac{\sigma^2}{N} \E_{k-1}\left[ \left( \E_{k} \oindicator{\Omega_{N,k} \cap Q_{N,k}} \right)^2 \right] \label{eq:betaN} \\
	&+ \E_{k-1} \bigg[ \E_{k} \left[ \left(r_k \mat{G}_{N,k}(z) c_k - \frac{\rho}{N} \tr \mat{G}_{N,k}(z) \right) \oindicator{\Omega_{N,k} \cap Q_{N,k}} \right] \nonumber \\
	&\qquad\qquad \times \E_{k} \left[ \left(r_k \mat{G}_{N,k}(w) c_k - \frac{\rho}{N} \tr \mat{G}_{N,k}(w) \right) \oindicator{\Omega_{N,k} \cap Q_{N,k}} \right] \bigg] \Bigg). \nonumber
\end{align}
For the first term, we note that
\begin{align*}
	\E \left| \sum_{k=1}^N \frac{\sigma^2}{N} \E_{k-1}\left[ \left( \E_{k} \oindicator{\Omega_{N,k} \cap Q_{N,k}}  \right)^2 \right]- \sigma^2 \right| &\leq 2 \frac{\sigma^2}{N} \sum_{k=1}^N \E \left| \E_{k} \oindicator{\Omega_{N,k} \cap Q_{N,k}} -1 \right| \\
		&\leq 2 \frac{\sigma^2}{N} \sum_{k=1}^N \Prob( \Omega_{N,k}^C \cup Q_{N,k}^C) \\
		&= o(1)
\end{align*}
by \eqref{eq:hatONC} and Corollary \ref{cor:QN}.  

We now consider the second term on the right hand side of \eqref{eq:betaN}.  Define
$$ b_{kij}(z) := \E_{k} \left[ \left( \mat{G}_{N,k}(z) \right)_{ij} \oindicator{\Omega_{N,k} \cap Q_{N,k}} \right]. $$
Then
\begin{align*}
	 \E_{k} &\left[ \left(r_k \mat{G}_{N,k}(z) c_k - \frac{\rho}{N} \tr \mat{G}_{N,k}(z) \right) \oindicator{\Omega_{N,k} \cap Q_{N,k}} \right] \\
	 &\qquad\qquad= \sum_{i,j=1}^{k-1} \left(r_k\right)_i b_{kij}(z) \left(c_k\right)_j  - \frac{\rho}{N} \sum_{i=1}^{k-1} b_{kii}(z),
\end{align*}
where $(r_k)_i$ denotes the $i$-th entry of $r_k$ and $(c_k)_j$ denotes the $j$-th entry of $c_k$.  Thus, we conclude that
\begin{align*}
	\beta_{N}(z,w) &= \sigma^2 + \frac{\rho^2}{N^2} \sum_{k=1}^N \sum_{i,j=1}^{k-1} b_{kij}(z) b_{kji}(w) + \frac{1}{N^2} \sum_{k=1}^N \sum_{i,j=1}^{k-1} b_{kij}(z) b_{kij}(w) \\
	&\qquad + \frac{\E[\xi_1^2 \xi_2^2] - 2 \rho^2 - 1}{N^2} \sum_{k=1}^N \sum_{i=1}^{k-1} b_{kii}(z) b_{kii}(w) + o(1) \\
	&=: \sigma^2 + \rho^2 \beta_{N,1}(z,w) + \beta_{N,2}(z,w) + \left(\E[\xi_1^2 \xi_2^2] - 2 \rho^2 - 1\right) \beta_{N,3}(z,w) + o(1).
\end{align*}
Here $o(1)$ denotes a term which converges to zero in probability.  From Lemma \ref{lemma:diagonal}, we obtain
$$ \E \left| \beta_{N,3}(z,w) - \frac{1}{2}m(z) m(w) \right| = o(1). $$
Thus, we find that
\begin{align}
	\beta_{N}(z,w) &= \sigma^2 + \rho^2 \beta_{N,1}(z,w) + \beta_{N,2}(z,w) \nonumber \\
	&\qquad\qquad + \left(\frac{\E[\xi_1^2 \xi_2^2] - 2 \rho^2 - 1}{2}\right) m(z) m(w) + o(1). \label{eq:betaNreduce}
\end{align}
It remains to compute the limit of $\beta_{N,1}(z,w)$ and $\beta_{N,2}(z,w)$.

\subsection{Limit of $\beta_{N,1}$ and $\beta_{N,2}$}

In order to compute the limit of $\beta_{N,1}$ and $\beta_{N,2}$, we will need the following decomposition.  Let $e_j$ ($j=1, \ldots, k-1, k+1, \ldots, N$) be the  $(N-1)$-vector whose $j$-th (or $(j-1)$-th) element is $1$ and others are $0$ if $j < k$ (or $j > k$ correspondingly).  Thus, for $i,j \neq k$, 
$$ e_i^\mathrm{T} \mat{G}_{N,k}(z) e_j = \left( \mat{G}_{N,k}(z) \right)_{ij}. $$
Clearly, $e_j$ also depends on $k$.  We now fix $1 \leq k \leq N$.  We will write $e_j$ and not denote the dependence on $k$.  We note that all of the constants in the bounds below are independent of $k$.  

For $i,j$ different from $k$, define
$$ \mat{X}_{N,k}^{(i,j)} := \mat{X}_{N,k} - \tau_{ij} \left( x_{ij} e_i e_j^\mathrm{T} + x_{ji} e_j e_i^\mathrm{T} \right), $$
where 
$$ \tau_{ij} := \left\{
	\begin{array}{lr}
		1, & \text{ if } i \neq j\\
		1/2, & \text{ if } i = j.
	\end{array} \right. $$
We also define
$$ \mat{G}_{N,k}^{(i,j)}(z) := \left( \mat{X}_{N,k}^{(i,j)} - z \mat{I} \right)^{-1}. $$
We will need the following lemmata.  

\begin{lemma} \label{lemma:ijbnd}
For $N$ sufficiently large (such that $\eps_N \leq c/16$),
$$ \sup_{z \in \mathcal{C}} \sup_{1 \leq k \leq N} \sup_{i,j \neq k} \| \mat{G}_{N,k}^{(i,j)}(z) \oindicator{\Omega_{N,k}} \| \leq \frac{2}{c}. $$
\end{lemma}
\begin{proof}
By Weyl's inequality, we have
$$ \left| \sigma_{N-1}\left( \mat{X}_{N,k} - z \mat{I} \right) - \sigma_{N-1} \left( \mat{X}_{N,k}^{(i,j)} - z \mat{I} \right) \right| \leq \| x_{ij} e_i e_j^\mathrm{T} + x_{ji} e_j e_i^\mathrm{T}, \| $$
and hence, from Lemma \ref{lemma:truncation}, we obtain
$$ \left| \sigma_{N-1}\left( \mat{X}_{N,k} - z \mat{I} \right) - \sigma_{N-1} \left( \mat{X}_{N,k}^{(i,j)} - z \mat{I} \right) \right| \leq 8 \eps_N. $$
Thus, on the event $\Omega_{N,k}$, we conclude that 
$$ \sigma_{N-1} \left( \mat{X}_{N,k}^{(i,j)} - z \mat{I} \right) \geq c - 8 \eps_N \geq \frac{c}{2} $$
for $8 \eps_N \leq c/2$.  Since this bound holds uniformly for $z \in \mathcal{C}$, $1 \leq k \leq N$, and $i,j \neq k$, the proof of the lemma is complete.  
\end{proof}

\begin{lemma} \label{lemma:ijdiagonal}
One has
$$ \sup_{z \in \mathcal{C}} \sup_{1 \leq k \leq N} \sup_{i,j \neq k} \sup_{l \neq k} \E \left|  e_l^\mathrm{T} \mat{G}_{N,k}^{(i,j)}(z) e_l \oindicator{\Omega_{N,k} \cap Q_{N,k}} - m(z) \right|^2 = o(1) $$
and
$$ \sup_{z \in \mathcal{C}} \sup_{1 \leq k \leq N} \sup_{i,j \neq k} \sup_{l \neq k} \E \left|  e_l^\mathrm{T} \mat{G}_{N,k}^{(i,j)}(z) e_l - m(z) \right|^2 \oindicator{\Omega_{N,k} \cap Q_{N,k}} = o(1). $$
\end{lemma}
\begin{proof}
By the resolvent identity, we observe that
$$ \left| e_l^\mathrm{T} \left( \mat{G}_{N,k}(z) - \mat{G}_{N,k}^{(i,j)} \right) e_l \right| \leq \| \mat{G}_{N,k}(z) \| \| \mat{G}_{N,k}^{(i,j)} \| \|x_{ij} e_i e_j^\mathrm{T} + x_{ji} e_i e_j^\mathrm{T} \|. $$
Thus, by Lemmas \ref{lemma:truncation} and \ref{lemma:ijbnd}, we have
$$ \left| e_l^\mathrm{T} \left( \mat{G}_{N,k}(z) - \mat{G}_{N,k}^{(i,j)} \right) e_l \right| \leq \frac{16 \eps_N}{c^2} $$
on the event $\Omega_{N,k}$.  The claim now follows from Lemma \ref{lemma:diagonal}.  
\end{proof}

We now consider the elements $b_{kl_1 l_2}(z)$.  By the resolvent identity, we have
$$ z \mat{G}_{N,k}(z) = -\mat{I} + \sum_{i,j \neq k} x_{ij} e_i e_j^\mathrm{T} \mat{G}_{N,k}(z). $$
Again by the resolvent identity, we note that
\begin{equation} \label{eq:resolvdiff}
	\mat{G}_{N,k}(z) - \mat{G}_{N,k}^{(i,j)}(z) = - \mat{G}_{N,k}^{(i,j)}(z) \tau_{ij}(x_{ij} e_i e_j^\mathrm{T} + x_{ji} e_j e_i^\mathrm{T}) \mat{G}_{N,k}(z). 
\end{equation}
Thus, we obtain the decomposition 
\begin{align*}
	z \mat{G}_{N,k}(z) &= - \mat{I} + \sum_{i,j \neq k} x_{ij} e_i e_j^\mathrm{T} \mat{G}_{N,k}^{(i,j)}(z) \\
		&\qquad- \sum_{i,j\neq k} x_{ij} e_i e_j^\mathrm{T} \mat{G}_{N,k}^{(i,j)}(z) \tau_{ij} (x_ij e_i e_j^\mathrm{T} + x_{ji} e_j e_i^\mathrm{T}) \mat{G}_{N,k}(z) \\
	&= -\mat{I} + \mat{A}_{Nk}(z) + \mat{C}_{N,k}(z) + \mat{D}_{N,k}(z) + \mat{E}_{N,k}(z) + \mat{F}_{N,k}(z), 
\end{align*}
where
\begin{align*}
	\mat{A}_{N,k}(z) &:= \sum_{i,j \neq k} x_{ij} e_i e_j^\mathrm{T} \mat{G}_{N,k}^{(i,j)}(z), \\
	\mat{C}_{N,k}(z) &:= -\sum_{i,j \neq k} \tau_{ij} x_{ij} x_{ji} \left( e_j^\mathrm{T} \mat{G}_{N,k}^{(i,j)}(z) e_j - m(z) \right) e_i e_i^\mathrm{T} \mat{G}_{N,k}(z), \\
	\mat{D}_{N,k}(z) &:= - \frac{\rho}{N} m(z) \sum_{i,j \neq k} e_i e_i^\mathrm{T} \mat{G}_{N,k}(z) \tau_{ij} \\
		&= -m(z) \frac{\rho}{N} (N - 3/2) \sum_{i \neq k} e_i e_i^\mathrm{T} \mat{G}_{N,k}(z), \\
	\mat{E}_{N,k}(z) &:= - m(z) \sum_{i,j \neq k} \tau_{ij} \left( x_{ij} x_{ji} - \frac{\rho}{N} \right) e_i e_i^\mathrm{T} \mat{G}_{N,k}(z), \\
	\mat{F}_{N,k}(z) &:= - \sum_{i,j \neq k} \tau_{ij} x_{ij}^2 e_j^\mathrm{T} \mat{G}_{N,k}^{(i,j)}(z) e_i e_i e_j^\mathrm{T} \mat{G}_{N,k}(z).
\end{align*}

We will now show that the contributions from $\mat{C}_{N,k}(z), \mat{E}_{N,k}(z)$, and $\mat{F}_{N,k}(z)$ are negligible.  Throughout, we will take advantage of the fact that the norm of a matrix is not less than that of its sub-matrices.  So, for instance, 
\begin{equation} \label{eq:subbnd1}
	\sup_{l_1 < k} \left| \sum_{l_2 < k} e_{l_1}^\mathrm{T} \mat{G}_{N,k}(z)  e_{l_2} e_{l_2}^\mathrm{T} \E_{k-1} \left[ \mat{G}_{N,k}(w) \oindicator{\Omega_{N,k} \cap Q_{N,k}} \right] e_{l_1} \right| \oindicator{\Omega_{N,k} \cap Q_{N,k}} \leq \frac{1}{c^2} 
\end{equation}
and
\begin{equation} \label{eq:subbnd2}
	\sup_{l_1 < k} \left| \sum_{l_2 < k} e_{l_1}^\mathrm{T} \mat{G}_{N,k}(z) e_{l_2} e_{l_1}^\mathrm{T} \E_{k-1} \left[ \mat{G}_{N,k}(w) \oindicator{\Omega_{N,k} \cap Q_{N,k}} \right] e_{l_2} \right| \oindicator{\Omega_{N,k} \cap Q_{N,k}} \leq \frac{1}{c^2}. 
\end{equation}

By \eqref{eq:subbnd1}, Lemma \ref{lemma:ijdiagonal}, and the Cauchy-Schwarz inequality, we have
\begin{align*}
	\frac{1}{N}& \sum_{l_1 < k} \E \left| \sum_{l_2 < k} e_{l_1}^{\mathrm{T}} \E_{k-1} \left[ \mat{C}_{N,k}(z) \oindicator{\Omega_{N,k} \cap Q_{N,k}} \right] e_{l_2} e_{l_2}^\mathrm{T} \E_{k-1} \left[ \mat{G}_{N,k}(w) \oindicator{\Omega_{N,k} \cap Q_{N,k}} \right] e_{l_1} \right| \\
	&\leq \frac{1}{Nc^2} \sum_{l_1 < k} \sum_{j \neq k} \E \left| x_{l_1 j} x_{j l_1} \oindicator{\Omega_{N,k} \cap Q_{N,k}} \left( e_j^\mathrm{T} \mat{G}_{N,k}^{(i,j)}(z) e_j - m(z) \right) \right| \\
	&\leq \frac{C}{N^2 c^2} \sum_{l_1 < k} \sum_{j \neq k} \sqrt{ \E \left| e_j \mat{G}_{N,k}^{(i,j)}(z) e_j - m(z) \right|^2 \oindicator{\Omega_{N,k}} } \\
	&= o(1)
\end{align*}
uniformly in $k$, where the constant $C>0$ depends only on the joint moment $\E|\xi_1^2 \xi_2^2|$.  

Similarly, by applying \eqref{eq:subbnd2} instead of \eqref{eq:subbnd1}, we find that
\begin{align*}
\frac{1}{N}& \sum_{l_1 < k} \E \left| \sum_{l_2 < k} e_{l_1}^{\mathrm{T}} \E_{k-1} \left[ \mat{C}_{N,k}(z) \oindicator{\Omega_{N,k} \cap Q_{N,k}} \right] e_{l_2} e_{l_1}^\mathrm{T} \E_{k-1} \left[ \mat{G}_{N,k}(w) \oindicator{\Omega_{N,k} \cap Q_{N,k}} \right] e_{l_2} \right| = o(1)
\end{align*}
uniformly in $k$.  

For the term involving $\mat{E}_{N,k}(z)$, we have
\begin{align*}
	\frac{1}{N} &\sum_{l_1 < k} \E \left| \sum_{l_2 < k} e_{l_1}^\mathrm{T} \E_{k-1} \left[ \mat{E}_{N,k}(z) \oindicator{\Omega_{N,k} \cap Q_{N,k}} \right] e_{l_2} e_{l_2}^\mathrm{T} \E_{k-1} \left[ \mat{G}_{N,k}(w) \oindicator{\Omega_{N,k} \cap Q_{N,k}} \right] e_{l_1} \right| \\
	&\leq \frac{C}{N c^2} \sum_{l_1 < k} \E \left| \sum_{j \neq k} \tau_{l_1 j} \left( x_{l_1 j} x_{j l_1} - \frac{\rho}{N} \right) \right| \\
	&\leq \frac{C}{N c^2} \sum_{l_1 < k} \sqrt{ \E \left| \sum_{j \neq k, l_1} \tau_{l_1 j} \left( x_{l_1 j} x_{j l_1} - \frac{\rho}{N} \right) \right|^2} +o(1)
\end{align*}
uniformly in $k$, where $C = \sup_{z \in \mathcal{C}} |m(z)| < \infty$ since $m(z)$ is continuous outside $\mathcal{E}_{\rho}$.  Since $\{(x_{ij}, x_{ji}) : 1 \leq i < j \leq N\}$ is a collection of independent random pairs, we have
\begin{align*}
	\E \left| \sum_{j \neq k, l_1} \tau_{l_1 j} \left( x_{l_1 j} x_{j l_1} - \frac{\rho}{N} \right) \right|^2 = \sum_{j \neq k, l_1} \left( \E(x_{l_1 j}^2 x_{j l_1}^2) - \frac{\rho^2}{N^2} \right) = O \left( \frac{1}{N} \right)
\end{align*}
uniformly in $k, l_1$.  Thus, we conclude that 
$$ \frac{1}{N} \sum_{l_1 < k} \E \left| \sum_{l_2 < k} e_{l_1}^\mathrm{T} \E_{k-1} \left[ \mat{E}_{N,k}(z) \oindicator{\Omega_{N,k} \cap Q_{N,k}} \right] e_{l_2} e_{l_2}^\mathrm{T} \E_{k-1} \left[ \mat{G}_{N,k}(w) \oindicator{\Omega_{N,k} \cap Q_{N,k}} \right] e_{l_1} \right| = o(1) $$
uniformly in $k$.  

Similarly, by applying \eqref{eq:subbnd2}, we obtain
$$ \frac{1}{N} \sum_{l_1 < k} \E \left| \sum_{l_2 < k} e_{l_1}^\mathrm{T} \E_{k-1} \left[ \mat{E}_{N,k}(z) \oindicator{\Omega_{N,k} \cap Q_{N,k}} \right] e_{l_2} e_{l_1}^\mathrm{T} \E_{k-1} \left[ \mat{G}_{N,k}(w) \oindicator{\Omega_{N,k} \cap Q_{N,k}} \right] e_{l_2} \right| = o(1) $$
uniformly in $k$.  

Finally, for the $\mat{F}_{N,k}(z)$ term, we have
\begin{align*}
	\sum_{l_2 < k} &e_{l_1}^\mathrm{T} \E_{k-1} \left[ \mat{F}_{N,k}(z) \oindicator{\Omega_{N,k} \cap Q_{N,k}} \right] e_{l_2} e_{l_2}^\mathrm{T} \E_{k-1} \left[ \mat{G}_{N,k}(w) \oindicator{\Omega_{N,k} \cap Q_{N,k}} \right] e_{l_1} \\
	&= - \sum_{j \neq k} \tau_{l_1 j} \E_{k-1} \left[ x_{l_1 j}^2 \oindicator{\Omega_{N,k} \cap Q_{N,k}} e_j^\mathrm{T} \mat{G}_{N,k}^{(l_1, j)}(z) e_{l_1} e_j^\mathrm{T} \mat{H}_{N,k}(z,w) e_{l_1} \right],
\end{align*}
where 
$$ H_{N,k}(z,w) := \sum_{l_2 < k} \mat{G}_{N,k}(z) e_{l_2} e_{l_2}^\mathrm{T} \E_{k-1} \left[ \mat{G}_{N,k}(w) \oindicator{\Omega_{N,k} \cap Q_{N,k}} \right]. $$
We observe that $\mat{H}_{N,k}(z,w)$ is the product of the sub-matrices formed from the first $k-1$ columns of $\mat{G}_{N,k}(z)$ and the first $k-1$ rows of $\E_{k-1}[\mat{G}_{N,k}(w) \oindicator{\Omega_{N,k} \cap Q_{N,k}}]$.  Thus, $\| \mat{H}_{N,k}(z,w) \| \leq \frac{1}{c^2}$ on $\Omega_{N,k}$.  Therefore, by Lemma \ref{lemma:ijbnd} and the Cauchy-Schwarz inequality, we have
\begin{align*}
	\frac{1}{N} &\E \left| \sum_{l_1, l_2 < k} e_{l_1}^\mathrm{T} \E_{k-1} \left[ \mat{F}_{N,k}(z) \oindicator{\Omega_{N,k} \cap Q_{N,k}} \right] e_{l_2} e_{l_2}^\mathrm{T} \E_{k-1} \left[ \mat{G}_{N,k}(w) \oindicator{\Omega_{N,k} \cap Q_{N,k}} \right] e_{l_1} \right| \\
	&\leq \frac{1}{N} \E \left| \sum_{l_1 < k} \sum_{ j \neq k} \tau_{l_1 j} x_{l_1 j}^2 \oindicator{\Omega_{N,k} \cap Q_{N,k}} e_j^\mathrm{T} \mat{G}_{N,k}^{(l_1,j)}(z) e_{l_1} e_j^\mathrm{T} \mat{H}_{N,k}(z,w) e_{l_1} \right| \\
	&\leq \frac{2}{Nc} \sqrt{\sum_{l_1, j \neq k} \E|x_{l_1 j}|^4 } \sqrt{ \sum_{l_1, j \neq k} \E \left| e_j^\mathrm{T} \mat{H}_{N,k}(z,w) e_{l_1} \right|^2 \oindicator{\Omega_{N,k}} } \\
	&\leq \frac{2C}{Nc} \sqrt{ \E \| \mat{H}_{N,k}(z,w) \|_2^2 \oindicator{\Omega_{N,k}} } \\
	&\leq \frac{2C}{\sqrt{N} c} \sqrt{ \E \| \mat{H}_{N,k}(z,w) \|^2 \oindicator{\Omega_{N,k}} }\\
	&\leq \frac{2C}{\sqrt{N} c^3}
\end{align*}
where $C>0$ depends only on the fourth moments of the atom variables $\xi_1, \xi_2, \zeta$.  

Similarly, we find that
$$ \frac{1}{N} \E \left| \sum_{l_1, l_2 < k} e_{l_1}^\mathrm{T} \E_{k-1} \left[ \mat{F}_{N,k}(z) \oindicator{\Omega_{N,k} \cap Q_{N,k}} \right] e_{l_2} e_{l_1}^\mathrm{T} \E_{k-1} \left[ \mat{G}_{N,k}(w) \oindicator{\Omega_{N,k} \cap Q_{N,k}} \right] e_{l_2} \right| = o(1) $$
uniformly in $k$.  

The inequalities above show that the contributions from the matrices $\mat{C}_{N,k}(z)$, $\mat{E}_{N,k}(z)$, and $\mat{F}_{N,k}(z)$ are negligible.  We now turn our attention to the contributive components.  

We first observe that 
\begin{align*}
	\sum_{l_1,l_2 < k} &e_{l_1}^\mathrm{T} \E_{k-1} \left[ \mat{I}_{N-1} \oindicator{\Omega_{N,k} \cap Q_{N,k}} \right] e_{l_2} e_{l_2}^\mathrm{T} \E_{k-1} \left[ \mat{G}_{N,k}(w) \oindicator{\Omega_{N,k} \cap Q_{N,k}} \right] e_{l_1} \\
	&= \sum_{l_1 < k} E_{k-1} \left[ \oindicator{\Omega_{N,k} \cap Q_{N,k}} \right] e_{l_1}^\mathrm{T} E_{k-1} \left[ \mat{G}_{N,k}(w) \oindicator{\Omega_{N,k} \cap Q_{N,k}} \right] e_{l_1}.
\end{align*}
Thus, since $\Omega_{N,k} \cap Q_{N,k}$ holds with overwhelming probability, we apply Lemma \ref{lemma:diagonal} and obtain
\begin{align*}
	\E &\left| \frac{1}{N} \sum_{l_1 < k} E_{k-1} \left[ \oindicator{\Omega_{N,k} \cap Q_{N,k}} \right] e_{l_1}^\mathrm{T} E_{k-1} \left[ \mat{G}_{N,k}(w) \oindicator{\Omega_{N,k} \cap Q_{N,k}} \right]e_{l_1} - \frac{k-1}{N} m(w) \right| \\
	&\leq \frac{C}{N} \sum_{l_1 < k} \E \left| \E_{k-1} \left[ \oindicator{\Omega_{N,k} \cap Q_{N,k}} \right] - 1 \right| \\
	&\qquad + \frac{1}{N} \sum_{l_1 < k} \E \left| e_{l_1}^\mathrm{T} E_{k-1} \left[ \mat{G}_{N,k}(w) \oindicator{\Omega_{N,k} \cap Q_{N,k}} \right] e_{l_1} - m(w) \right| \\
	&= o(1)
\end{align*}
uniformly in $k$.  

Similarly, 
\begin{align*}
	\frac{1}{N} \E &\Bigg| \sum_{l_1, l_2 < k } e_{l_1}^\mathrm{T} \E_{k-1} \left[ \mat{I}_{N-1} \oindicator{\Omega_{N,k} \cap Q_{N,k}} \right] e_{l_2} e_{l_1}^\mathrm{T} \E_{k-1} \left[ \mat{G}_{N,k}(w) \oindicator{\Omega_{N,k} \cap Q_{N,k}} \right] e_{l_2}^\mathrm{T} \\
	&\qquad\qquad - (k-1) m(w) \Bigg| = o(1) 
\end{align*}
uniformly in $k$.  

For the $\mat{D}_{N,k}(z)$ terms, we observe that
\begin{align*}
	\sum_{l_1, l_2 < k} &e_{l_1}^\mathrm{T} \E_{k-1} \left[ \mat{D}_{n,k}(z) \oindicator{\Omega_{N,k} \cap Q_{N,k}} \right] e_{l_2} e_{l_2}^\mathrm{T} \E_{k-1} \left[ \mat{G}_{N,k}(w) \oindicator{\Omega_{N,k} \cap Q_{N,k}}\right] e_{l_1} \\
	&= -m(z) \rho \frac{(N-3/2)}{N} \sum_{l_1, l_2 < k} b_{k l_1 l_2}(z) b_{k l_2 l_1}(w)
\end{align*}
and
\begin{align*}
	\sum_{l_1, l_2 < k} &e_{l_1}^\mathrm{T} \E_{k-1} \left[ \mat{D}_{N,k}(z) \oindicator{\Omega_{N,k} \cap Q_{N,k}} \right] e_{l_2} e_{l_1}^\mathrm{T} \E_{k-1} \left[ \mat{G}_{N,k}(w) \oindicator{\Omega_{N,k} \cap Q_{N,k}} \right] e_{l_2} \\
	&= -m(z) \rho \frac{(N-3/2)}{N} \sum_{l_1, l_2 < k} b_{k l_1 l_2}(z) b_{k l_1 l_2}(w).
\end{align*}

It remains to consider the terms involving $\mat{A}_{N,k}(z)$.  First, we collect a number of useful calculations in the following lemmata.  

\begin{lemma} \label{lemma:mean0}
For $l_1,l_2 < k$, one has 
\begin{align*}
	\E \left| e_j^\mathrm{T} \E_{k-1} \left[\mat{G}_{N,k}^{(l_1, j)}(z) x_{l_1j} \oindicator{\Omega_{N,k} \cap Q_{N,k}} \right] e_{l_2} \right|^2 = O(N^{-100})
\end{align*}
uniformly if $j \geq k$, and
\begin{align*}
	\E \Bigg| &e_j^\mathrm{T} \E_{k-1} \left[\mat{G}_{N,k}^{(l_1, j)}(z) x_{l_1j} \oindicator{\Omega_{N,k} \cap Q_{N,k}} \right] e_{l_2} \\
	&\qquad - x_{l_1 j} e_{j}^\mathrm{T} \E_{k-1} \left[ \mat{G}_{N,k}^{(l_1,j)}(z) \oindicator{\Omega_{N,k} \cap Q_{N,k}} \right] e_{l_2} \Bigg|^2 = O(N^{-100})
\end{align*}
uniformly if $j < k$.  
\end{lemma}
\begin{proof}
Define the event
$$ \Omega_{N,k}^{(l_1,j)} := \left\{ \| \mat{G}_{N,k}^{(l_1,j)}(z) \| \leq \frac{2}{c} \right\}. $$
By Lemma \ref{lemma:ijbnd}, $\Omega_{N,k} \subset \Omega_{N,k}^{(l_1,j)}$ for $N$ sufficiently large.  Thus,
\begin{align*}
	\E &\left| e_j^\mathrm{T} \E_{k-1} \left[ \mat{G}_{N,k}^{(l_1,j)}(z) x_{l_1 j} \oindicator{\Omega_{N,k} \cap Q_{N,k}} \right] e_{l_2} - e_{j}^\mathrm{T} \E_{k-1} \left[ \mat{G}_{N,k}^{(l_1,j)}(z) x_{l_1 j} \oindicator{\Omega_{N,k}^{(l_1,j)}} \right] e_{l_2} \right|^2 \\
	&\qquad\qquad\leq \left( \frac{8}{c} \eps_N \right)^2 \left[ \Prob(Q_N^C) + \Prob(\Omega_{N,k}^C) \right] = O(N^{-100}) 
\end{align*}
since $Q_N$ and $\Omega_{N,k}$ hold with overwhelming probability.  By independence, we have
$$ \E_{k-1} \left[ \mat{G}_{N,k}^{(l_1,j)}(z) x_{l_1 j}\oindicator{\Omega_{N,k}^{(l_1,j)}} \right] = \left\{
		\begin{array}{lr}
			0, & \text{if } j \geq k\\
			 x_{l_1 j} \E_{k-1} \left[ \mat{G}_{N,k}^{(l_1,j)}(z) \oindicator{\Omega_{N,k}^{(l_1,j)}}\right]  & \text{if } j < k.
		 \end{array}
	\right. $$
	
By using once more the fact that $Q_N$ and $\Omega_{N,k}$ hold with overwhelming probability, it is straight-forward to verify that
\begin{align*}
	\E &\left| x_{l_1 j} e_j^\mathrm{T} \E_{k-1} \left[ \mat{G}_{N,k}^{(l_1,j)}(z) \oindicator{\Omega_{N,k}^{(l_1,j)}} \right] e_{l_2} - x_{l_1 j} e_j^\mathrm{T} \E_{k-1} \left[ \mat{G}_{N,k}^{(l_1, j)}(z) \oindicator{\Omega_{N,k} \cap Q_{N,k}} \right] e_{l_2} \right|^2 \\
	&= O(N^{-100}),
\end{align*}
and the proof of the lemma is complete.  
\end{proof}

\begin{lemma} \label{lemma:L2}
One has
\begin{align}
	\sup_{1 \leq k \leq N} \sup_{l_1 < k} \E \Bigg| \sum_{j \neq k} \sum_{l_2 < k}& e_{j}^\mathrm{T} \E_{k-1} \left[ x_{l_1 j} \mat{G}_{N,k}^{(l_1, j)}(z) \oindicator{\Omega_{N,k} \cap Q_{N,k}} \right] e_{l_2} \nonumber \\
	&\qquad\times e_{l_2}^\mathrm{T} \E_{k-1} \left[ \mat{G}_{N,k}^{(l_1, j)}(w) \oindicator{\Omega_{N,k} \cap Q_{N,k}} \right] e_{l_1} \Bigg|^2 = o(1) \label{eq:L2:1}
\end{align}
and
\begin{align} 
	\sup_{1 \leq k \leq N} \sup_{l_1 < k} \E \Bigg| \sum_{j \neq k} \sum_{l_2 < k}& e_{j}^\mathrm{T} \E_{k-1} \left[ x_{l_1 j} \mat{G}_{N,k}^{(l_1, j)}(z) \oindicator{\Omega_{N,k} \cap Q_{N,k}} \right] e_{l_2} \nonumber \\
	&\qquad\times e_{l_1}^\mathrm{T} \E_{k-1} \left[ \mat{G}_{N,k}^{(l_1, j)}(w) \oindicator{\Omega_{N,k} \cap Q_{N,k}} \right] e_{l_2} \Bigg|^2 = o(1). \label{eq:L2:2}
\end{align}
\end{lemma}
\begin{proof}
We will only verify \eqref{eq:L2:1} as the treatment of \eqref{eq:L2:2} is similar.  In view of Lemma \ref{lemma:mean0}, it suffices to show that
\begin{align}
 \E \Bigg| \sum_{j, l_2 < k}& x_{l_1 j} e_{j}^\mathrm{T} \E_{k-1} \left[ \mat{G}_{N,k}^{(l_1, j)}(z) \oindicator{\Omega_{N,k} \cap Q_{N,k}} \right] e_{l_2} \nonumber \\
	&\qquad\times e_{l_2}^\mathrm{T} \E_{k-1} \left[ \mat{G}_{N,k}^{(l_1, j)}(w) \oindicator{\Omega_{N,k} \cap Q_{N,k}} \right] e_{l_1} \Bigg|^2 = o(1) \label{eq:L2suffice}
\end{align}
uniformly in $k, l_1$.  We write the left-hand side of \eqref{eq:L2suffice} as
\begin{align}
	\E &\sum_{j, j' < k} x_{l_1 j} x_{l_1 j'} \label{eq:L2fullexpand} \\
	&\quad \times \sum_{l_2 < k} e_j^\mathrm{T} \E_{k-1} \left[ \mat{G}_{N,k}^{(l_1,j)}(z) \oindicator{\Omega_{N,k} \cap Q_{N,k}} \right] e_{l_2} e_{l_2}^\mathrm{T} \E_{k-1} \left[ \mat{G}_{N,k}^{(l_1,j)}(w) \oindicator{\Omega_{N,k} \cap Q_{N,k}} \right] e_{l_1} \nonumber \\
	&\quad \times \sum_{l_3 < k} e_{j'}^\mathrm{T} \E_{k-1} \left[ \overline{\mat{G}_{N,k}^{(l_1,j')}(z)} \oindicator{\Omega_{N,k} \cap Q_{N,k}} \right] e_{l_3} e_{l_3}^\mathrm{T} \E_{k-1} \left[ \overline{\mat{G}_{N,k}^{(l_1,j')}(w)} \oindicator{\Omega_{N,k} \cap Q_{N,k}} \right] e_{l_1}. \nonumber
\end{align}

We now claim that 
\begin{align}
	\sum_{j < k} \E \Bigg| \sum_{l_2 < k} &e_j^\mathrm{T} \E_{k-1} \left[ \mat{G}_{N,k}^{(l_1,j)}(z) \oindicator{\Omega_{N,k} \cap Q_{N,k}} \right] e_{l_2} \nonumber \\
	&\times e_{l_2}^\mathrm{T} \E_{k-1} \left[ \mat{G}_{N,k}^{(l_1,j)}(w) \oindicator{\Omega_{N,k} \cap Q_{N,k}} \right] e_{l_1} \Bigg|^2 = O(1) \label{eq:L2O1bnd}
\end{align}
uniformly in $k, l_1$.  Indeed, by the triangle inequality, we have 
\begin{align*}
\sum_{j < k} \E &\Bigg| \sum_{l_2 < k} e_j^\mathrm{T} \E_{k-1} \left[ \mat{G}_{N,k}^{(l_1,j)}(z) \oindicator{\Omega_{N,k} \cap Q_{N,k}} \right] e_{l_2} \\
	&\qquad \qquad \times e_{l_2}^\mathrm{T} \E_{k-1} \left[ \mat{G}_{N,k}^{(l_1,j)}(w) \oindicator{\Omega_{N,k} \cap Q_{N,k}} \right] e_{l_1} \Bigg|^2 \\
	&\leq \frac{C}{c^2} \sum_{j < k} \E \left\| \mat{G}_{N,k}^{(l_1,j)}(z) - \mat{G}_{N,k}(z) \right\|^2 \oindicator{\Omega_{N,k} \cap Q_{N,k}} \\
	&\qquad + \frac{C}{c^2} \sum_{j < k} \E \left\| \mat{G}_{N,k}^{(l_1,j)}(w) - \mat{G}_{N,k}(w) \right\|^2 \oindicator{\Omega_{N,k} \cap Q_{N,k}} \\
	&\qquad + C \E \sum_{j < k} \Bigg| \sum_{l_2 < k} e_j^\mathrm{T} \E_{k-1} \left[ \mat{G}_{N,k}(z) \oindicator{\Omega_{N,k} \cap Q_{N,k}} \right] e_{l_2} \\
	&\qquad \qquad \qquad \qquad \times e_{l_2}^\mathrm{T} \E_{k-1} \left[ \mat{G}_{N,k}(w) \oindicator{\Omega_{N,k} \cap Q_{N,k}} \right] e_{l_1} \Bigg|^2
\end{align*}
for some absolute constant $C > 0$.  The first two terms are $O(1)$ since
\begin{align} 
	\E \| &\mat{G}_{N,k}^{(l_1,j)}(z) - \mat{G}_{N,k}(z) \|^2 \oindicator{\Omega_{N,k} \cap Q_{N,k}} \nonumber \\
	&\leq \E \| \mat{G}_{N,k}(z) \|^2 \| \mat{G}_{N,k}^{(l_1,j)}(z) \|^2 \oindicator{\Omega_{N,k} \cap Q_{N,k}} (|x_{l_1 j}| + |x_{j l_1}|)^2 \nonumber \\
	&\leq \frac{8}{c^4} \left( \E |x_{l_1 j}|^2 + \E|x_{j l_1}|^2 \right) = O(N^{-1}) \label{eq:L21N}
\end{align}
uniformly in $k, l_1$.  The last term is controlled by observing that
$$ \sum_{l_2 < k} \E_{k-1} \left[ \mat{G}_{N,k}(z) \oindicator{\Omega_{N,k} \cap Q_{N,k}} \right] e_{l_2}  e_{l_2}^\mathrm{T} \E_{k-1} \left[ \mat{G}_{N,k}(w) \oindicator{\Omega_{N,k} \cap Q_{N,k}} \right] $$
is the product of the sub-matrices formed from the first $k-1$ columns of $\E_{k-1} \left[ \mat{G}_{N,k}(z) \oindicator{\Omega_{N,k} \cap Q_{N,k}} \right]$ and the first $k-1$ rows of $\E_{k-1} \left[ \mat{G}_{N,k}(w) \oindicator{\Omega_{N,k} \cap Q_{N,k}} \right]$.  Hence, 
\begin{align*}
	\E \sum_{j < k} &\left| \sum_{l_2 < k} e_j^\mathrm{T} \E_{k-1} \left[ \mat{G}_{N,k}(z) \oindicator{\Omega_{N,k} \cap Q_{N,k}} \right] e_{l_2}  e_{l_2}^\mathrm{T} \E_{k-1} \left[ \mat{G}_{N,k}(w) \oindicator{\Omega_{N,k} \cap Q_{N,k}} \right] e_{l_1} \right|^2 \\
	&\leq \E \left\| \sum_{l_2 < k} \E_{k-1} \left[ \mat{G}_{N,k}(z) \oindicator{\Omega_{N,k} \cap Q_{N,k}} \right] e_{l_2}  e_{l_2}^\mathrm{T} \E_{k-1} \left[ \mat{G}_{N,k}(w) \oindicator{\Omega_{N,k} \cap Q_{N,k}} \right] e_{l_1}  \right\|^2 \\
	&\leq \frac{1}{c^4},  
\end{align*}
and the proof of \eqref{eq:L2O1bnd} is complete.  

We now return to \eqref{eq:L2fullexpand}.  We first consider the diagonal terms ($j = j'$).  In this case, applying \eqref{eq:L2O1bnd}, we obtain
\begin{align*}
	\E &\sum_{j < k} |x_{l_1 j}|^2 \Bigg| \sum_{l_2 < k} e_j^\mathrm{T} \E_{k-1} \left[ \mat{G}_{N,k}^{(l_1,j)}(z) \oindicator{\Omega_{N,k} \cap Q_{N,k}} \right] e_{l_2} \\
	&\qquad\qquad\qquad\qquad \times e_{l_2}^\mathrm{T} \E_{k-1} \left[ \mat{G}_{N,k}^{(l_1,j)}(w) \oindicator{\Omega_{N,k} \cap Q_{N,k}} \right] e_{l_1} \Bigg|^2 \\
	&\leq (4 \eps_N)^2 \sum_{j < k} \Bigg| \sum_{l_2 < k} e_j^\mathrm{T} \E_{k-1} \left[ \mat{G}_{N,k}^{(l_1,j)}(z) \oindicator{\Omega_{N,k} \cap Q_{N,k}} \right] e_{l_2} \\
	&\qquad\qquad\qquad\qquad \times e_{l_2}^\mathrm{T} \E_{k-1} \left[ \mat{G}_{N,k}^{(l_1,j)}(w) \oindicator{\Omega_{N,k} \cap Q_{N,k}} \right] e_{l_1} \Bigg|^2 \\
	&= o(1)
\end{align*}
uniformly in $k, l_1$.  

We now consider the cross-terms ($j \neq j'$):
\begin{align}
\E &\sum_{\substack{j, j' < k \\ j \neq j'}} x_{l_1 j} x_{l_1 j'} \label{eq:L2crossterms} \\
	&\quad \times \sum_{l_2 < k} e_j^\mathrm{T} \E_{k-1} \left[ \mat{G}_{N,k}^{(l_1,j)}(z) \oindicator{\Omega_{N,k} \cap Q_{N,k}} \right] e_{l_2} e_{l_2}^\mathrm{T} \E_{k-1} \left[ \mat{G}_{N,k}^{(l_1,j)}(w) \oindicator{\Omega_{N,k} \cap Q_{N,k}} \right] e_{l_1} \nonumber \\
	&\quad \times \sum_{l_3 < k} e_{j'}^\mathrm{T} \E_{k-1} \left[ \overline{\mat{G}_{N,k}^{(l_1,j')}(z)} \oindicator{\Omega_{N,k} \cap Q_{N,k}} \right] e_{l_3} e_{l_3}^\mathrm{T} \E_{k-1} \left[ \overline{\mat{G}_{N,k}^{(l_1,j')}(w)} \oindicator{\Omega_{N,k} \cap Q_{N,k}} \right] e_{l_1}. \nonumber
\end{align}
For this case, we define 
$$ \mat{X}_{N,k}^{(l_1,j),(l_1,j')} := \mat{X}_{N,k}^{(l_1,j)} -  \tau_{l_1 j'} \left( x_{l_1 j'} e_{l_1} e_{j'}^\mathrm{T} + x_{j' l_1} e_{j'} e_{l_1}^\mathrm{T} \right) $$
and
$$ \mat{G}_{N,k}^{(l_1,j),(l_1,j')}(z) := \left( X_{N,k}^{(l_1,j),(l_1,j')} - z \mat{I} \right)^{-1}. $$
Using Lemma \ref{lemma:ijbnd}, it follows that, for all $j,j' < k$,
$$ \| \mat{G}_{N,k}^{(l_1,j),(l_1,j')}(z) \|, \| \mat{G}_{N,k}^{(l_1,j),(l_1,j')}(w) \| \leq \frac{2}{c} $$
on the event $\Omega_{N,k}$ for $N$ sufficiently large.  We also define the event
$$ \Omega_{N,k}^{(l_1,j),(l_1,j')} := \left\{ \| \mat{G}_{N,k}^{(l_1,j),(l_1,j')}(z) \| \leq \frac{2}{c}, \| \mat{G}_{N,k}^{(l_1,j),(l_1,j')}(w) \| \leq \frac{2}{c} \right\}. $$
We observe that if we can replace each occurrence of $\mat{G}_{N,k}^{(l_1,j)}(z)$, $\mat{G}_{N,k}^{(l_1,j)}(w)$, $\mat{G}_{N,k}^{(l_1,j')}(z)$, $\mat{G}_{N,k}^{(l_1,j')}(w)$ in \eqref{eq:L2crossterms} with the corresponding matrix $\mat{G}_{N,k}^{(l_1,j),(l_1,j')}(z)$, $\mat{G}_{N,k}^{(l_1,j),(l_1,j')}(w)$ and if we can also replace each occurrence of $\oindicator{\Omega_{N,k} \cap Q_{N,k}}$ with the corresponding indicator function $\oindicator{\Omega_{N,k}^{(l_1,j),(l_1,j')}}$, we would obtain 
\begin{align}
\E &\sum_{\substack{j, j' < k \\ j \neq j'}} x_{l_1 j} x_{l_1 j'} \sum_{l_2 < k} e_j^\mathrm{T} \E_{k-1} \left[ \mat{G}_{N,k}^{(l_1,j),(l_1,j')}(z) \oindicator{\Omega_{N,k}^{(l_1,j),(l_1,j')}} \right] e_{l_2} \label{eq:L2zerocontrib} \\
	&\qquad \qquad \qquad \qquad \times e_{l_2}^\mathrm{T} \E_{k-1} \left[ \mat{G}_{N,k}^{(l_1,j),(l_1,j')}(w) \oindicator{\Omega_{N,k}^{(l_1,j),(l_1,j')}} \right] e_{l_1} \nonumber \\
	&\qquad \times \sum_{l_3 < k} e_{j'}^\mathrm{T} \E_{k-1} \left[ \overline{\mat{G}_{N,k}^{(l_1,j),(l_1,j')}(z)} \oindicator{\Omega_{N,k}^{(l_1,j),(l_1,j')}} \right] e_{l_3} \nonumber \\
	&\qquad \qquad \qquad \qquad \times e_{l_3}^\mathrm{T} \E_{k-1} \left[ \overline{\mat{G}_{N,k}^{(l_1,j),(l_1,j')}(w)} \oindicator{\Omega_{N,k}^{(l_1,j),(l_1,j')}} \right] e_{l_1} \nonumber \\
	&= 0 \nonumber
\end{align}
by independence.  Thus, it suffices to show that such replacements are possible. 

We begin by showing we can replace each occurrence of $\mat{G}_{N,k}^{(l_1,j)}(z)$, $\mat{G}_{N,k}^{(l_1,j)}(w)$, $\mat{G}_{N,k}^{(l_1,j')}(z)$, $\mat{G}_{N,k}^{(l_1,j')}(w)$ in \eqref{eq:L2crossterms} with the corresponding matrix $\mat{G}_{N,k}^{(l_1,j),(l_1,j')}(z)$, $\mat{G}_{N,k}^{(l_1,j),(l_1,j')}(w)$ by applying the substitution  
$$ \mat{G}_{N,k}^{(l_1,j)}(z) = \mat{G}_{N,k}^{(l_1,j),(l_1,j')}(z) - \mat{G}_{N,k}^{(l_1,j),(l_1,j')}(z) \tau_{l_1 j'} ( x_{l_1 j'} e_{l_1} e_{j'}^{\mathrm{T}} + x_{j' l_1} e_{j'} e_{l_1}^\mathrm{T}) \mat{G}_{N,k}^{(l_1,j)}(z). $$
Indeed, the difference caused by this replacement can be divided into several error terms.  The treatment of each term is similar.  As an illustration of their treatment, the first error term is bounded by 
\begin{align*}
	\E \sum_{\substack{j,j' < k \\ j \neq j'}} &|x_{l_1 j}| |x_{l_1 j'}|^2 \tau_{l_1 j'} \Bigg| \sum_{l_2 < k} e_j^\mathrm{T} \E_{k-1} \left[ \mat{G}_{N,k}^{(l_1,j),(l_1,j')}(z) \oindicator{\Omega_{N,k} \cap Q_{N,k}} e_{l_1} e_{j'}^\mathrm{T} \mat{G}_{N,k}^{(l_1,j)}(z) \right] e_{l_2} \\
	&\qquad \qquad \qquad \qquad \times e_{l_2}^\mathrm{T} \E_{k-1} \left[ \mat{G}_{N,k}^{(l_1,j)}(w) \oindicator{\Omega_{N,k} \cap Q_{N,k}} \right] e_{l_1} \Bigg| \\
	&\qquad \qquad \times \Bigg| \sum_{l_3 < k} e_{j'}^\mathrm{T} \E_{k-1} \left[ \mat{G}_{N,k}^{(l_1,j')}(z) \oindicator{\Omega_{N,k} \cap Q_{N,k}} \right] e_{l_3} \\
	&\qquad \qquad \qquad \qquad \times e_{l_3}^\mathrm{T} \E_{k-1} \left[ \mat{G}_{N,k}^{(l_1,j')}(w) \oindicator{\Omega_{N,k} \cap Q_{N,k}} \right] e_{l_1} \Bigg| \\
	&\leq \frac{4}{c^2} \E \sum_{\substack{j,j' < k \\ j \neq j'}} |x_{l_1 j}| |x_{l_1 j'}|^2 \left| e_j^\mathrm{T} \mat{G}_{N,k}^{(l_1,j),(l_1,j')}(z) \oindicator{\Omega_{N,k} \cap Q_{N,k}} e_{l_1} \right| \\
	&\qquad \qquad \times \Bigg| \sum_{l_3 < k} e_{j'}^\mathrm{T} \E_{k-1} \left[ \mat{G}_{N,k}^{(l_1,j')}(z) \oindicator{\Omega_{N,k} \cap Q_{N,k}} \right] e_{l_3} \\
	&\qquad \qquad \qquad \qquad \times e_{l_3}^\mathrm{T} \E_{k-1} \left[ \mat{G}_{N,k}^{(l_1,j')}(w) \oindicator{\Omega_{N,k} \cap Q_{N,k}} \right] e_{l_1} \Bigg|.
\end{align*}
Thus, by the Cauchy-Schwarz inequality, it suffices to show that
\begin{align} \label{eq:L2suffice1}
	\sum_{\substack{j,j' < k \\ j \neq j'}} \E |x_{l_1 j}|^2 \left| e_j^\mathrm{T} \mat{G}_{N,k}^{(l_1,j),(l_1,j')}(z) \oindicator{\Omega_{N,k} \cap Q_{N,k}} e_{l_1} \right|^2 = O(1) 
\end{align}
and
\begin{align}
	\sum_{j,j' < k} \E|x_{l_1 j'}|^4 \Bigg| \sum_{l_3 < k} &e_{j'}^\mathrm{T} \E_{k-1} \left[ \mat{G}_{N,k}^{(l_1,j')}(z) \oindicator{\Omega_{N,k} \cap Q_{N,k}} \right] e_{l_3} \label{eq:L2suffice2} \\
	&\qquad \times e_{l_3}^\mathrm{T} \E_{k-1} \left[ \mat{G}_{N,k}^{(l_1,j')}(w) \oindicator{\Omega_{N,k} \cap Q_{N,k}} \right] e_{l_1} \Bigg|^2 = o(1) \nonumber
\end{align}
uniformly in $k, l_1$.  

We first consider \eqref{eq:L2suffice1}.  Using the same technique as in the proof of Lemma \ref{lemma:mean0}, we can replace $\oindicator{\Omega_{N,k} \cap Q_{N,k}}$ with $\oindicator{\Omega_{N,k}^{(l_1,j)}}$, where
$$ \Omega_{N,k}^{(l_1,j)} := \left\{ \| \mat{G}_{N,k}^{(l_1,j)}(z) \| \leq \frac{2}{c}, \| \mat{G}_{N,k}^{(l_1,j)}(w) \| \leq \frac{2}{c} \right\}. $$
Thus, the left-hand side of \eqref{eq:L2suffice1} is bounded by
\begin{align} \label{eq:L2suffice1help}
	\frac{1+\sigma^2}{N} \sum_{\substack{j,j' < k \\ j \neq j'}} \E &\left| e_j^\mathrm{T} \mat{G}_{N,k}^{(l_1,j),(l_1,j')}(z) e_{l_1} \oindicator{\Omega_{N,k}^{(l_1,j)}} \right|^2 + o(1).
\end{align}
Passing back to $\oindicator{\Omega_{N,k} \cap Q_{N,k}}$, we find that, for some absolute constant $C>0$, \eqref{eq:L2suffice1help} is bounded by
\begin{align*}
	C&\frac{1+\sigma^2}{N} \sum_{j,j' < k} \E \left| e_j^\mathrm{T} \left( \mat{G}_{N,k}^{(l_1,j),(l_1,j')}(z) - \mat{G}_{N,k}^{(l_1,j)}(z) \right) e_{l_1} \oindicator{\Omega_{N,k} \cap Q_{N,k}} \right|^2 \\
	&\quad + C\frac{1+\sigma^2}{N} \sum_{j,j' < k} \E \left| e_j^\mathrm{T} \left( \mat{G}_{N,k}^{(l_1,j)}(z) - \mat{G}_{N,k}(z) \right) e_{l_1}\oindicator{\Omega_{N,k} \cap Q_{N,k}} \right|^2 \\
	&\quad + C \frac{1+\sigma^2}{N} \sum_{j,j' < k} \E \left| e_j^\mathrm{T} \mat{G}_{N,k}(z) e_{l_1} \oindicator{\Omega_{N,k} \cap Q_{N,k}} \right|^2 + o(1) \\
	&\leq C\frac{1+\sigma^2}{N} \sum_{j,j' < k} \E \left\| \mat{G}_{N,k}^{(l_1,j),(l_1,j')}(z) - \mat{G}_{N,k}^{(l_1,j)}(z) \right\|^2 \oindicator{\Omega_{N,k} \cap Q_{N,k}}  \\
	&\quad + C\frac{1+\sigma^2}{N} \sum_{j,j' < k} \E \left\|\mat{G}_{N,k}^{(l_1,j)}(z) - \mat{G}_{N,k}(z) \right\|^2 \oindicator{\Omega_{N,k} \cap Q_{N,k}} \\
	&\quad + C \frac{1+\sigma^2}{N} \sum_{j,j' < k} \E \left| e_j^\mathrm{T} \mat{G}_{N,k}(z) e_{l_1} \oindicator{\Omega_{N,k} \cap Q_{N,k}} \right|^2 + o(1) \\
	&= O(1)
\end{align*}
uniformly in $k, l_1$ by \eqref{eq:L21N}.  

For \eqref{eq:L2suffice2}, we use the same technique to replace $\oindicator{\Omega_{N,k} \cap Q_{N,k}}$ with $\oindicator{\Omega_{N,k}^{(l_1,j')}}$ (and then pass back to $\oindicator{\Omega_{N,k} \cap Q_{N,k}}$).  Thus, the left-hand side of \eqref{eq:L2suffice2} is bounded by 
\begin{align*}
	\frac{C}{N^2} \sum_{j,j' < k} \Bigg| \sum_{l_3 < k} &e_{j'}^\mathrm{T} \E_{k-1} \left[ \mat{G}_{N,k}^{(l_1,j')}(z) \oindicator{\Omega_{N,k} \cap Q_{N,k}} \right] e_{l_3} \\
	&\qquad \times e_{l_3}^\mathrm{T} \E_{k-1} \left[ \mat{G}_{N,k}^{(l_1,j')}(w) \oindicator{\Omega_{N,k} \cap Q_{N,k}} \right] e_{l_1} \Bigg|^2 + o(1),
\end{align*}
where $C > 0$ depends only on the fourth moments of the atom variables $\xi_1, \xi_2$.  Thus, \eqref{eq:L2suffice2} now follows from \eqref{eq:L2O1bnd}.  

Lastly we observe that by using the same technique as in the proof of Lemma \ref{lemma:mean0}, one can replace each occurrence of $\oindicator{\Omega_{N,k} \cap Q_{N,k}}$ with the corresponding indicator function $\oindicator{\Omega_{N,k}^{(l_1,j),(l_1,j')}}$.  Since the error caused by these replacements is $o(1)$ uniformly in $k, l_1$, the proof of the lemma is complete.  
\end{proof}

\begin{lemma} \label{lemma:L22}
One has
\begin{align}
	\sup_{1 \leq k \leq N} \sup_{l_1 < k} \E \Bigg| \sum_{j, l_2 < k} &\tau_{l_1, j} \left( x_{l_1 j} x_{j l_1} - \frac{\rho}{N} \right) e_j^\mathrm{T} E_{k-1} \left[ \mat{G}_{N,k}(z) \oindicator{\Omega_{N,k} \cap Q_{N,k}} \right] e_{l_2} \nonumber \\
	&\qquad\times e_{l_2}^\mathrm{T} \E_{k-1} \left[ \mat{G}_{N,k}(w) \oindicator{\Omega_{N,k} \cap Q_{N,k}} \right] e_j \Bigg|^2 = o(1) \label{eq:L22:1}
\end{align}
and
\begin{align}
\sup_{1 \leq k \leq N} \sup_{l_1 < k} \E \Bigg| \sum_{j, l_2 < k} &\tau_{l_1, j} \left( x_{l_1 j}^2 - \frac{1}{N} \right) e_j^\mathrm{T} E_{k-1} \left[ \mat{G}_{N,k}(z) \oindicator{\Omega_{N,k} \cap Q_{N,k}} \right] e_{l_2} \nonumber \\
	&\qquad\times e_{j}^\mathrm{T} \E_{k-1} \left[ \mat{G}_{N,k}(w) \oindicator{\Omega_{N,k} \cap Q_{N,k}} \right] e_{l_2} \Bigg|^2 = o(1). \label{eq:L22:2}
\end{align}
\end{lemma}
\begin{proof}
We will only verify \eqref{eq:L22:1} as the treatment of \eqref{eq:L22:2} is similar.  Indeed, we express the left-hand side of \eqref{eq:L22:1} (without the suprema) as
\begin{align*}
	\E &\sum_{j,j' < k} \tau_{l_1 j} \tau_{l_1 j'} \left( x_{l_1 j} x_{j l_1} - \frac{\rho}{N} \right) \left( x_{l_1 j'} x_{j' l_1} - \frac{\rho}{N} \right) \\
	& \quad \times \sum_{l_2 < k} e_j^\mathrm{T} \E_{k-1} \left[ \mat{G}_{N,k}(z) \oindicator{\Omega_{N,k} \cap Q_{N,k}}  \right] e_{l_2} e_{l_2}^\mathrm{T} e_{k-1} \left[ \mat{G}_{N,k}(w) \oindicator{\Omega_{N,k} \cap Q_{N,k}} \right] e_j \\
	& \quad \times \sum_{l_3 < k} e_{j'}^\mathrm{T} \E_{k-1} \left[ \overline{\mat{G}_{N,k}(z)} \oindicator{\Omega_{N,k} \cap Q_{N,k}} \right] e_{l_3} e_{l_3}^\mathrm{T} \E_{k-1} \left[ \overline{\mat{G}_{N,k}(w)} \oindicator{\Omega_{N,k} \cap Q_{N,k}} \right] e_{j'}.
\end{align*}

We first consider the diagonal terms ($j = j'$).  Indeed, in this case we obtain
\begin{align*}
	\E \sum_{j < k} &\tau_{l_1 j}^2 \left|x_{l_1 j} x_{j l_1} - \frac{\rho}{N} \right|^2 \Bigg| \sum_{l_2 < k} e_j^{\mathrm{T}} \E_{k-1} \left[\mat{G}_{N,k}(z) \oindicator{\Omega_{N,k} \cap Q_{N,k}} \right] e_{l_2} \\
	&\qquad\qquad\qquad\qquad\qquad \times e_{l_2}^\mathrm{T} \E_{k-1} \left[ \mat{G}_{N,k}(w) \oindicator{\Omega_{N,k} \cap Q_{N,k}} \right] e_j \Bigg|^2 \\
	&\leq \frac{1}{c^4} \E \sum_{j < k} \left|x_{l_1 j} x_{j l_1} - \frac{\rho}{N} \right|^2 = o(1)
\end{align*}
uniformly in $k, l_1$.  

We now consider the cross-terms ($j \neq j'$):
\begin{align}
	\E &\sum_{\substack{j,j' < k \\j \neq j'}} \tau_{l_1 j} \tau_{l_1 j'} \left( x_{l_1 j} x_{j l_1} - \frac{\rho}{N} \right) \left( x_{l_1 j'} x_{j' l_1} - \frac{\rho}{N} \right) \label{eq:crossterms} \\
	& \quad\times \sum_{l_2 < k} e_j^\mathrm{T} \E_{k-1} \left[ \mat{G}_{N,k}(z) \oindicator{\Omega_{N,k} \cap Q_{N,k}}  \right] e_{l_2} e_{l_2}^\mathrm{T} e_{k-1} \left[ \mat{G}_{N,k}(w) \oindicator{\Omega_{N,k} \cap Q_{N,k}} \right] e_j \nonumber \\
	& \quad\times \sum_{l_3 < k} e_{j'}^\mathrm{T} \E_{k-1} \left[ \overline{\mat{G}_{N,k}(z)} \oindicator{\Omega_{N,k} \cap Q_{N,k}} \right] e_{l_3} e_{l_3}^\mathrm{T} \E_{k-1} \left[ \overline{\mat{G}_{N,k}(w)} \oindicator{\Omega_{N,k} \cap Q_{N,k}} \right] e_{j'}. \nonumber
\end{align}
For this case, we define 
$$ \mat{X}_{N,k}^{(l_1,j),(l_1,j')} := \mat{X}_{N,k}^{(l_1,j)} -  \tau_{l_1 j'} \left( x_{l_1 j'} e_{l_1} e_{j'}^\mathrm{T} + x_{j' l_1} e_{j'} e_{l_1}^\mathrm{T} \right) $$
and
$$ \mat{G}_{N,k}^{(l_1,j),(l_1,j')}(z) := \left( X_{N,k}^{(l_1,j),(l_1,j')} - z \mat{I} \right)^{-1}. $$
Using Lemma \ref{lemma:ijbnd}, it follows that, for all $j,j' < k$,
$$ \| \mat{G}_{N,k}^{(l_1,j),(l_1,j')}(z) \|, \| \mat{G}_{N,k}^{(l_1,j),(l_1,j')}(w) \| \leq \frac{2}{c} $$
on the event $\Omega_{N,k}$ for $N$ sufficiently large.  

Since 
$$ \sum_{j,j' < k} \E \left| x_{l_1 j} x_{j l_1} - \frac{\rho}{N} \right| \left| x_{l_1 j'} x_{j' l_1} - \frac{\rho}{N} \right| = O(1) $$
uniformly in $k,l_1$, we can apply the deterministic bounds
$$ \| \mat{G}_{N,k}(z) - \mat{G}_{N,k}^{(l_1,j)}(z) \| \leq \frac{16}{c^2} \eps_N, \quad \| \mat{G}_{N,k}(z)^{(l_1,j)} - \mat{G}_{N,k}^{(l_1,j),(l_1,j')}(z) \| \leq \frac{32}{c^2} \eps_N, $$
which hold on $\Omega_{N,k}$, to replace each occurrence of $\mat{G}_{N,k}(z)$ and $\mat{G}_{N,k}(w)$ in \eqref{eq:crossterms} with $\mat{G}_{N,k}^{(l_1,j),(l_1,j')}(z)$ and $\mat{G}_{N,k}^{(l_1,j),(l_1,j')}(w)$.  Moreover, by the same technique used in the proof of Lemma \ref{lemma:mean0}, we can replace $\oindicator{\Omega_{N,k} \cap Q_{N,k}}$ with $\oindicator{\Omega_{N,k}^{(l_1,j),(l_1,j')}}$, where
$$ \Omega_{N,k}^{(l_1,j),(l_1,j')} := \left\{ \| \mat{G}_{N,k}^{(l_1,j),(l_1,j')}(z) \| \leq \frac{2}{c}, \| \mat{G}_{N,k}^{(l_1,j),(l_1,j')}(w) \| \leq \frac{2}{c} \right\}. $$

After these replacements, \eqref{eq:crossterms} reduces to
\begin{align*}
\E &\sum_{\substack{j,j' < k \\ j \neq j'}} \tau_{l_1 j} \tau_{l_1 j'} \left( x_{l_1 j} x_{j l_1} - \frac{\rho}{N} \right) \left( x_{l_1 j'} x_{j' l_1} - \frac{\rho}{N} \right) \\
	& \times \sum_{l_2 < k} e_j^\mathrm{T} \E_{k-1} \left[ \mat{G}_{N,k}^{(l_1,j),(l_1,j')}(z) \oindicator{\Omega_{N,k}^{(l_1,j),(l_1,j')}}  \right] e_{l_2} \\
	&\qquad\qquad \times e_{l_2}^\mathrm{T} e_{k-1} \left[ \mat{G}_{N,k}^{(l_1,j),(l_1,j')}(w) \oindicator{\Omega_{N,k}^{(l_1,j),(l_1,j')}} \right] e_j \\
	& \times \sum_{l_3 < k} e_{j'}^\mathrm{T} \E_{k-1} \left[ \overline{\mat{G}_{N,k}^{(l_1,j),(l_1,j')}(z)} \oindicator{\Omega_{N,k}^{(l_1,j),(l_1,j')}} \right] e_{l_3} \\
	&\qquad\qquad \times e_{l_3}^\mathrm{T} \E_{k-1} \left[ \overline{\mat{G}_{N,k}^{(l_1,j),(l_1,j')}(w)} \oindicator{\Omega_{N,k}^{(l_1,j),(l_1,j')}}\right] e_{j'} + o(1).
\end{align*}
As 
$$ \E  \left( x_{l_1 j} x_{j l_1} - \frac{\rho}{N} \right) \left( x_{l_1 j'} x_{j' l_1} - \frac{\rho}{N} \right)  = 0 $$
for $j \neq j'$, by independence, we find that the contribution from \eqref{eq:crossterms} is $o(1)$ uniformly in $k,l_1$.  
\end{proof}

We now consider the terms involving $\mat{A}_{N,k}(z)$.  Indeed, by \eqref{eq:resolvdiff} and Lemma \ref{lemma:L2}, we write
\begin{align*}
	\sum_{l_2 < k} &e_{l_1}^\mathrm{T} \E_{k-1} \left[ \mat{A}_{N,k}(z) \oindicator{\Omega_{N,k} \cap Q_{N,k}} \right] e_{l_2} e_{l_2}^\mathrm{T} \E_{k-1} \left[ \mat{G}_{N,k}(w) \oindicator{\Omega_{N,k} \cap Q_{N,k}}\right] e_{l_1} \\
	&= \sum_{j \neq k} \sum_{l_2 < k} e_j^\mathrm{T} \E_{k-1} \left[ \mat{G}_{N,k}^{(l_1,j)}(z) x_{l_1 j} \oindicator{\Omega_{N,k} \cap Q_{N,k}} \right] e_{l_2} \\
	&\qquad \times e_{l_2}^\mathrm{T} \E_{k-1} \left[ \left( \mat{G}_{N,k}(w) - \mat{G}_{N,k}^{(l_1, j)}(w) \right) \oindicator{\Omega_{N,k} \cap Q_{N,k}} \right] e_{l_1} + o(1) \\
	&= -\sum_{j \neq k} \sum_{l_2 < k} e_j^\mathrm{T} \E_{k-1} \left[ \mat{G}_{N,k}^{(l_1,j)}(z) x_{l_1 j} \oindicator{\Omega_{N,k} \cap Q_{N,k}} \right] e_{l_2} \\
	&\qquad \times e_{l_2}^{\mathrm{T}} \E_{k-1} \left[ x_{l_1 j} \tau_{l_1 j} \mat{G}_{N,k}^{(l_1,j)}(w) e_{l_1} e_{j}^\mathrm{T} \mat{G}_{N,k}(w) \oindicator{\Omega_{N,k} \cap Q_{N,k}} \right] e_{l_1} \\
	&\quad - \sum_{j \neq k} \sum_{l_2 < k} e_j^\mathrm{T} \E_{k-1} \left[ \mat{G}_{N,k}^{(l_1,j)}(z) x_{l_1 j} \oindicator{\Omega_{N,k} \cap Q_{N,k}} \right] e_{l_2} \\
	&\qquad \times e_{l_2}^\mathrm{T} \E_{k-1} \left[ x_{j l_1} \tau_{l_1 j} \mat{G}_{N,k}^{(l_1,j)}(w) e_j e_{l_1}^\mathrm{T} \mat{G}_{N,k}(w) \oindicator{\Omega_{N,k} \cap Q_{N,k}} \right] e_{l_1} + o(1) \\
	&=: - \Psi_{N,k,1}(z,w) - \Psi_{N,k,2}(z,w) + o(1),
\end{align*}
where $o(1)$ denotes a term which tends to zero in $L_2$ as $N \to \infty$.  

By Lemma \ref{lemma:mean0} and the Cauchy-Schwarz inequality, we have
\begin{align*}
	\E &| \Psi_{N,k,1}(z,w)| \\
	&\leq \E \Bigg| \sum_{j,l_2 < k} \tau_{l_1 j} x_{l_1 j}^2 e_j \E_{k-1} \left[ \mat{G}_{N,k}^{(l_1,j)}(z) \oindicator{\Omega_{N,k} \cap Q_{N,k}} \right] e_{l_2} \\
	&\qquad\qquad \times e_{l_2}^\mathrm{T} \mat{G}_{N,k}^{(l_1,j)}(w) e_{l_1} e_j^\mathrm{T} \mat{G}_{N,k}(w) e_{l_1} \oindicator{\Omega_{N,k} \cap Q_{N,k}} \Bigg| + O(N^{-1}) \\
	&\leq \Bigg( \E \sum_{j < k} |x_{l_1 j}|^4 \Bigg| \sum_{l_2 < k} e_j^\mathrm{T} \E_{k-1} \left[ \mat{G}_{N,k}^{(l_1,j)}(z) \oindicator{\Omega_{N,k} \cap Q_{N,k}} \right] e_{l_2} \\
	&\qquad\qquad \times e_{l_2}^\mathrm{T} \mat{G}_{N,k}^{(l_1,j)}(w) e_{l_1} \Bigg|^2 \oindicator{\Omega_{N,k} \cap Q_{N,k}} \Bigg)^{1/2} \\
	&\qquad \times \left( \E \sum_{j < k} \left| e_j^\mathrm{T} \mat{G}_{N,k}(w) e_{l_1} \right|^2 \oindicator{\Omega_{N,k} \cap Q_{N,k}} \right)^{1/2} + O(N^{-1}) \\
	&\leq \frac{(4 \eps_N)^2}{c} \Bigg( \E \sum_{j < k} \Bigg| \sum_{l_2 < k} e_j^\mathrm{T} \E_{k-1} \left[ \mat{G}_{N,k}^{(l_1,j)}(z) \oindicator{\Omega_{N,k} \cap Q_{N,k}} \right] e_{l_2} \\
	&\qquad\qquad\qquad\qquad \times e_{l_2}^\mathrm{T} \mat{G}_{N,k}^{(l_1,j)}(w) e_{l_1} \Bigg|^2 \oindicator{\Omega_{N,k} \cap Q_{N,k}} \Bigg)^{1/2} + O(N^{-1}) \\
	&= o(1)
\end{align*}
uniformly in $k, l_1$.  Here we used that
$$ \E \sum_{j < k} \Bigg| \sum_{l_2 < k} e_j^\mathrm{T} \E_{k-1} \left[ \mat{G}_{N,k}^{(l_1,j)}(z) \oindicator{\Omega_{N,k} \cap Q_{N,k}} \right] e_{l_2} e_{l_2}^\mathrm{T} \mat{G}_{N,k}^{(l_1,j)}(w) e_{l_1} \Bigg|^2 \oindicator{\Omega_{N,k} \cap Q_{N,k}} = O(1). $$
This follows by replacing $\mat{G}_{N,k}^{(l_1,j)}(z)$ and $\mat{G}_{N,k}^{(l_1,j)}(w)$ with $\mat{G}_{N,k}(z)$ and $\mat{G}_{N,k}(w)$; indeed, it can be proven in the same way as in the proof of \eqref{eq:L2O1bnd}.  

We now consider $\Psi_{N,k,2}(z,w)$.  By Lemma \ref{lemma:mean0}, we only need to consider
\begin{align*}
	\sum_{l_2,j < k} \tau_{l_1,j} x_{l_1 j} x_{j l1} &e_j^\mathrm{T} \E_{k-1} \left[ \mat{G}_{N,k}^{(l_1,j)}(z) \oindicator{\Omega_{N,k} \cap Q_{N,k}} \right] e_{l_2} \\
	&\qquad \times e_{l_2}^\mathrm{T} \E_{k-1} \left[ \mat{G}_{N,k}^{(l_1,j)}(w) \oindicator{\Omega_{N,k} \cap Q_{N,k}} e_j e_{l_1}^\mathrm{T} \mat{G}_{N,k}(w) e_{l_1} \right]. 
\end{align*}
In view of Lemma \ref{lemma:diagonal} and the fact that 
\begin{equation} \label{eq:mombnduse}
	\sqrt{\E|x_{l_1 j} x_{j l_1}|^2 } = O(N^{-1}),
\end{equation}
this reduces to studying
\begin{align*}
	m(w) \sum_{j,l_2 < k} \tau_{l_1 j} x_{l_1 j} x_{j l_1} &e_j^\mathrm{T} \E_{k-1} \left[ \mat{G}_{N,k}^{(l_1,j)}(z) \oindicator{\Omega_{N,k} \cap Q_{N,k}} \right] e_{l_2} \\
	&\qquad \times e_{l_2}^\mathrm{T} \E_{k-1} \left[ \mat{G}_{N,k}^{(l_1,j)}(w) \oindicator{\Omega_{N,k} \cap Q_{N,k}} \right] e_j.
\end{align*}

From \eqref{eq:resolvdiff} and Lemma \ref{lemma:ijbnd}, we have the deterministic bound
$$ \| \mat{G}_{N,k}(z) - \mat{G}_{N,k}^{(l_1,j)}(z) \| \leq \frac{16}{c^2} \eps_N $$
on the event $\Omega_{N,k}$.  Thus, by the bound above and \eqref{eq:mombnduse}, the only non-negligible contribution from $\Psi_{N,k,2}(z,w)$ is given by
\begin{align*}
	m(w) \sum_{j,l_2 < k} \tau_{l_1 j} x_{l_1 j} x_{j l_1} &e_j \E_{k-1} \left[ \mat{G}_{N,k}(z) \oindicator{\Omega_{N,k} \cap Q_{N,k}} \right] e_{l_2}\\
	&\qquad \times e_{l_2}^\mathrm{T} \E_{k-1} \left[ \mat{G}_{N,k}(w) \oindicator{\Omega_{N,k} \cap Q_{N,k}} \right] e_j.
\end{align*}
Therefore, by Lemma \ref{lemma:L22}, we conclude that the only non-negligible contribution from
$$ \sum_{l_2 < k} e_{l_1}^\mathrm{T} \E_{k-1} \left[ \mat{A}_{N,k}(z) \oindicator{\Omega_{N,k} \cap Q_{N,k}} \right] e_{l_2} e_{l_2}^\mathrm{T} \E_{k-1} \left[ \mat{G}_{N,k}(w) \oindicator{\Omega_{N,k} \cap Q_{N,k}}\right] e_{l_1} $$
is given by 
\begin{align*}
	- m(w) \frac{\rho}{N} \sum_{j, l_2 < k} \tau_{l_1 j} b_{k j l_2}(z) b_{k l_2 j}(w).
\end{align*}
Hence,
\begin{align*}
	\frac{1}{N} \sum_{l_1, l_2 < k} &e_{l_1}^\mathrm{T} \E_{k-1} \left[ \mat{A}_{N,k}(z) \oindicator{\Omega_{N,k} \cap Q_{N,k}} \right] e_{l_2} e_{l_2}^\mathrm{T} \E_{k-1} \left[ \mat{G}_{N,k}(w) \oindicator{\Omega_{N,k} \cap Q_{N,k}}\right] e_{l_1} \\
	&= - m(w) \frac{(k-1)}{N} \frac{\rho}{N} \sum_{j, l_2 < k} b_{k j l_2}(z) b_{k l_2 j}(w) + o(1),
\end{align*}
where $o(1)$ denotes a term which tends to zero in probability (uniformly in $k$) as $N \to \infty$.  Similarly, using \eqref{eq:L2:2} and \eqref{eq:L22:2} from Lemmas \ref{lemma:L2} and \ref{lemma:L22}, we find that
\begin{align*}
	\frac{1}{N} \sum_{l_1, l_2 < k} &e_{l_1}^\mathrm{T} \E_{k-1} \left[ \mat{A}_{N,k}(z) \oindicator{\Omega_{N,k} \cap Q_{N,k}} \right] e_{l_2} e_{l_1}^\mathrm{T} \E_{k-1} \left[ \mat{G}_{N,k}(w) \oindicator{\Omega_{N,k} \cap Q_{N,k}}\right] e_{l_2} \\
	&= - m(w) \frac{(k-1)}{N} \frac{1}{N} \sum_{j, l_2 < k} b_{k j l_2}(z) b_{k l_2 j}(w) + o(1).
\end{align*}

Combining the contributive components above, we obtain
\begin{align*}
	\frac{1}{N} &\left[ z + \rho m(z) \frac{(N - 3/2)}{N} \right] \sum_{l_1, l_2 < k} b_{k l_1 l_2}(z) b_{k l_2 l_1}(w) \\
	&= - m(w) \rho \frac{(k-1)}{N} \frac{1}{N} \sum_{l_1, l_2 < k} b_{k l_1 l_2}(z) b_{k l_2 l_1}(w) - m(w) \left( \frac{k-1}{N} \right) + o(1),
\end{align*}
where $o(1)$ denotes a term which tends to zero in probability (uniformly in $k$) as $N \to \infty$.  Since $z + \rho m(z) = \frac{-1}{m(z)}$, we obtain
\begin{align*}
	\frac{1}{N} &\sum_{l_1, l_2 < k} b_{k l_1 l_2}(z) b_{k l_2 l_1}(w) \\
	&= m(z) m(w) \rho \frac{(k-1)}{N} \frac{1}{N} \sum_{l_1, l_2 < k} b_{k l_1 l_2}(z) b_{k l_2 l_1}(w) + m(z) m(w) \left( \frac{k-1}{N} \right) + o(1) \\
	&= \frac{m(z) m(w) \left(\frac{k-1}{N} \right) }{ 1 - \rho \left( \frac{k-1}{N} \right) m(z) m(w) } + o(1).
\end{align*}
As
\begin{align*}
	\lim_{N \to \infty} &\frac{1}{N} \sum_{k=1}^N \frac{m(z) m(w) \left(\frac{k-1}{N} \right) }{ 1 - \rho \left( \frac{k-1}{N} \right) m(z) m(w) } = \int_{0}^1 \frac{t m(z) m(w)}{1 - t \rho m(z) m(w)} dt,
\end{align*}
we conclude that
\begin{equation} \label{eq:betaN1}
	\rho^2 \beta_{N,1}(z,w) \longrightarrow - \rho - \frac{ \log( 1- \rho m(z) m(w)) }{ m(z) m(w) } 
\end{equation}
in probability as $N \to \infty$.  Similarly, we obtain 
\begin{equation} \label{eq:betaN2}
	\beta_{N,2}(z,w) \longrightarrow -1 - \frac{ \log (1 - m(z) m(w)) }{m(z) m(w)} 
\end{equation}
in probability as $N \to \infty$.

From \eqref{eq:betaNreduce}, \eqref{eq:betaN1}, and \eqref{eq:betaN2}, we conclude that $\beta_{N}(z,w)$ converges in probability to $\beta(z,w)$ as $N \to \infty$, where $\beta(z,w)$ is defined in \eqref{def:beta}.  In particular, this implies that
$$ \sum_{k=1}^N \E_{k-1} \left[ Z_{N,k}''(z) Z_{N,k}''(w) \right] \longrightarrow \frac{\partial^2}{\partial z \partial w} m(z) m(w) \beta(z,w) $$
in probability as $N \to \infty$.  This completes the proof of Theorem \ref{thm:finitedim}.

\begin{remark}
We note that the functions $\beta(z,w)$ and $m(z)$ depend on $\hat{\rho}$ instead of $\rho$ (to use the notation from above).  However, in view of Lemma \ref{lemma:truncation} and Lemma \ref{lemma:cont} in Appendix \ref{sec:cont}, one can replace all occurrences of $\hat{\rho}$ with $\rho$.  Similarly, from part \eqref{item:highmom} of Lemma \ref{lemma:truncation}, one can replace $\E[\hat{\xi}_1^2 \hat{\xi}_2^2]$ with $\E[\xi_1^2 \xi_2^2]$.  
\end{remark}

\section{Tightness} \label{sec:tight}

We now verify that the process $\Xi_N$ is tight in the space of continuous functions on the contour $\mathcal{C}$.  It follows from the Arzela-Ascoli criteria (see e.g. \cite{PB68}) that it suffices to verify the following result.

\begin{lemma} \label{lemma:tight}
There exists a constant $C>0$ (independent of $N$) such that
$$ \E \left|\frac{\Xi_N(z) - \Xi_N(w)}{z-w} \right|^2 \leq C $$
for all $z,w \in \mathcal{C}$ with $z \neq w$.  
\end{lemma}

It is not surprising that the proof of Lemma \ref{lemma:tight} uses many of the same techniques described in Section \ref{sec:distr}.  We will also use the same notation introduced in Section \ref{sec:distr}.  Recall that $c_k$ denotes the $k$-th column of $\mat{X}_N$ with the $k$-th entry removed and $r_k$ denotes the $k$-th row of $\mat{X}_N$ with the $k$-th entry removed.  We begin with a few preliminary results which are corollaries of Lemmas \ref{lemma:alpha} and \ref{lemma:lde}.

\begin{lemma} \label{lemma:alphazw}
There exists a constant $C>0$ (independent of $N$) such that 
$$ \sum_{k=1}^N \sqrt{ \E \left| \frac{\alpha_{N,k}(z) - \alpha_{N,k}(w)}{z-w} \right|^4 \oindicator{\Omega_{N,k}} } \leq C $$
for all $z,w \in \mathcal{C}$ with $z \neq w$, where $\alpha_{N,k}(z)$ is defined in \eqref{def:alpha}.  
\end{lemma}
\begin{proof}
As
$$ \mat{G}_{N,k}(z) - \mat{G}_{N,k}(w) = (z-w) \mat{G}_{N,k}(z) \mat{G}_{N,k}(w), $$
 it suffices to show that
$$ \sum_{k=1}^N \sqrt{ \E \left| r_k \mat{G}_{N,k}(z) \mat{G}_{n,k}(w) c_k - \frac{\rho}{N} \tr \mat{G}_{N,k}(z) \mat{G}_{N,k}(w) \right|^4 \oindicator{\Omega_{N,k}} } \leq C $$
for all $z,w \in \mathcal{C}$.  The claim now follows from \eqref{eq:GNk2}.  
\end{proof}

\begin{lemma} \label{lemma:GNk2zw}
There exists a constant $C>0$ (independent of $N$) such that 
$$ \sup_{1 \leq k \leq N} \E \left| \frac{r_k \left(\mat{G}_{N,k}^2(z) - \mat{G}_{N,k}^2(w) \right) c_k}{ z-w} \right|^4\oindicator{\Omega_{N,k}} \leq C $$
for all $z,w \in \mathcal{C}$ with $z \neq w$.  
\end{lemma}
\begin{proof}
By the triangle inequality, we have
\begin{align*}
	\E &\left| r_k \left(\mat{G}_{N,k}^2(z) - \mat{G}_{N,k}^2(w) \right) c_k \right|^4\oindicator{\Omega_{N,k}} \\
	&\leq C\bigg( \E \left| r_k \left( \mat{G}_{N,k}^2(z) - \mat{G}_{N,k}(z) \mat{G}_{N,k}(w) \right) c_k \right|^4 \oindicator{\Omega_{N,k}} \\
	&\qquad + \E \left| r_k \left( \mat{G}_{N,k}(z) \mat{G}_{N,k}(w) - \mat{G}_{N,k}^2(w) \right) c_k \right|^4 \oindicator{\Omega_{N,k}} \bigg) 
\end{align*}
for some absolute constant $C>0$.  By the resolvent identity, we obtain
\begin{equation} \label{eq:zzwresolv}
	\mat{G}_{N,k}^2(z) - \mat{G}_{N,k}(z) \mat{G}_{N,k}(w) = (z-w)\mat{G}_{N,k}^2(z) \mat{G}_{N,k}(w) 
\end{equation}
and
\begin{equation} \label{eq:zwwresolv}
	\mat{G}_{N,k}(z) \mat{G}_{N,k}(w) - \mat{G}_{N,k}^2(w) = (z-w) \mat{G}_{N,k}(z) \mat{G}_{N,k}^2(w). 
\end{equation}
Thus, it suffices to show that 
\begin{equation} \label{eq:2bndshow}
	\sup_{1 \leq k \leq N} \E \left| r_k \mat{G}_{N,k}^2(z) \mat{G}_{N,k}(w) c_k \right|^4 \oindicator{\Omega_{N,k}}\leq C 
\end{equation}
and
\begin{equation} \label{eq:2bndshow2}
	\sup_{1 \leq k \leq N}\E \left| r_k \mat{G}_{N,k}(z) \mat{G}_{N,k}^2(w) c_k \right|^4 \oindicator{\Omega_{N,k}} \leq C 
\end{equation}
for all $z,w \in \mathcal{C}$.   

We will prove \eqref{eq:2bndshow}; the proof of \eqref{eq:2bndshow2} is similar and left to the reader.  Since 
$$ \left| \frac{\rho}{N} \tr \left( \mat{G}_{N,k}^2(z) \mat{G}_{N,k}(w) \right) \right| \leq \| \mat{G}_{N,k}(z)\|^2 \| \mat{G}_{N,k}(w) \| \leq \frac{1}{c^3} $$
on the event $\Omega_{N,k}$, it suffices to show 
$$ \sup_{1 \leq k \leq N} \E \left| r_k \mat{G}_{N,k}^2(z) \mat{G}_{N,k}(w) c_k - \frac{\rho}{N} \tr \left( \mat{G}_{N,k}^2(z) \mat{G}_{N,k}(w) \right) \right|^4 \oindicator{\Omega_{N,k}}\leq C. $$
The claim now follows from Lemma \ref{lemma:lde}.  
\end{proof}

\begin{lemma} \label{lemma:GNk2zwtrace}
There exists a constant $C>0$ (independent of $N$) such that 
$$ \sum_{k=1}^N \E \left| \frac{ r_k \left(\mat{G}_{N,k}^2(z) - \mat{G}_{N,k}^2(w) \right) c_k - \frac{\rho}{N} \tr \left(\mat{G}_{N,k}^2(z) - \mat{G}_{N,k}^2(w) \right) }{ z-w } \right|^2 \oindicator{\Omega_{N,k}} \leq C $$
for all $z,w \in \mathcal{C}$ with $z \neq w$.  
\end{lemma}
\begin{proof}
By the triangle inequality, \eqref{eq:zzwresolv}, and \eqref{eq:zwwresolv}, it suffices to show that
\begin{equation} \label{eq:sumzzw}
	\sum_{k=1}^N \E \left| r_k \mat{G}_{N,k}^2(z) \mat{G}_{N,k}(w) c_k - \frac{\rho}{N} \tr \mat{G}_{N,k}^2(z) \mat{G}_{N,k}(w) \right|^2 \oindicator{\Omega_{N,k}} \leq C 
\end{equation}
and
\begin{equation} \label{eq:sumzww}
	\sum_{k=1}^N \E \left| r_k \mat{G}_{N,k}(z) \mat{G}_{N,k}^2(w) c_k - \frac{\rho}{N} \tr \mat{G}_{N,k}(z) \mat{G}_{N,k}^2(w) \right|^2 \oindicator{\Omega_{N,k}} \leq C 
\end{equation}
for all $z,w \in \mathcal{C}$.  Both \eqref{eq:sumzzw} and \eqref{eq:sumzww} follow from Lemma \ref{lemma:lde}.  
\end{proof}

We are now ready to prove Lemma \ref{lemma:tight}.  

\begin{proof}[Proof of Lemma \ref{lemma:tight}]
Since
\begin{align*}
	\E &\left| (\tr \mat{G}_N(z) \oindicator{\Omega_{N}} - \tr \mat{G}_{N}(z) \oindicator{\Omega_N \cap Q_N}) - (\tr \mat{G}_N(w) \oindicator{\Omega_{N}} - \tr \mat{G}_{N}(w) \oindicator{\Omega_N \cap Q_N}) \right|^2 \\
	&\qquad= \E \left| \tr \mat{G}_{N}(z) - \tr \mat{G}_N(w) \right|^2 \oindicator{\Omega_N} \oindicator{Q_N^C} \\
	&\qquad\leq |z-w|^2 \E \left| \tr \mat{G}_N(z) \mat{G}_N(w) \right|^2 \oindicator{\Omega_N} \oindicator{Q_N^C} \\
	&\qquad\leq \frac{|z-w|^2}{c^4} N^2 \Prob(Q_N^C) \\
	&\qquad\leq C |z-w|^2
\end{align*}
by the resolvent identity and Corollary \ref{cor:QN}, it suffices to show that
$$ \E \left| \sum_{k=1}^N (\E_k - \E_{k-1}) \frac{ \left(\tr \mat{G}_N(z)\oindicator{\Omega_N \cap Q_N} - \tr \mat{G}_N(w)\oindicator{\Omega_N \cap Q_N} \right) }{z-w} \right|^2 \leq C $$
for all $z,w \in \mathcal{C}$ with $z \neq w$.  Here we applied the same martingale decomposition as we used previously in Section \ref{sec:distr}.  

As before, we write
\begin{align*}
	\sum_{k=1}^N &(\E_k - \E_{k-1}) \left(\tr \mat{G}_N(z)\oindicator{\Omega_N \cap Q_N} - \tr \mat{G}_N(w)\oindicator{\Omega_N \cap Q_N} \right) \\
		&= \sum_{k=1}^N (\E_k - \E_{k-1}) \bigg[ \left( \tr \mat{G}_N(z)\oindicator{\Omega_N \cap Q_N} - \tr \mat{G}_{N,k}(z)\oindicator{\Omega_{N,k} \cap Q_{N,k}} \right) \\
		&\qquad\qquad\qquad\qquad - \left( \tr \mat{G}_N(w)\oindicator{\Omega_N \cap Q_N}  - \tr \mat{G}_{N,k}(w)\oindicator{\Omega_{N,k} \cap Q_{N,k}}\right) \bigg].
\end{align*}
We now observe that
\begin{align*}
	 \E \bigg| \sum_{k=1}^n& (\E_k - \E_{k-1}) \bigg[ \left( \tr \mat{G}_{N,k}(z) \oindicator{\Omega_{N,k} \cap Q_{N,k}} - \tr \mat{G}_{N,k}(w) \oindicator{\Omega_{N,k} \cap Q_{N,k}} \right) \\
	 &\qquad\qquad\qquad - \left( \tr \mat{G}_{N,k}(z) \oindicator{\Omega_{N} \cap Q_{N}} - \tr \mat{G}_{N,k}(w) \oindicator{\Omega_{N} \cap Q_{N}} \right) \bigg] \bigg|^2 \\
	 &\leq C \sum_{k=1}^N \E \left| \tr \mat{G}_{N,k}(z) - \tr \mat{G}_{N,k}(w) \right|^2 \oindicator{\Omega_{N,k}\cap Q_{N,k}} \oindicator{\Omega_N^C \cup Q_N^C} \\
	 &\leq C |z-w|^2 \sum_{k=1}^N \left| \tr \mat{G}_{N,k}(z) \mat{G}_{N,k}(w) \right|^2\oindicator{\Omega_{N,k}\cap Q_{N,k}} \oindicator{\Omega_N^C \cup Q_N^C} \\
	 &\leq \frac{C |z-w|^2}{c^4} N^3 \Prob(\Omega_N^C \cup Q_N^C) \\
	 &\leq C |z-w|^2
\end{align*}
by \eqref{eq:hatONC} and Corollary \ref{cor:QN}.  Thus, it suffices to show that
\begin{align*}
	\E \left| \sum_{k=1}^N (\E_k - \E_{k-1}) \frac{ \left( \tr \mat{G}_{N}(z) - \tr \mat{G}_{N,k}(z) \right) - \left( \tr \mat{G}_{N}(w) - \tr \mat{G}_{N,k}(w) \right) }{z-w} \oindicator{\Omega_N \cap Q_N} \right|^2 \leq C
\end{align*}
for all $z,w \in \mathcal{C}$ with $z \neq w$.  

By Proposition \ref{prop:trace} and \eqref{def:alpha}, we have
\begin{align*}
	 &\sum_{k=1}^N (\E_k - \E_{k-1}) \left[ \left( \tr \mat{G}_{N}(z) - \tr \mat{G}_{N,k}(z) \right)\oindicator{\Omega_N \cap Q_N} - \left( \tr \mat{G}_{N}(w) - \tr \mat{G}_{N,k}(w) \right)\oindicator{\Omega_N \cap Q_N} \right]  \\
	 &= \sum_{k=1}^N(\E_k - \E_{k-1}) \left[ \left( T_{N,k1}(z) - T_{N,k,2}(z) \right)\oindicator{\Omega_N \cap Q_N} - \left( T_{N,k,1}(w) - T_{N,k,2}(w) \right)\oindicator{\Omega_N \cap Q_N} \right],
\end{align*}
where
$$ T_{N,k,1}(z) := \frac{ (1 + r_k \mat{G}_{N,k}^2(z) c_k ) \left(z + \frac{\rho}{N} \tr \mat{G}_{N,k}(z) \right)^{-1} \alpha_{N,k}(z) }{ \frac{1}{\sqrt{N}} y_{kk} - z- r_k \mat{G}_{N,k}(z) c_k } $$
and
$$ T_{N,k,2}(z) := (1 + r_k \mat{G}_{N,k}^2(z) c_k ) \left(z + \frac{\rho}{N} \tr \mat{G}_{N,k}(z) \right)^{-1}. $$
We now note that
\begin{align*}
	&\sum_{k=1}^N \E \left|  \left( T_{N,k,2}(z) - T_{N,k,2}(w) \right) \oindicator{\Omega_{N} \cap Q_N} - \left( T_{N,k,2}(z) - T_{N,k,2}(w) \right) \oindicator{\Omega_{N,k} \cap Q_{N,k}} \right|^2 \\
	&\leq \sum_{k=1}^N \E \left| T_{N,k,2}(z) - T_{N,k,2}(w) \right|^2 \oindicator{\Omega_{N,k} \cap Q_{N,k}} \oindicator{\Omega_N^C \cup Q_N^C} \\
	&\leq C|z-w|^2 \sum_{k=1}^N \sqrt{ \E \left| 1 + r_k \mat{G}_{N,k}^2(z) c_k \right|^4 \oindicator{\Omega_N,k} } \sqrt{ \Prob (\Omega_N^C \cup Q_N^C) } \\
	&\qquad + C \sum_{k=1}^N \sqrt{ \E \left| r_k \mat{G}_{N,k}^2(z) c_k - r_k \mat{G}_{N,k}^2(w) c_k \right|^4 \oindicator{\Omega_{N,k}} } \sqrt{ \Prob (\Omega_N^C \cup Q_N^C) } \\
	&\leq C |z-w|^2
\end{align*}
by the Cauchy-Schwarz inequality and Lemmas \ref{lemma:lde} and \ref{lemma:GNk2zw}. 

Therefore, by the triangle inequality, it suffices to show that
\begin{equation} \label{eq:tnk1}
	\E \left| \sum_{k=1}^N (\E_k - \E_{k-1}) \frac{ T_{N,k,1}(z) - T_{N,k,1}(w) }{z-w} \oindicator{\Omega_{N} \cap Q_N}  \right|^2 \leq C 
\end{equation}
and
\begin{equation} \label{eq:tnk2}
	\E \left| \sum_{k=1}^N (\E_k - \E_{k-1}) \frac{ T_{N,k,2}(z) - T_{N,k,2}(w) }{z-w} \oindicator{\Omega_{N,k} \cap Q_{N,k}}  \right|^2 \leq C
\end{equation}
for all $z,w \in \mathcal{C}$ with $z \neq w$.  

The bound in \eqref{eq:tnk1} follows from Lemma \ref{lemma:tnk1} below.  It remains to prove \eqref{eq:tnk2}.  We first observe that
\begin{align*}
	(\E_k - \E_{k-1}) &T_{N,k,2}(z) \oindicator{\Omega_{N,k} \cap Q_{N,k}} \\
	&= \E_k \left(r_k \mat{G}_{N,k}^2(z) c_k - \frac{\rho}{N} \tr \mat{G}_{N,k}^2(z) \right) \left(z + \frac{\rho}{N} \tr \mat{G}_{N,k}(z) \right)^{-1} \oindicator{\Omega_{N,k} \cap Q_{N,k}}. 
\end{align*}
In addition, by the resolvent identity, we have
\begin{align*}
	&\bigg| \left(r_k \mat{G}_{N,k}^2(z) c_k - \frac{\rho}{N} \tr \mat{G}_{N,k}^2(z) \right) \left(z + \frac{\rho}{N} \tr \mat{G}_{N,k}(z) \right)^{-1} \\
	&\qquad - \left(r_k \mat{G}_{N,k}^2(w) c_k - \frac{\rho}{N} \tr \mat{G}_{N,k}^2(w) \right) \left(z + \frac{\rho}{N} \tr \mat{G}_{N,k}(w) \right)^{-1} \bigg| \\
	&\leq C|z-w| \left| r_k \mat{G}_{N,k}^2(z) c_k - \frac{\rho}{N} \tr \mat{G}_{N,k}^2(z) \right| \\
	&\qquad + C\left| \left( r_k \mat{G}_{N,k}^2(z) c_k - \frac{\rho}{N} \tr \mat{G}_{N,k}^2(z) \right) - \left( r_k \mat{G}_{N,k}^2(w) c_k - \frac{\rho}{N} \tr \mat{G}_{N,k}^2(w) \right) \right|
\end{align*}
on the event $Q_{N,k}$.  So we conclude that
\begin{align*}
	\E &\left| \sum_{k=1}^N (\E_k - \E_{k-1}) \left( T_{N,k,2}(z) - T_{N,k,2}(w)\right) \oindicator{\Omega_{N,k} \cap Q_{N,k}}  \right|^2 \\
	&\leq C |z-w|^2 \sum_{k=1}^N \E \left| r_k \mat{G}_{N,k}^2(z) c_k - \frac{\rho}{N} \tr \mat{G}_{N,k}^2(z) \right|^2 \oindicator{\Omega_{N,k}} \\
	&\qquad + C \sum_{k=1}^N \E \left|  r_k \left( \mat{G}_{N,k}^2(z) - \mat{G}_{N,k}^2(w) \right) c_k - \frac{\rho}{N} \tr \left( \mat{G}_{N,k}^2(z) - \mat{G}_{N,k}^2(w) \right) \right|^2 \oindicator{\Omega_{N,k}} \\
	&\leq C|z-w|^2
\end{align*}
by Lemmas \ref{lemma:lde} and \ref{lemma:GNk2zwtrace}.  This completes the proof of \eqref{eq:tnk2}.  
\end{proof}

It remains to prove estimate \eqref{eq:tnk1}.  

\begin{lemma} \label{lemma:tnk1}
There exists a constant $C>0$ (independent of $N$) such that \eqref{eq:tnk1} holds for all $z,w \in \mathcal{C}$ with $z \neq w$.
\end{lemma}
\begin{proof}
By several applications of the triangle inequality and the trivial bounds 
$$ \sup_{z \in \mathcal{C}} \left| \left(\mat{G}_N(z)\right)_{kk} \right| = \sup_{z \in \mathcal{C} } \left| \frac{1}{\frac{1}{\sqrt{N}} y_{kk} - z- r_k \mat{G}_{N,k}(z) c_k } \right| \leq \frac{1}{c} $$
and
$$ \sup_{z \in \mathcal{C}} \left| \left( z + \frac{\rho}{N} \tr \mat{G}_{N,k}(z) \right)^{-1} \right| \leq \frac{2}{c}, $$
which both hold on the event $\Omega_N \cap Q_N$, the verification of \eqref{eq:tnk1} follows from the following four inequalities.  
\begin{enumerate}[(i)]
\item We have
\begin{align*}
	\sum_{k=1}^N &\E \left| r_k \mat{G}_{N,k}^2(z)c_k - r_k \mat{G}_{N,k}^2(w) c_k \right|^2 \left| \alpha_{N,k}(z) \right|^2 \oindicator{\Omega_{N} \cap Q_N} \\
	&\leq \sum_{k=1}^N \sqrt{ \E \left| r_k \mat{G}_{N,k}^2(z)c_k - r_k \mat{G}_{N,k}^2(w) c_k \right|^4  \oindicator{\Omega_{N,k}} } \sqrt{ \E \left| \alpha_{N,k}(z) \right|^4 \oindicator{\Omega_{N,k}} } \\
	&\leq C |z-w|^2
\end{align*}
by the Cauchy Schwarz inequality and Lemmas \ref{lemma:alpha} and \ref{lemma:GNk2zw}.  
\item Similarly, by Lemmas \ref{lemma:alpha} and \ref{lemma:lde}, we obtain
\begin{align*}
	\sum_{k=1}^N \E &\left| (1 + r_k \mat{G}_{N,k}^2(w) c_k ) \alpha_{N,k}(z) \left( \left( z + \frac{\rho}{N} \tr \mat{G}_{N,k}(z) \right)^{-1} - \left( w + \frac{\rho}{N} \tr \mat{G}_{N,k}(w)  \right)^{-1} \right) \right|^2 \oindicator{\Omega_{N} \cap Q_N} \\
	&\leq C |z-w|^2 \sum_{k=1}^N \sqrt{ \E \left| 1 + r_k \mat{G}_{N,k}^2(w) c_k \right|^4\oindicator{\Omega_{N,k}} } \sqrt{ \E \left| \alpha_{N,k}(z) \right|^4 \oindicator{\Omega_{N,k}} } \\
	&\leq C |z-w|^2.  
\end{align*}
\item In view of Lemmas \ref{lemma:lde} and \ref{lemma:alphazw}, we have
\begin{align*}
	\sum_{k=1}^N \E &\left| (1 + r_k \mat{G}_{N,k}^2(w) c_k) (\alpha_{N,k}(z) - \alpha_{N,k}(w)) \right|^2 \oindicator{\Omega_{N} \cap Q_N} \\
	&\leq \sum_{k=1}^N \sqrt{ \E \left| 1 + r_k \mat{G}_{N,k}^2(w) c_k \right|^4 \oindicator{\Omega_{N,k}} } \sqrt{ \E \left| \alpha_{N,k}(z) - \alpha_{N,k}(w) \right|^4\oindicator{\Omega_{N,k}} } \\
	&\leq C |z-w|^2.
\end{align*}
\item Finally, recalling that
$$ \left( \mat{G}_{N}(z) \right)_{kk} = \frac{1}{\frac{1}{\sqrt{N}} y_{kk} - z - r_k \mat{G}_{N,k}(z) c_k } $$
on the event $\Omega_N$, we have
\begin{align*} 
	\sum_{k=1}^N &\E \left| (1+r_k \mat{G}_{N,k}^2(w) c_k) \alpha_{N,k}(w) \left( \left( \mat{G}_{N}(z) \right)_{kk} - \left( \mat{G}_{N}(w) \right)_{kk} \right) \right|^2 \oindicator{\Omega_N \cap Q_N} \\
	&\leq C|z-w|^2 \sum_{k=1}^N \sqrt{ \E \left| 1 + r_k \mat{G}_{N,k}^2(w) c_k \right|^4\oindicator{\Omega_{N,k}} } \sqrt{ \E \left| \alpha_{N,k}(w) \right|^4 \oindicator{\Omega_{N,k}} } \\
	&\leq C|z-w|^2
\end{align*}
by the resolvent identity and Lemmas \ref{lemma:alpha} and \ref{lemma:lde}.
\end{enumerate}
The proof of the lemma is complete.  
\end{proof}

\section{Concluding the proof of Theorem \ref{thm:main}} \label{sec:conclusion}

It follows from Theorem \ref{thm:finitedim} and Lemma \ref{lemma:tight}, that $\Xi_N$ converges weakly as $N \to \infty$ to a mean-zero Gaussian process in the space of continuous functions on the contour $\mathcal{C}$ (see \cite{PB68} for details).  This implies that
$$ \tr f(X_N) \oindicator{E_N} - \E \tr f(X_N) \oindicator{E_N} $$
converges to a Gaussian random variable with mean zero.  Since $f$ was defined as a linear combination of $f_1, \ldots, f_k$, we conclude (by the Cram\'{e}r-Wold device) that the random vector
$$ \left( \tr f_j \left( \frac{1}{\sqrt{N}} \mat{Y}_N \right) \oindicator{E_N} - \E  \tr f_j \left( \frac{1}{\sqrt{N}} \mat{Y}_N \right) \oindicator{E_N} \right)_{j=1}^k $$
converges to a mean-zero multivariate Gaussian $(\mathcal{G}(f_1), \ldots, \mathcal{G}(f_k))$ as $N \to \infty$.  We now compute the covariance structure.  

Indeed, recall the definitions of $\upsilon(z,w)$ and $\beta(z,w)$ given in \eqref{def:upsilon} and \eqref{def:beta}.  In particular, we observe that 
\begin{equation} \label{eq:upsilonswitch}
	\upsilon(z,w) = \upsilon(w,z) 
\end{equation}
for all $z, w \in \mathcal{C}$.  Since $\Xi_N$ converges weakly to a mean-zero Gaussian process in the space of continuous functions on the contour $\mathcal{C}$, it follows from Theorem \ref{thm:finitedim} and \eqref{eq:upsilonswitch} that 
$$ \E[\mathcal{G}(f_i) \mathcal{G}(f_j) ] =  -\frac{1}{4 \pi^2} \oint_{\mathcal{C}} \oint_{\mathcal{C}} f_i(z) f_j(w) \upsilon(z,w) dz dw. $$
Thus, by integration by parts, we obtain
$$ \E[\mathcal{G}(f_i) \mathcal{G}(f_j) ] = -\frac{1}{4 \pi^2} \oint_{\mathcal{C}} \oint_{\mathcal{C}} f_i'(z) f_j'(w) m(z) m(w) \beta(z,w) dz dw. $$
The proof of Theorem \ref{thm:main} is now complete.

\appendix

\section{Proof of Theorem \ref{thm:lsv}} \label{sec:lsvproof}

The proof of Theorem \ref{thm:lsv} is essentially the same as the proof given in \cite{OR}. The primary differences are that
\begin{enumerate}[(i)]
\item we require control of the least singular value of the matrix $\mat Y_N$ whose entries have been truncated to be $O(N^{1/2 - \eps})$ instead of $O(1)$, and
\item the given bounds are shown to hold with overwhelming probability instead of almost surely.  
\end{enumerate}
Focusing on these differences, we now sketch the proof of Theorem \ref{thm:lsv}.

We begin the proof of Theorem \ref{thm:lsv} with a few reductions.  We first observe that it suffices to prove Theorem \ref{thm:lsv} with the ellipsoid $\mathcal{E}_{\hat{\rho}}$ instead of $\mathcal{E}_{\rho}$ (to use the notation from above).  Indeed, as $\lim_{N \to \infty} \hat{\rho} = \rho$, we observe that
$$ \{z \in \mathbb{C} : \dist(z, \mathcal{E}_{\rho}) \geq \delta \} \subset \{z \in \mathbb{C} : \dist(z, \mathcal{E}_{\hat{\rho}}) \geq \delta/100 \} $$
for $N$ sufficiently large.  From this point forward, we will simply write $\rho$ to denote the correlation of the truncated entries (i.e. $\rho = \E[ \hat{\xi}_1 \hat{\xi}_2]$ from the notation above).  

We continue to write $\mat{X}_N := \frac{1}{\sqrt{N}} \mat{Y}_N$, where the entries of $\mat{Y}_N$ have been truncated as in Section \ref{sec:truncation}.  We observe that it suffices to prove the theorem under the assumption that the diagonal entries of $\mat{X}_N$ are zero.  Indeed, by Lemma \ref{lemma:truncation}, the entries of $\mat{X}_N$ are bounded in absolute value by $4\eps_N$.  Let $\breve{\mat X}_N$ be constructed from the matrix $\mat{X}_N$ by setting the diagonal entries to zero.  Then, by Weyl's inequality, we have
$$ \sup_{z \in \mathbb{C}} | \sigma_N(\breve{\mat X}_N - z\mat{I}) - \sigma_N(\mat{X}_N - z \mat{I}) | \leq \| \breve{\mat X}_N - \mat{X}_N \| \leq 4\eps_N $$
almost surely.  Henceforth, we will assume that the diagonal entries of $\mat{X}_N$ are zero.  

Following the arguments in \cite{OR}, the proof now amounts to showing that for $z$ of distance greater than $\delta$ from $ \mathcal{E}_{\rho}$, there are no eigenvalues of $(\mat X_N - z \mat{I})^*(\mat X_N-z \mat{I})$ less than some constant $c$ (where $c$ is allowed to depend on $z$).  To this end, fix $z \in \mathbb{C}$ outside $ \mathcal{E}_{\rho}$ and define the probability measure
$$ \nu_{\mat X_N -z \mat I} := \frac{1}{2N} \sum_{i=1}^N \left( \delta_{\sigma_i(\mat{X}_N - z \mat{I})} + \delta_{-\sigma_i(\mat{X}_N - z \mat{I})} \right), $$
where $\sigma_1(\mat X_N -z \mat I), \ldots, \sigma_N(\mat X_N -z \mat I)$ are the singular values of $\mat X_N -z \mat I$.  

It was shown in \cite{Nell, NO} that almost surely $\nu_{ \mat X_N -z \mat I}$ converges weakly to a probability measure $\nu_z$ (which depends on both $z$ and $\rho$).  Moreover, for $z$ outside the ellipsoid $\mathcal{E}_{\rho}$, $\nu_z$ satisfies the following property.  

\begin{lemma}[Theorem 6.1 from \cite{OR}] \label{lemma:existc}
Fix $-1 \leq \rho \leq 1$, and let $\delta > 0$.  Then there exists $c>0$ such that $\nu_z([-2c,2c]) = 0$ for all $z \in \mathbb{C}$ with $\dist(z,\mathcal{E}_{\rho}) \geq \delta$.  
\end{lemma}

\begin{remark}
The proof of Lemma \ref{lemma:existc} when $\rho = \pm 1$ follows from \cite[Theorem 5.2]{BSbook}; see \cite[Remark 6.2]{OR} for further details.  
\end{remark}

Let $c>0$ be given by Lemma \ref{lemma:existc} above.  In order to complete the proof, we will show that the trace of the resolvent of a linearization of $(\mat X_N - z)^*(\mat X_N-z)$ is sufficiently close to the limiting Stieltjes transform of $\nu_z$ with overwhelming probability.    From this we will conclude that $\nu_{\mat X_N -z \mat I}([-c,c]) = 0$ with overwhelming probability.

We now introduce some notation.  We define the \emph{Hermitization} of an $N \times N$ matrix $\mat X = (X_{ij})_{i,j=1}^N$ to be an $N \times N$ matrix with entries that are  $2 \times 2$ block matrices where the $ij^{\mathrm{th}}$ entry is the $2 \times 2$ block
\[ \begin{pmatrix} 0 & {X}_{ij} \\ \overline{{X}}_{ji}& 0   \end{pmatrix}. \]
We note that the Hermitization of $\mat X$ can be conjugated by a $2N \times 2N$ permutation matrix to obtain 
\[\begin{pmatrix} 0 & \mat X  \\ \mat X^*  &  0 \end{pmatrix}.\]

%%Definiations 

Let $\mat H_N$ to be the Hermization of $\mat X_N$. 
We will generally treat $\mat H_N$ as an $N \times N$ matrix with entries that are $2\times 2$ blocks, but occasionally it will instead be useful to consider $\mat H_N$ as a $2N \times 2N$ matrix.

Additionally, we define the $2 \times 2$ matrix
\begin{align}
\label{qdef}
\mat  q := \begin{pmatrix} \eta & z \\ \overline{z}& \eta   \end{pmatrix} %:= \eta^I+z^\sigma\]
\end{align}
with $ \eta=E+i t \in \C^{+} := \{w \in \mathbb{C} : \Im(w) > 0\}$.  We also define the resolvent of the Hermitianization
\[\mat  R_N(\mat q) =\mat  R_N(\eta,z) := (\mat H_N - \mat I_N \otimes \mat  q)^{-1}. \]

%Notation
For any matrix $\mat H$ with entries that are $2 \times 2$ blocks, we mean $\tr_N(\mat H) := \frac{1}{N} \sum_{i} \mat H_{ii}$ where $\mat H_{ii}$ is the $i^{\mathrm{th}}$ diagonal $2 \times 2$ block of $\mat H$. When working with $N \times N$ matrices with entries that are $2 \times 2$ blocks, we use superscripts to refer to entries of the $2 \times 2$ blocks. Additionally, when forming an $N \times N$ matrix whose $ij^{\mathrm{th}}$ entry is the $ab^{\mathrm{th}}$ entry ($a,b \in \{1,2\}$) of the $ij^{\mathrm{th}}$ $2 \times 2$ block we also use superscripts. For example, $\mat R^{21}$ is the $N \times N$ matrix formed from taking each $\mat R_{ij}$ block and replacing it by its (2,1)-entry.  

%Let $\mat \Gamma_N(\mat q): = \tr_N(\mat R_N (\mat q) )$. By the symmetry of $\mat H_N$, $\sum_{l} \mat R^{22}_{ll} = \sum_{l} \mat R^{11}_{ll}$, i.e. $\mat \Gamma^{11}_N = \mat \Gamma^{22}_N$. Let $a_N(\mat q) := \mat \Gamma^{11}_N(\mat q) $, $b_N(\mat q) := \mat \Gamma^{12}_N(\mat q)$, and $c_N(\mat q) := \mat \Gamma^{21}_N(\mat q) $.

We write $\mat H_i$ to be the $i^{\mathrm{th}}$ column (of $2 \times 2$ blocks) of $\mat H_N$ and $\mat H_i^{(i)}$ to be the $i^{\mathrm{th}}$ column of $\mat H_N$ with the $i^{\mathrm{th}}$ block removed. We let $\mat R_N^{(i)}$ be the resolvent of $\mat H_N$ where the $i^{\mathrm{th}}$ row and $i^{\mathrm{th}}$ column of $\mat H_N$ (viewed as an $N \times N$ matrix of $2 \times 2$ blocks) have been removed, and set $\mat \Gamma_N^{(i)}(\mat{q}) := \frac{1}{N} \sum_{j \not = i} \mat R_{jj}^{(i)}$.

Let $\Gamma(\mat q)$ be the $2 \times 2$ matrix Stieltjes transform with positive imaginary part which satisfies the fixed point equation
\begin{align*}
\mat  \Gamma(\mat q) = - (\mat  q + \mat \Sigma(\mat  \Gamma(\mat q)) )^{-1},
 \end{align*} 
where $\mat \Sigma$ is the operator on $2 \times 2$ matrices defined by
\[\mat \Sigma \begin{pmatrix} a & b \\ c & d \end{pmatrix} :=\begin{pmatrix} d & \rho  c \\ \rho b & a \end{pmatrix}. \]
It follows from \cite{Nell,NO} that $\mat \Gamma_N$ converges to $\mat \Gamma$ almost surely.

To complete the proof we proceed as in \cite{OR} and prove an a priori bound on  $\mat \Gamma_N(\mat{q}) - \mat \Gamma(\mat{q})$. This will provide an estimate on the number of small singular values of $\mat X_N - z \mat I$.  This estimate will be used to prove a better bound on $\mat \Gamma_N(\mat{q}) - \mat \Gamma(\mat{q})$ from which the desired result will follow.

We now develop an a priori bound on $\mat \Gamma_N(\mat{q}) - \mat \Gamma(\mat{q})$ for $\eta = E+ i t_N \in \mathbb{C}^+$, with $t_N$ going to zero polynomially in $N$ and $E \in [0,c]$.

By the Schur complement, the diagonal entries of the resolvent are
\begin{align*}
%R_{ii} &= -(q + H_{i \cdot}^{(i)} R_N^{(i)} H_{\cdot i}^{(i)})^{-1} \\
\mat R_{ii} &= -(\mat q + \mat H_{i }^{(i)} \mat R_N^{(i)} \mat H_{ i}^{(i)})^{-1} \\
&=  - (\mat q +\mat \Sigma(\mat \Gamma_N) - \mat \Sigma( \mat \Gamma_N) +\mat \Sigma(\mat \Gamma_N^{(i)})   -\mat \Sigma(\mat \Gamma_N^{(i)}) + \mat  H^{(i)*}_i \mat R_N^{(i)} \mat H^{(i)}_i )^{-1}.
\end{align*}
%Recall that the diagonal elements of $X_N$ and hence $H_N$ have been set to zero.
Let
\begin{equation*}
\mat{ \widehat \gamma}_N^{(i)} := \mat H_i^{(i)*} \mat R_N^{(i)} \mat H_i^{(i)} - \mat \Sigma( \mat \Gamma_N^{(i)}).
\end{equation*}
 Summing over $i$ gives a formula for the trace:
\begin{align*}
\mat \Gamma_N( \mat q) %& =\sum_{i} ( - q - H^{(i)}_{i } R_N^{(i)} H^{(i)}_{ i})^{-1}   \nonumber \\
&=\sum_{i} - (\mat q + \mat \Sigma( \mat \Gamma_N(\mat q)) -\mat \Sigma(\mat \Gamma_N(\mat q)) + \mat \Sigma(\mat \Gamma_N^{(i)}(\mat q))    + \mat{\widehat \gamma}_N^{(i)} )^{-1}.
\end{align*}

\begin{lemma} 
\label{epsbound}
There exist some $\alpha,\beta>0$ such that if $\mat q$ is as in \eqref{qdef} with $t_N \geq N^{-\beta}$, then with overwhelming probability
\[ \sup_{1\leq i\leq N, E \in [0,c] } \|\mat \Sigma(\mat \Gamma_N(\mat q)) - \mat \Sigma(\mat \Gamma_N^{(i)}(\mat q))   -\mat {\widehat \gamma}_N^{(i)}\| = O(N^{-\alpha}).\]
\end{lemma}
We will require that $\alpha+\beta<2\eps$ and $\beta < \alpha$.

\begin{proof}[Proof of Lemma \ref{epsbound}]
As in \cite{OR}, we restrict to an $N^{-1}$-net, $S_N$, of $[0,c]$. In \cite{OR} the estimate
\begin{equation*}
%\label{rank2}
\|\mat \Sigma( \mat \Gamma_N - \mat \Gamma_N^{(i)}) \| = O((N t_N)^{-1})
\end{equation*}
is deterministic and can be repeated.

To bound $\|\mat{\widehat \gamma}_N^{(i)}\|$, we apply Lemma \ref{lemma:lde} to each entry of this block. Noting that the $2p^{\mathrm{th}}$ moment of any entry of $\mat Y_N$ is $O(N^{(2p-4)(1/2-\eps)})$, we have
\begin{align}
\label{gammoment}
\E &[|\mat{\widehat  \gamma}_N^{(i)ab}|^p ] \notag \\
&\leq \frac{ K_p }{ N^{p} } \E\left[ (\tr(\mat R^{(i)a'b'}(\mat R^{(i)a'b'})^*))^{p/2}+N^{(2p-4)(1/2-\eps)} \tr(\mat R^{(i)a'b'}(\mat R^{(i)a'b'})^*)^{p/2} \right] \notag \\
&\leq K_p t_N^{-p} (N^{-p/2} + N^{-2 p \varepsilon})
\end{align}
with $a$ and $b$ either 1 or 2, and $a'=a+1 \pmod 2$, $b'=b+1 \pmod 2$. The final estimate uses that $N$ times the operator norm of a self-adjoint matrix bounds its trace; a trivial estimate then shows that the operator norm is bounded by $t_N^{-2}$.

%\[ \tr(R^{(i)ab} (R^{(i)ab})^*) \leq N \eta^{-2} \]
%\[ \tr(R^{(i)12} (R^{(i)12})^*) \leq N \eta^{-2} \]
By Markov's inequality and the union bound, we obtain
\begin{align*}
\P\left(\max_{1\leq i\leq N, E_j \in S_N} \|\mat{\widehat \gamma}_N^{(i)} \| \geq N^{-\alpha} \right) &\leq \sum_{1\leq i\leq N, E_j \in S_N} N^{p \alpha} \E( \|\mat{ \widehat \gamma}_N^{(i)} \|^p ) \notag\\
&\leq  K_p   N^{2+ p (\alpha+\beta-2\eps)} 
\end{align*}
Since $p \geq 2$ is arbitrary, the proof of the lemma is complete.  
\end{proof}

The proof of \cite[Lemma 6.5]{OR} can now be repeated nearly verbatim on the event 
\[\sup_{1\leq i\leq N, E \in [0,c] } \|\mat \Sigma(\mat \Gamma_N(\mat q)) - \mat \Sigma(\mat \Gamma_N^{(i)}(\mat q))   -\mat {\widehat \gamma}_N^{(i)}\| = O(N^{-\alpha})\]
 to show that 
\begin{equation}
\label{STap}
\sup_{E\in[0,c]} \|\mat \Gamma_N(\mat q) - \mat \Gamma(\mat q)\| = o( N^{-\beta}) 
\end{equation}
with overwhelming probability.

The arguments in the remainder of \cite[Section 6.2]{OR} are deterministic and can be repeated to turn estimate \eqref{STap} into the following estimates on the empirical spectral measure which hold with overwhelming probability:
\[ \max_{k\leq N} \E_k[ \nu_{\mat X_N - z \mat I}([0,c])] = o(N^{-\beta}),~~~ \max_{k\leq N} \E_k[ \nu_{\mat X_{N,i} - z \mat I}([0,c])] = o(N^{-\beta}). \]
These bounds give, by the spectral theorem, an $O(1)$ bound on the trace of the resolvent (see \cite[Lemma 6.7]{OR}). In light of these estimates, we define the event $\Lambda_i$ that $\nu_{X_{N,i}-zI}([0,c])= o( N^{-\beta})$.

Now we use the a priori bound to estimate $\mat \Gamma_N - \E[\mat  \Gamma_N]$ and $\E[\mat \Gamma_N] -\mat \Gamma$.  Indeed, we estimate $\mat \Gamma_N - \E[\mat  \Gamma_N]$ by rewriting it as a sum of martingale differences:
\begin{align*}
\mat \Gamma_N - \E[\mat  \Gamma_N] = \sum_{i=1}^N (\E_{i} - \E_{i-1} ) (\mat \Gamma_N - \mat \Gamma_N^{(i)}) = \frac{1}{N} \sum_{i=1}^N   (\E_{i} - \E_{i-1} ) (\mat  R_{ii} \mat  \zeta_N^{(i)}),
\end{align*}
where 
\[ \mat \zeta_N^{(i)} := (\mat  I_2 + \mat H_i^{(i)*} \mat R_N^{(i)} \mat R_N^{(i)} \mat H_i^{(i)}). \]

To complete the proof it suffices to show that for arbitrary $\eta > 0$, and any $l > 0$
\[ \P\left( \max_{E\in\mathcal{S}_N} N t_N \| \mat \Gamma_N - \E[ \mat \Gamma_N] \| \geq \eta \right) = O_{l}(N^{-l}).\]

%Let $\Lambda_N$ be the event that $\max_{1\leq i\leq N} \| \epsilon_N^i \| \leq N^{-\alpha}$.

Recalling that 
\[ \P( \cup_{i=1}^{N}\{\oindicator{\Lambda_i} =0\}) = O_l(N^{-l}) ,\]
leads to the estimate
\begin{align*}
\P&( \max_{E\in\mathcal{S}_N}  N t_N \| \mat \Gamma_N - \E[ \mat \Gamma_N] \| > \eta) \\
&= \P( ( \max_{E\in\mathcal{S}_N}  N t_N \| \mat \Gamma_N - \E[ \mat \Gamma_N] \| > \eta ) \cap_{i=1}^{N} \{\oindicator{\Lambda_i} =1\}) \\
&\qquad +  \P( ( \max_{E\in\mathcal{S}_N}  N t_N \| \mat \Gamma_N - \E[ \mat \Gamma_N] \| > \eta ) \cap_{i=1}^{N} \{\oindicator{\Lambda_i} =1\}^C)  \\
&\leq \P(  \max_{E\in\mathcal{S}_N} t_N \| \sum_{i=1}^N (\E_{i} - \E_{i-1}  ) \mat R_{ii} \mat \zeta_N^{(i)}\| \oindicator{\Lambda_i} > \eta )  + O_l(N^{-l}). 
\end{align*}

%Note that by Schur's Complement \eqref{Schur1} the first term is an entry of the resolvent:
%\begin{align*}
%R_{ii} &= - ( q + H_i^{(i)*} R_N^{(i)} H_i^{(i)} )^{-1}.

Then defining 
\begin{align*}
\widehat{ \mat R}_{ii}:= -(q + \E[\mat \Sigma(\mat \Gamma_N^{(i)})])^{-1}
\end{align*}
and
%Note that this is not actually an entry of a resolvent. 
%In order to control the fluctuations of $R_{ii}$, we use the resolvent identity to compare the $R_{ii}$ with $\widehat R_{ii}$:
%\begin{equation}
%\label{Rtohat}
%R_{ii} - \widehat R_{ii} = \widehat R_{ii} 
%(H_i^{(i)*} R_N^{(i)} H_i^{(i)} - \E[\Sigma(\Gamma_N^{(i)})])
% R_{ii} .
%\end{equation}
%This motivates the definition 
\begin{equation*}
\mat  \gamma_N^{(i)} := \mat H_i^{(i)*} \mat R_N^{(i)} \mat H_i^{(i)} - \E[\mat \Sigma(\mat \Gamma_N^{(i)})],
\end{equation*}
leads to the following expansion:
\begin{align*}
%& \begin{pmatrix}     \eta - X_{i \cdot}^{(i)*} R^{22(i)} X_{i \cdot}^{(i)} & z - X_{i \cdot}^{(i)*} R^{21(i)} X_{\cdot i}^{(i)} \\ z - X_{\cdot i }^{(i)*} R^{12(i)} X_{i\cdot}^{(i)} &   \eta - X_{\cdot i }^{(i)*} R^{11(i)} X_{\cdot i }^{(i)}   \end{pmatrix}^{-1}
\mat R_{ii} \mat \zeta_N^{(i)} 
 =& \left( \mat R_{ii} - \mat {\widehat R}_{ii} \right) \mat  \zeta_N^{(i)} + \widehat{\mat  R}_{ii} \zeta_N^{(i)}\\
% %%
 &= \widehat{ \mat  R}_{ii}  \mat  \gamma_N^{(i)}  \mat  R_{ii} \mat  \zeta_N^{(i)}  + \widehat{ \mat  R}_{ii} \mat  \zeta_N^{(i)}\\
 &= \widehat{ \mat  R}_{ii}  \mat  \gamma_N^{(i)}  \widehat{ \mat  R}_{ii}  (\mat  I_2 + \mat \Sigma( \tr_N( \mat R^{(i)}_N \mat R^{(i)}_N )) \\
 &+ \widehat{ \mat  R}_{ii} \mat   \gamma_N^{(i)} 
\widehat{ \mat  R}_{ii} \left(\mat  H_i^{(i)*} \mat  R_N^{(i)} \mat R_N^{(i)} \mat  H_{i}^{(i)} - \mat \Sigma( \tr_N(\mat  R^{(i)}_N \mat R^{(i)}_N )) \right) \\
&+ \widehat{\mat  R}_{ii}
 \mat   \gamma_N^{(i)}
  \widehat{\mat  R}_{ii}
  \mat  \gamma_N^{(i)} \mat R_{ii}  \mat \zeta_N^{(i)} + \widehat{\mat  R}_{ii} \mat  \zeta_N^{(i)}.
\end{align*}

Finally Burkholder's and Rosenthal's  inequality can be combined with the following lemma to give the desired bound.

\begin{lemma} \label{lemma:bndpeps} 
For $a,b \in \{1,2\}$ and any $p \geq 2$, 
\begin{equation*}
\E[|\mat \Gamma_N^{(i)ab}(\mat q) - \E[\mat \Gamma_N^{(i)ab}(\mat q)]|^{p}] \leq K_p N^{-p/2} t_N^{-p}
\end{equation*}
and
\begin{equation*}
\E [| \gamma_N^{(i)ab} |^p ] \leq K_p t_N^{-p} (N^{-p/2} + N^{-2 p \eps}) .
%\E [| H_{i \cdot}^{ab} R^{(i)bc} H_{\cdot i}^{cd} - (\delta_{bc}(1-\rho) +\rho) \E[\tr_N(R^{(i)bc})] |^p ] \leq \frac{ K_p  t_N^{-p} }{N^{p/2} } 
\end{equation*}
In addition, there exists a constant $K>0$ such that, for all large $N$,
\begin{equation*}
\| \mat{\widehat{R}}_{ii} \| \leq K .
\end{equation*}

\end{lemma}

\begin{proof}
The proof of Lemma \ref{lemma:bndpeps} follows the proof given in \cite{OR} almost exactly.  The only change is that one must use \eqref{gammoment} to bound $\E [|\mat {\widehat \gamma}_N^{(i)ab}|^p$.
\end{proof}

The arguments in \cite{OR} can now be repeated exactly to show that $\E[\mat \Gamma_N(\mat q)] -  \mat \Gamma(\mat q)  = O(N^{-1})$ as it only involves low moment estimates.

Thus, we conclude that, with probability $1-O_l(N^{-l})$, $\mat \Gamma_N(\mat q) - \mat  \Gamma(\mat q)  = o((N t_N)^{-1})$.  After conditioning on this event, the final arguments in \cite{OR} can be repeated verbatim to show that $\nu_{\mat{X}_N - z \mat{I}}([-c,c]) = 0$ with overwhelming probability.

\section{Continuity of $m(z)$ in $\rho$} \label{sec:cont}

Recall that $m(z)$ (defined in \eqref{eq:def:mz}) is a solution of 
\begin{equation} \label{eq:rhomz}
	\rho m^2(z) + z m(z) + 1 = 0 
\end{equation}
for $z \not\in \mathcal{E}_{\rho}$.  Since $m$ is a function of $\rho$, we will explicitly write $m_{\rho}$ to denote this dependence.  Recall from Lemma \ref{lemma:truncation} that $\hat{\rho} \to \rho$ as $N \to \infty$.  Since $\hat{\rho}$ depends on $N$, $m_{\hat{\rho}}$ also depends on $N$.  We will show that $m_{\hat{\rho}}(z)$ converges to $m_{\rho}(z)$ as $N \to \infty$ for any fixed $z \not\in \mathcal{E}_{\rho}$.    

\begin{lemma} \label{lemma:cont}
Let $\delta > 0$.  Then, for any $z \notin \mathcal{E}_{\rho, \delta}$, 
$$ \lim_{N \to \infty} m_{\hat{\rho}}(z) = m_{\rho}(z). $$
\end{lemma}
\begin{proof}
Fix $z \notin \mathcal{E}_{\rho, \delta}$.  Since $\lim_{N \to \infty} \hat{\rho} = \rho$, it follows that $z \notin \mathcal{E}_{\hat{\rho},\delta}$ for $N$ sufficiently large.  From \eqref{eq:rhomz}, we make the following two observations.
\begin{enumerate}[(i)]
\item Since the roots of a (monic) polynomial are continuous functions of the coefficients (see \cite{CC,T}), we conclude from \eqref{eq:rhomz} that $m_\rho(z)$ is a continuous function of $\rho \in [-1,1]\setminus\{0\}$. \label{item:obs1}
\item Similarly, by multiplying \eqref{eq:rhomz} by $\rho$, we see that $\rho m_{\rho}(z)$ is a continuous function of $\rho \in [-1,1]$. \label{item:obs2}
\end{enumerate}

We now divide the proof into two cases.  From observation \eqref{item:obs1}, it follows that $\lim_{N \to \infty} m_{\hat{\rho}}(z) = m_{\rho}(z)$ in the case $\rho \neq 0$.

We now consider the case when $\rho = 0$.  Since $z \notin \mathcal{E}_{0, \delta}$, it follows that $|z| > 1$.  From observation \eqref{item:obs2} and \eqref{eq:def:mz}, we have
$$ \lim_{N \to \infty} \hat{\rho} m_{\hat{\rho}}(z) = 0, $$
and hence there exists $c > 0$ such that
$$ \left| \hat{\rho} m_{\hat{\rho}}(z) + z \right| \geq c $$
for $N$ sufficiently large.  By \eqref{eq:rhomz}, it follows that
$$ |m_{\hat{\rho}}(z)| \leq \frac{1}{c} $$
for $N$ sufficiently large.  Let $m_0(z) = -1/z$ (i.e. $m_0(z)$ is given by \eqref{eq:rhomz} when $\rho=0$).  Then subtracting the equation for $m_{\hat{\rho}}(z)$ from the equation for $m_0(z)$ yields
$$ |z| |m_0(z) - m_{\hat{\rho}}(z)| = |\hat{\rho}| |m_{\hat{\rho}}(z) |^2 \leq \frac{|\hat{\rho}|}{c^2} $$
for $N$ sufficiently large.  Since $|z| > 1$ and $\hat{\rho} \to 0$, we conclude that 
$$ \lim_{N \to \infty} m_{\hat{\rho}}(z) = m_0(z), $$ 
and the proof is complete.  
\end{proof}

\end{document}